\documentclass{elsarticle}

\usepackage{numcompress}

\usepackage{lipsum}
\usepackage{amsfonts}
\usepackage{amsmath,amssymb}
\usepackage{amsthm,bm}
\usepackage{fourier}
\usepackage{xcolor}
\usepackage{mathrsfs}
\usepackage{graphicx}
\usepackage{epstopdf}
\usepackage[ntheorem]{empheq} % this loads amsmath as well
\empheqset{box=\bigfbox}
\usepackage[final]{pdfpages}
\usepackage{subcaption}
\usepackage{tikz}
\usepackage{amsopn}

\DeclareMathOperator*{\argmin}{arg\,min}

\def\tred{\textcolor{black}}
\def\zzz{\textcolor{black}}
\newcommand{\Vn}{\mathbf{n}}
\newcommand{\Vb}{\mathbf{b}}

\newcommand{\Vx}{\mathbf{x}}
\newcommand{\R}{\mathbb{R}}
\newcommand{\Pol}{\mathbb{P}}

\newcommand{\sign}{\mathrm{sign}}

\newcommand{\card}{\mathrm{card}}
\newcommand{\Par}{\mathcal{P}}
\newcommand{\Sk}{\mathcal{S}}
\newcommand{\D}{\mathcal{D}}

\newtheorem{ssmptn}{Assumption}

\newtheorem{prpstn}{Proposition}

\newtheorem{thrm}{Theorem}
\newtheorem{rmrk}{Remark}

\newcommand{\llbrace}{\lbrace \hspace{-.045cm}\lbrace }
\newcommand{\rrbrace}{\rbrace \hspace{-.05cm}\rbrace }

\begin{document}

\begin{frontmatter}
%\title{Stabilized Finite Element Method with Discontinuous Galerkin based residual minimization}
\title{\tred{An adaptive stabilized conforming finite element method via residual minimization on dual discontinuous Galerkin norms}}

\author[one,three]{Victor M. Calo}

\author[four,five]{Alexandre Ern}

\author[six]{Ignacio Muga}
%\author[mymainaddress]{Elsevier Inc}

%\author[seven,eight,nine]{David Pardo}

\author[one]{Sergio Rojas\corref{mycorrespondingauthor}}
\cortext[mycorrespondingauthor]{Corresponding author}
\ead{srojash@gmail.com}

\address[one]{School of Earth and Planetary Sciences, Curtin University, Kent Street, Bentley, Perth, WA 6102, Australia}
%\address[second]{Curtin Institute for Computation, Curtin University, Kent Street, Bentley, Perth, WA 6102, Australia}

\address[three]{Mineral Resources, Commonwealth Scientific and Industrial Research Organisation (CSIRO), Kensington, Perth, WA 6152, Australia}

\address[four]{Universit\'e Paris-Est, CERMICS (ENPC), 6-8 avenue Blaise Pascal, 77455 Marne la Vall\'ee cedex 2, France}

\address[five]{INRIA Paris, 75589 Paris, France}

\address[six]{Instituto de Matem\'aticas, Pontificia Universidad Cat\'olica de Valpara\'iso, Casilla 4059, Valpara\'iso, Chile}

%\address[seven]{University of the Basque Country (UPV/EHU), Leioa, Spain}

%\address[eight]{BCAM - Basque Center for Applied Mathematics, Bilbao, Spain}

%\address[nine]{IKERBASQUE, Basque Foundation for Science, Bilbao, Spain}

\begin{abstract}
We design and analyze a new adaptive stabilized finite element method. We construct a discrete approximation of the solution in a continuous trial space by minimizing the residual measured in a dual norm of a discontinuous test space that has inf-sup stability. We formulate this residual minimization as a stable saddle-point problem which delivers a stabilized discrete solution and a residual representation that drives the adaptive mesh refinement. Numerical results on an advection-reaction model problem show competitive error reduction rates when compared to discontinuous Galerkin methods on uniformly refined meshes and smooth solutions. Moreover, the technique leads to optimal decay rates for adaptive mesh refinement and solutions having sharp layers.
\end{abstract}

\begin{keyword}
stabilized finite elements \sep residual minimization\sep inf-sup stability\sep advection-reaction\sep adaptive mesh refinement\sep Discontinuous Petrov-Galerkin
\MSC[2010] 65N12\sep 65N30\sep 76M10
\end{keyword}

\end{frontmatter}

%\linenumbers

\section{Introduction}

Continuous Galerkin Finite Element Methods (CG-FEM) are popular solution strategies in engineering applications. However, these methods  suffer from instability under reasonable physical assumptions if the mesh is not \zzz{fine enough}. This limitation led to the development of several alternative methods that achieve stability differently (see, e.g.,~\cite{ ern2013theory, HUGHES2017} and references therein). \zzz{Broadly speaking, one can identify} two approaches to enhance stability in CG-FEM: residual and fluctuation-based stabilization techniques\zzz{~\cite{ Ern_Guermond_16}}. The Least-Squares Finite Element Method (LS-FEM) \zzz{exemplifies the residual-based stabilization\tred{,} which was} pioneered in the late \zzz{'60s} and early \zzz{'70s} (cf.,~\cite{ Diska:68, Lucka:69, BrSc70} and~\cite{ Jiang99} for a recent overview). The \zzz{Galerkin/Least-Squares (GaLS) method~\cite{ HFH89} addresses} the suboptimal convergence rate of LS-FEM in the $L^2$-norm. \zzz{The authors in~\cite{ CohDahWelM2AN2012} extended} the Least-Squares approach to more general \emph{dual norms} \zzz{and connected} residual minimization in test dual norms and the mixed formulation (cf., Equation~\eqref{eq:mix_form}). In this context,\zzz{~\cite{ ChaEvaQiuCAMWA2014} investigated} the use of different test norms to stabilize convection-diffusion problems. \zzz{Also}, the idea of residual minimization in non-standard dual norms is at the heart of the recent Discontinuous Petrov--Galerkin (DPG) methods (see, e.g.,~\cite{ zitelli2011class, DemHeuSINUM2013, ChaHeuBuiDemCAMWA2014} and~\cite{ DemGopBOOK-CH2014} for \zzz{an} overview).
Alternatively, several strategies based on fluctuation stabilization exist. Some methods penalize the gradient of some fluctuation or some fluctuation of the gradient~\cite{ GuM2AN, BecBra:01, MaSkT:07}. Other methods penalize the gradient jumps across the mesh interfaces~\cite{ Burman:05, BurHa:04}. Yet, other techniques enlarge the trial and test spaces with discontinuous functions and penalize the solution jumps across the mesh interfaces, as in the discontinuous Galerkin (DG) methods~\cite{ osti_4491151, LesRa:74, JohPi:86, CoKaS:00, BrMaS:04, Ern2006} (see \zzz{also}~\cite{ di2011mathematical}).

In this work, we introduce a new adaptive stabilized FEM and study its numerical performance when approximating an advection-reaction model problem. We combine residual minimization with the inf-sup stability offered by a large class of DG methods. More precisely, the discrete solution lives in a continuous trial space (e.g., $H^1$-conforming finite elements) and minimizes the residual in a dual norm with respect to the DG test functions. This DG norm provides inf-sup stability to the formulation. In practice, such a residual minimization implies a stable saddle-point problem involving the continuous trial space and the discontinuous test space. \zzz{Two advantages compensate the extra cost of solving a significantly} larger linear system (compared to just forming the normal equations as in LS-FEM). Firstly, \zzz{the continuous} solution is more accurate, especially in the $L^2$-norm, than the LS-FEM solution. \zzz{We} prove that the present method leads to error estimates of the same quality as those delivered by \zzz{the} DG methods, yet for a discrete solution belonging to a continuous trial space only. The second advantage is that in \zzz{our} approach, the component in the DG space plays the role of a residual representative that \zzz{drives the adaptive mesh refinements}. Thus, we \zzz{simultaneously compute} a discrete solution in the continuous trial space and an error representation in the discontinuous test space. From the practical point of view, we illustrate the benefits of the adaptive mesh refinement procedure \zzz{numerically} in \tred{an} advection-reaction model problem with solutions possessing sharp inner layers. 

We now put our work in perspective \zzz{with the state-of-the-art}. Compared with the LS-FEM paradigm, we solve a larger linear system\zzz{. However,} the advantage is that we obtain a residual representation that guides the mesh adaptation \zzz{to deliver better} error decay rates \zzz{relative} to the number of degrees of freedom. Compared to \zzz{the} Discontinuous Petrov--Galerkin (DPG) methods~\cite{ DemGopCMAME2010, DemGopNMPDE2011}, we share the goal of minimizing the residual \zzz{in} an adequate norm, using broken test spaces, implying discrete stability\zzz{,} and a residual representation to guide adaptivity. Nevertheless, the main difference is that \zzz{the DPG methods focus on} conforming formulations and localizable test norms. In particular, for first-order PDEs\zzz{,} and if the trial space is continuous, the norm equipping the test space must be the $L^2$-norm. Stronger norms can be applied using ultra-weak formulations\zzz{,} which only require $L^2$ regularity for the trial space. This \zzz{lower regularity, in practice,} leads to a discontinuous approximation of the trial solution, \zzz{which requires} additional unknowns to represent traces of the solution over the mesh skeleton\zzz{. However,} the broken test space structure with localizable \zzz{norms} allows for static condensation\tred{, which, when combined with multigrid techniques, leads to linear cost solvers. Instead, in our work}, the trial space is continuous, but the test space \zzz{has} a stronger norm than $L^2$\zzz{. We inherit} the stability of the method from the DG formulation that provides the inner product and \zzz{the norm on} which we build our method.
\tred{Weak primal solutions in $L^2$ and residuals measured in the dual norm of the {\it optimal test norm are considered in~\cite{ DahHuaSchWelSINUM2012}.} That is, the $L^2$ norm of the adjoint operator applied to the test function. The residual in the mixed formulation drives adaptivity, but the discrete stability was only assumed. The proof of discrete stability came later, in the context of DPG formulations for transport and a particular family of nested meshes~\cite{ BroDahSteMOC2018}. In a different context, building on non-Hilbert spaces for first-order PDEs, residuals measured in the Lebesgue space $L^p$, with $1\leq p<+\infty$, were considered in~\cite{ GueSINUM2004} (which requires to solve a non-linear problem for $p\neq 2$). An extension of these results to weak advection-reaction problems, where the solutions live in $L^p$ while the residual measure is the dual graph-norm, was given in~\cite{ MugTylZeeCMAM2019}.}

%{\color{blue} The particular application of residual minimization techniques to advection-reaction problems (or in general, to first-order PDEs) has been studied in many conforming contexts. Indeed, the pioneer DPG works of Demkowicz \& Gopalakrishnan~\cite{ DemGopCMAME2010,DemGopNMPDE2011} contained applications to pure convection problems in weak hybrid form, where residuals are minimized in the dual norm of a broken graph test space. From an non-hybrid point of view, 
% Dahmen~et~al.~\cite{ DahHuaSchWelSINUM2012} considered weak solutions in $L^2$, with residuals measured in the dual norm of the {\it optimal test norm} (i.e., the $L^2$ norm of the adjoint operator applied to the test function). The residual variable coming from the mixed formulation was used to drive adaptivity, but discrete stability was only assumed. The first effort to prove discrete stability came later, in the context of DPG formulations for transport, when Broersen, Dahmen \& Stevenson~\cite{ BroDahSteMOC2018} remarkably proved the discrete stability property of a special family of nested meshes. On a completely different hand, from a non Hilbert-space perspective of first-order PDEs, Guermond~\cite{ GueSINUM2004} considered residuals measured in the Lebesgue space $L^p$, with $1\leq p<+\infty$ (which in turns implies to solve a non-linear problem for $p\neq 2$). The extension of these results to weak advection-reaction problems, i..e, the solutions are in $L^p$ while the residuals are measured in dual graph-norms, was given later in~\cite{ MugTylZeeCMAM2019}.  
%}

We organize the paper as follows. In Section~\ref{sec:prelim}, we recall some basic facts concerning the continuous and discrete settings. In particular, we present the abstract framework from~\cite{ Ern2006} for the error analysis of nonconforming approximation methods (see also~\cite{ di2011mathematical} and Strang's Second Lemma). We devise and analyze the present stabilized FEM in Section~\ref{sec:res_min}, where our main results are Theorems~\ref{th:FEMwDG} and~\ref{th:err_rep}. In Section~\ref{sec:AR}, we briefly describe the mathematical setting for the advection-reaction model problem and recall its DG discretization using both centered and upwind fluxes. In Section~\ref{sec:num}, we present numerical results illustrating the performance of our adaptive stabilized FEM. Finally, we \zzz{conclude with} Section~\ref{sec:conc}.

\section{Continuous and discrete settings}\label{sec:prelim}

\subsection{Well-posed model problem}

Let $X$ (continuous trial space) and $Y$ (continuous test space) be (real) Banach spaces equipped with norms $\|\cdot\|_X$ and $\|\cdot\|_Y$, respectively. Assume that $Y$ is reflexive. We want to solve the following linear model problem:
\begin{equation}\label{eq:vf_cont}
\left\{\begin{array}{l}
\text{Find } u \in X,  \text{ such that:} \smallskip \\
b(u,v) = \left<l,v\right>_{Y^\ast \, , \, Y}, \quad \forall \, v \in Y,
\end{array}
\right.
\end{equation}
where $b$ is a bounded bilinear form on $X \times Y$, $l \in Y^\ast$ is a bounded linear form on $Y$, and $\left<\cdot,\cdot\right>_{Y^\ast \, , \, Y}$ is the duality pairing in $Y^\ast\times Y$.  Equivalently, \zzz{we write} problem~\eqref{eq:vf_cont} in operator form by introducing the operator $B : X \rightarrow Y^\ast$ such that:
\begin{equation}
\left<B \, z \, , \, v\right>_{Y^\ast \, , \, Y} := b( z , v), \quad \forall \, (z,v) \in X \times Y, 
\end{equation}
leading to the following problem:
\begin{equation}\label{eq:B_cont}
\left\{\begin{array}{l}
\text{Find } u \in X,  \text{ such that:} \smallskip \\
B u = l \, \text{ in } Y^\ast.
\end{array}
\right.
\end{equation}
We assume that there exists a constant $C_b > 0$ such that 
\begin{equation}
\inf_{0\neq z \in X} \sup_{0 \neq y \in Y} \frac{|b(z,y)|}{\|z\|_X \|y\|_Y} \geq C_b\,,
\end{equation}
and that $\left\{ y \in Y : b(z,y) = 0, \forall \, z \in X \right\} = \{0\}$.  These two assumptions are equivalent to problem~\eqref{eq:vf_cont} (or equivalently~\eqref{eq:B_cont}) being well-posed owing to the Banach--Ne\u{c}as--Babu\u{s}ka theorem (see, e.g.,~\cite[Theorem 2.6]{ ern2013theory}). Moreover, the following a~priori estimate is then satisfied:
\begin{equation}\label{eq:cont_bound}
\|u\|_X \leq C^{-1}_b \|l\|_{Y^\ast}.
\end{equation}
%Equivalently, one can assume that 
%\begin{equation}
%\inf_{0\neq y \in Y} \sup_{0 \neq z \in X} \frac{|b(z,y)|}{\|z\|_X \|y\|_Y} \geq C_{\rm sta},
%\end{equation}
%and that $\left\{ z \in X : b(z,y) = 0, \forall \, y \in Y \right\} = \{0\}$.
Finally, we mention that when $Y=X$, a sufficient condition for well-posedness is that the bilinear form $b$ is coercive, that is, there exists a constant $C_b > 0$ such that
\begin{equation}\label{eq:coerc}
b(v,v) \geq C_b \|v\|_X^2, \quad \forall \, v \in X.
\end{equation}
Then problem~\eqref{eq:vf_cont} (or equivalently~\eqref{eq:B_cont}) is well-posed owing to the the Lax--Milgram lemma (see, e.g.,~\cite{ lax2005parabolic}), and \zzz{it satisfies} the same a~priori estimate~\eqref{eq:cont_bound}.

\subsection{Functional setting}

For any open and bounded set $\D \subset \R^d$, let $L^2(\D)$ be the standard Hilbert space of square-integrable functions over $\D$ for the Lebesgue measure, and denote by $\displaystyle (\cdot,\cdot)_\D$ and $\|\cdot\|_{\D} := \sqrt{(\cdot,\cdot)_\D}$ its inner product and inherited norm, respectively. Denote by $L^2(\D;\R^d)$ the corresponding space composed of square-integrable vector-valued functions and, \zzz{abusing notation}, still by $(\cdot, \cdot)_\D$ the associated inner product. Considering weak derivatives, we recall the following well-known Hilbert space:
\begin{align}\label{eq:h1}
H^1(\D) &:= \left\{ v \in L^2(\D) \, : \, \nabla v \in L^2(\D;\R^d)\right\}, %\smallskip\\
%H^2(\D) &:= \left\{ v \in H^1(\D) \, : \, \div \left(\nabla v\right) \in L^2(\D) \right\},  \smallskip\\
%H(\div,\D) &:= \left\{ \Vtau \in \left[L^2(\D) \right]^d \, : \, \div(\Vtau) \in L^2(\D) \right\},
\end{align}
equipped with the following inner product and norm (respectively):
\begin{equation}
(v,u)_{1,\D} := (v,u)_{\D} + (\nabla v,\nabla u)_{\D}, \qquad \|v\|_{1,\D} := \sqrt{(v,v)_{1,\D}},% \smallskip \\
%(v,u)_{2,\D} &= (v,u)_{1,\D} + (\div(\nabla v),\div(\nabla u))_{\D}, & \|v\|_{2,%\D} &= \sqrt{(v,v)_{2,\D}}, \smallskip \\
%(\Vtau,\Vsigma)_{\div,\D} &= (\Vtau,\Vsigma)_{\D} + (\div \, \Vtau, \div \, \Vsigma)_{\D}, & \|\Vtau\|_{\div,\D} &= \sqrt{(\Vtau,\Vtau)_{\div,\D}}.
\end{equation}
Let $H^{1/2}(\partial \D)$ be the standard Dirichlet trace space of $H^1(\D)$ over the boundary $\partial \D$, and $H^{-1/2}(\partial \D)$ be its dual space with $L^2(\partial\D)$ as pivot space. More generally, to warrant that traces are \zzz{well-defined} at least in $L^2(\partial \D)$, \zzz{we call} upon the fractional-order Sobolev spaces $H^s(\D)$, with $s>\frac12$.
%When assuming that the boundary of $\D$ can be split into $\partial \D = \Gamma_N \cup \Gamma_D$, with $\Gamma_D \neq \emptyset$ being a Dirichlet-type boundary, will be required make use of additional Hilbert spaces:
%\begin{align}
%H^1_D(\D) := \left\{ v \in H^1(\D) \, : \, u=0, \text{ on } \Gamma_D \right\}, \qquad
%H^2_D(\D) := \left\{ v \in H^2(\D) \, : \, u=0, \text{ on } \Gamma_D \right\}.
%\end{align}

\subsection{Discrete setting}\label{sec:discrete}
\begin{figure}[t!]
\centering
\includegraphics[scale=0.9]{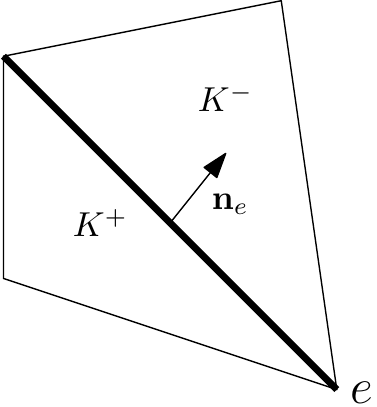}
\caption{Skeleton orientation over the internal face $e = \partial K^+ \cap \partial K^-$.}
\label{fig:skeleton}
\end{figure}
Let $\Par_h =\{ K_m \}^N_{m=1} $ be a conforming partition of the domain $\D$ into $N$ open disjoint elements $K_m$, such that 
\begin{equation}
\D_h := \bigcup^N_{m = 1} K_m \quad \text{ satisfies } \quad \D = \text{int}\left(\overline{\D_h}\right).
\end{equation}
We denote by $\partial K_m$ the boundary of the element $K_m$, by $\Sk^0_h$, $\Sk^\partial_h$ the set of interior and boundary edges/faces (respectively) of the mesh, and by $\Sk_h := \Sk^0_h \cup \Sk^\partial_h$ the skeleton of $\D_h$. Over $\D_h$, we define the following standard broken Hilbert space:
\begin{align}
H^1(\D_h) &:= \left\{ v \in L^2(\D) \, : \, \nabla v|_{K_m} \in L^2(K_m;\R^d), \, \forall \, m = 1,\dots, N \right\}, 
\end{align}
\zzz{with the following inner product definition:}
\begin{equation}
(v,u)_{1,\D_h} := \sum^N_{m=1} (v,u)_{1,K_m}. %\qquad (\Vtau,\Vsigma)_{\D_h} = \sum^N_{m=1} (\Vtau,\Vsigma)_{K_m}
\end{equation}
For any $v \in H^1(\D_h)$ and any interior face/edge $e = \partial K^+ \cap \partial K^- \in \Sk^0_h$  (see Figure~\ref{fig:skeleton}), we define the jump $\llbracket v \rrbracket_e$ and the average $\llbrace v \rrbrace_e$ of a smooth enough function $v$ as follows:
\begin{align}
\llbracket v \rrbracket_e(x) &:= v^+(x) - v^-(x),\label{eq:jump1}\\
\llbrace  v \rrbrace_e(x) &:= \dfrac{1}{2} (v^+(x) + v^-(x)), \label{eq:jump2}
\end{align}
for a.e.~$x\in e$, where $v^+$ and $v^-$ denote the traces over $e$ defined from a \zzz{fixed, predefined orientation of the} unit normal vector $\Vn_e$. \zzz{We omit the} subscript $e$ from the jump and average operators when there is no ambiguity.
 
Finally, we denote by $\Pol^{p}(K_m)$, $p\ge0$, the set of polynomials of total degree at most $p$ defined on the element $K_m$, and we consider the following broken polynomial space:
\begin{align} \label{eq:def_bk_pol}
\Pol^p(\D_h) &:= \left\{ v \in L^2(\D) \, : \, v|_{K_m} \in \Pol^p(K_m), \, \forall \, m = 1,\dots, N  \right\}.
\end{align}

\subsection{Abstract setting for nonconforming approximation methods} \label{sec:stand_DG}

Let $V_h$ be a finite-dimensional space composed of functions defined on $ \D_h$ (typically a broken polynomial space).  We approximate the continuous problem~\eqref{eq:vf_cont} as follows:
\begin{equation}\label{eq:vf_discrete}
\left\{\begin{array}{l}
\text{Find } \theta_h \in V_h,  \text{ such that:} \\
b_h(\theta_h, v_h) = \langle l_h,v_h\rangle_{V_h^\ast\times V_h}, \quad \forall \, v_h \in V_h,
\end{array}
\right.
\end{equation}
where $b_h(\cdot,\cdot)$ denotes a discrete bilinear form defined over $V_h \times V_h$, and $l_h(\cdot)$ a discrete linear form over $V_h$.  We say that the approximation setting is nonconforming whenever $V_h\not\subset X$ or $V_h\not\subset Y$.

To ascertain \zzz{the well-posedness of the} discrete problem~\eqref{eq:vf_discrete} and to perform the error analysis, we follow the framework introduced for DG approximations in~\cite{ Ern2006} (see also~\cite{ di2011mathematical} and Strang's Second Lemma) which relies on the following three assumptions:
\begin{ssmptn}\label{as:inf-sup}
  (Stability): The space $V_h$ \zzz{is} equipped with a norm $\|\cdot\|_{V_h}$ such that there exists a constant $C_{\emph{sta}}>0$, uniformly with respect to the mesh size, such that:
  \begin{equation}\label{eq:infsup_h}
    \inf_{0\neq z_h \in V_h} \sup_{0\neq v_h \in V_h} \dfrac{b_h(z_h,v_h)}{\| z_h \|_{V_h}\| v_h \|_{V_h}}  \geq C_{\emph{sta}}.
  \end{equation}
\end{ssmptn}
\begin{ssmptn} \label{as:reg-consi}
  (Strong consistency with regularity) The exact solution $u$ of~\eqref{eq:vf_cont} belongs to a subspace $X_{\#} \subset X$ such that the discrete bilinear form $b_h(\cdot,\cdot)$ supports evaluations in the extended space $V_{h,\#} \times V_h$ with $V_{h,\#} := X_{\#} + V_h$, and the following holds true:
  \begin{equation}\label{eq:consistency}
    b_h(u,v_h) = \langle l_h,v_h\rangle_{V_h^\ast\times V_h}, \quad \forall \, v_h \in V_h,
  \end{equation}
  which amounts to $b_h(u-\theta_h,v_h) = 0$, for all $v_h \in V_h$ (Galerkin's orthogonality).
\end{ssmptn}
\begin{ssmptn} \label{as:bound}
  (Boundedness): The stability norm $\|\cdot\|_{V_h}$ can be extended to $V_{h,\#}$ and there is a second norm $\|\cdot\|_{V_h,\#}$ on $V_{h,\#}$ satisfying the following two properties:
  \begin{enumerate}[(i)]
  \item $\|v\|_{V_h} \leq \|v\|_{V_h,\#}$, for all $v \in V_{h,\#}$;
  \item there exists a constant $C_{\emph{bnd}}<\infty$, \zzz{uniform} with respect to the mesh size, such that
    \begin{equation}\label{eq:continuity}
      b_h(z,v_h) \leq C_{\emph{bnd}} \, \|z\|_{V_h,\#} \| v_h \|_{V_h}, \quad \forall \, (z, v_h) \in V_{h,\#} \times V_{h}.
    \end{equation}
  \end{enumerate}
\end{ssmptn}
For a linear form $\phi_h:V_h\to\R$, we set
\begin{equation}\label{eq:dualnorm}
\|\phi_h\|_{V_h^\ast}:=\sup_{0\ne v_h\in V_h} \frac{\langle \phi_h,v_h\rangle_{V_h^\ast\times V_h}}{\|v_h\|_{V_h}},
\end{equation}
where $\|\cdot\|_{V_h}$ is the stability norm identified in Assumption~\ref{as:inf-sup}. The above assumptions lead to the following a~priori and error estimates (see~\cite{ Ern2006, di2011mathematical}).

\begin{thrm}[\tred{Discrete well-posedness and a~priori estimates I}]\label{thrm:well-posedness}
Denote by $u$ the solution of the continuous problem~\eqref{eq:vf_cont} and suppose that the assumptions~\ref{as:inf-sup}--\ref{as:bound} are satisfied. Then there exists a unique $\theta_h \in V_h$ solution to the discrete problem~\eqref{eq:vf_discrete}, and the following two estimates are satisfied:
\begin{equation}
\|\theta_h\|_{V_h} \leq \dfrac{1}{C_{\emph{sta}}}\|l_h\|_{V_h^\ast},
\end{equation}
and
\begin{equation}
\|u-\theta_h\|_{V_h} \leq \left(1 + \dfrac{C_{\emph{bnd}}}{C_{\emph{sta}}} \right) \inf_{v_h \in V_h}\|u-v_h\|_{V_h,\#}.
\end{equation} 
\end{thrm}
%\begin{proof}
%See~\ref{ap:main_proof}.
%\end{proof}
%
\section{Residual minimization problem}\label{sec:res_min}

\zzz{Our schemes build on the following two ideas:}
\begin{itemize}
\item[a)] First, we select a finite-dimensional space $V_h$ and a discrete bilinear form $b_h(\cdot,\cdot)$ such that Assumptions~\ref{as:inf-sup}, \ref{as:reg-consi}, and~\ref{as:bound} hold true.
%, that is, there is a norm $\|\cdot\|_{V_h}$ such that the inf-sup condition~\eqref{eq:infsup_h} is satisfied, uniformly with respect to the mesh size.
\item[b)] Second, we identify a subspace $U_h \subset V_h$ such that 
$U_h$ has the same approximation capacity as the original space $V_h$ for the types of solutions we want to approximate. The \zzz{primary} example we have in mind is to choose $V_h$ as a broken polynomial space and $U_h$ as the $H^1$-conforming subspace. We refer the reader to Remark~\ref{rem:best_app} for a further discussion on the approximation capacity of the spaces $U_h$ and $V_h$ in the context of advection-reaction equations.
\end{itemize}

Starting from a stable formulation of the form~\eqref{eq:vf_discrete} in $V_h$, we use the trial subspace $U_h \subset V_h$ to solve the following residual minimization problem:
\begin{equation}\label{eq:min_prob}
\left\{\begin{array}{l}
\text{Find } u_h \in U_h \subset V_h,  \text{ such that:} \smallskip \\
\displaystyle u_h = \argmin_{z_h \in U_h} \dfrac{1}{2}\|l_h- B_h \, z_h\|^2_{V_h^\ast} = \argmin_{z_h \in U_h} \dfrac{1}{2}\|R^{-1}_{V_h}(l_h- B_h z_h)\|^2_{V_h},
\end{array}
\right.
\end{equation}
where $B_h : V_{h,\#} \rightarrow V_h^{\ast}$ is defined as: 
\begin{align}\label{eq:B_h}
 \langle B_h z, v_h \rangle_{V_h^{\ast} \times V_h} := b_h(z,v_h),
%B_h^\ast : V_h \rightarrow U_h^{\prime}, \qquad &\text{defined as }  %\langle B_h^{\ast} v_h, u_h \rangle_{U_h^{\ast} \times U_h} := %b_h(u_h,v_h).
\end{align}
and $R^{-1}_{V_h}$ denotes the inverse of the Riesz map:
\begin{equation}\label{eq:riesz}
  \begin{array}{rcl}
    R_{V_h} &:& V_h \rightarrow V_h^\ast \smallskip\\
            && \left< R_{V_h} y_h, v_h\right>_{V_h^\ast \times V_h} := (y_h,v_h)_{V_h}, \quad \forall \ v_h\in V_h.
  \end{array}
\end{equation}
The second equality in~\eqref{eq:min_prob} follows from the fact that the Riesz map is an isometric isomorphism.  Classically, the minimizer in~\eqref{eq:min_prob} is a critical point of the minimizing functional, which translates into the following linear problem:
\begin{align}\label{eq:min_vf}
\left\{\begin{array}{l}
\text{Find } u_h \in U_h, \text{ such that:} \\
(R^{-1}_{V_h} (l_h-B_h \, u_h),R^{-1}_{V_h} B_h \delta u_h)_{V_h} = 0, \quad \forall \, \delta u_h \in U_h.
\end{array}\right.
\end{align}
Defining the residual representation function as 
\begin{equation}\label{eq:e_h}
\varepsilon_h := R^{-1}_{V_h} (l_h-B_h u_h) \in V_h,
\end{equation} 
problem~\eqref{eq:min_vf} is equivalent to finding the pair $(\varepsilon_h, u_h) \in V_h \times U_h$, such that: 
\begin{subequations}
\label{eq:mix_form}
\begin{empheq}[left=\left\{,right=\right.,box=]{alignat=3}
%
%\left\{\begin{array}{l}
 %\begin{array}{lcll}
\label{eq:mix_forma}
&\,\, (\varepsilon_h \, , \, v_h)_{V_h} + b_h(u_h \, , \, v_h)  && =  l_h(v_h),   &\quad&\forall\, v_h \in V_h, \\
\label{eq:mix_formb}
&\,\,b_h(z_h\, , \, \varepsilon_h) &&=  0,  &\quad&\forall\, z_h \in U_h.
\end{empheq}
\end{subequations}
Conversely, if the pair $(\varepsilon_h, u_h) \in V_h \times U_h$ solves~\eqref{eq:mix_form}, then~\eqref{eq:e_h} holds true and $u_h$ is the minimizer of the quadratic functional in~\eqref{eq:min_prob}.

\begin{thrm}[\tred{Discrete well-posedness and a~priori estimates II}]\label{th:FEMwDG}
  If the assumptions~\ref{as:inf-sup}-\ref{as:bound} are satisfied, then the saddle-point problem~\eqref{eq:mix_form} has \zzz{a unique} solution $(\varepsilon_h,u_h)\in V_h\times U_h$. Moreover, such a solution satisfies the following a~priori bounds:
 %                                                  %
  \begin{equation}\label{eq:bounds}
    \|\varepsilon_h\|\leq \|l_h\|_{V_h^\ast} \qquad \hbox{ and }\qquad \|u_h\|_{V_h} \leq \dfrac{1}{C_{\emph{sta}}}\|l_h\|_{V_h^\ast}\,\,,
  \end{equation}
  and the following a~priori error estimate holds:
  \begin{equation}\label{eq:apriori}
    \|u-u_h\|_{V_h} \leq \left(1 + \dfrac{C_{\emph{bnd}}}{C_{\emph{sta}}}\right) \inf_{z_h \in U_h}\|u-z_h\|_{V_{h,\#}}\,\,,
  \end{equation}
  \zzz{where} $u\in X_{\#}$ is the exact solution to the continuous problem~\eqref{eq:vf_cont}.
\end{thrm}
\begin{proof}                                                    %
  See~\ref{ap:main_proof}.
\end{proof}
\tred{The identity \eqref{eq:e_h} suggests considering $\varepsilon_h$ as a residual representative. The following simple result proves that $\|\varepsilon_h\|_{V_h}$ is an efficient error estimator.
\begin{prpstn}[\tred{Efficiency of the residual representative}] Under the same hypotheses as in Theorem~\ref{th:FEMwDG}, the following holds:
\begin{equation}\label{eq:efficiency}
\|\varepsilon_h\|_{V_h} \leq C_{\emph{bnd}} \, \|u-u_h\|_{V_h,\#}.
\end{equation}
\end{prpstn}
\begin{proof}
Using the isometry property of the Riesz operator, the consistency Assumption~\ref{as:reg-consi} and the boundedness Assumption~\ref{as:bound}, we obtain
\begin{equation*}
\|\varepsilon_h\|_{V_h} = \|B_h(\theta_h- u_h)\|_{V_h^\ast} = \|B_h(u- u_h)\|_{V_h^\ast} \leq C_{\emph{bnd}}\, \|u-u_h\|_{V_h,\#}.
\end{equation*}
\end{proof}
To prove that the residual representative $\varepsilon_h$ is also reliable, we introduce the following additional assumption:}
\begin{ssmptn}(Saturation)\label{as:saturation}
  Let $(\varepsilon_h,u_h)\in V_h\times U_h$ solve the saddle-point problem~\eqref{eq:mix_form}. Let $\theta_h\in V_h$ be the unique solution to~\eqref{eq:vf_discrete}. There exists a real number $\delta\in [0,1)$, uniform with respect to the mesh size, such that $\|u-\theta_h\|_{V_h} \le \delta \|u-u_h\|_{V_h}$.
\end{ssmptn}
\begin{prpstn}[\zzz{Reliability of the residual representative}] \label{th:err_rep}
\zzz{Let 
%$u_h\in U_h$ be the second component of the pair 
$(\varepsilon_h,u_h)\in V_h\times U_h$ solve the saddle-point problem~\eqref{eq:mix_form}}. Let $\theta_h\in V_h$ be the unique solution to~\eqref{eq:vf_discrete}. Then the following holds true:
  \zzz{
  \begin{equation} \label{eq:bnd_uh-thetah}
    \|\theta_h-u_h\|_{V_h} \le \dfrac{1}{C_{\emph{sta}}} \|\varepsilon_h\|_{V_h}.
  \end{equation}
  }
  Moreover, if the saturation Assumption~\ref{as:saturation} is satisfied, then the following a~posteriori error estimate holds:
  \begin{equation} \label{eq:a_posteriori}
\|u-u_h\|_{V_h} \le \dfrac{1}{(1-\delta)C_{\emph{sta}}} \|\varepsilon_h\|_{V_h}.
\end{equation} 
\end{prpstn}

\begin{proof}
  We observe that
  \zzz{
  \begin{eqnarray*}
    \| B_h(\theta_h-u_h) \|_{V_h^\ast} = \|R^{-1}_{V_h} (l_h-B_h u_h)\|_{V_h} = \| \varepsilon_h\|_{V_h},
  \end{eqnarray*}
  }
  which leads to the bound~\eqref{eq:bnd_uh-thetah} owing to the inf-sup condition~\eqref{eq:infsup_h}. Then, invoking the triangle inequality and the saturation assumption leads to
  \[
    \|u-u_h\|_{V_h}\le \|u-\theta_h\|_{V_h} + \zzz{\|\theta_h-u_h\|_{V_h}} \le \delta \|u-u_h\|_{V_h} + \zzz{\|\theta_h-u_h\|_{V_h}}.
\]
Re-arranging the terms and using~\eqref{eq:bnd_uh-thetah} proves the a~posteriori estimate~\eqref{eq:a_posteriori}.
\end{proof}
\tred{%
  Assumption~\ref{as:saturation} states that, when $\theta_h \neq u_h$, $\theta_h$ is a better approximation than $u_h$ to the analytical solution $u$ with respect to the norm $\|\cdot\|_{V_h}$. This is a reasonable assumption since $u_h$ is sought in the subspace $U_h \subset V_h$.  However, in some cases, this assumption may not be satisfied if the mesh is not fine enough, as we show in Section~\ref{sec:num}. This observation motivates us to consider the following weaker version of Assumption~\ref{as:saturation}.
\begin{ssmptn}[Weaker condition]\label{as:weak}
Let $(\varepsilon_h,u_h)\in V_h\times U_h$ solve the saddle-point problem~\eqref{eq:mix_form}. Let $\theta_h\in V_h$ be the unique solution to~\eqref{eq:vf_discrete}. There exists a real number $\delta_w > 0$, uniform with respect to the mesh size, such that $\|u-\theta_h\|_{V_h} \le \delta_w \|\theta_h-u_h\|_{V_h}$.
\end{ssmptn}
A straightforward verification shows that Assumption~\ref{as:saturation} implies Assumption~\ref{as:weak} with constant $\delta_w = \dfrac{\delta}{1-\delta}$, but the converse is true provided $\delta_w <\frac{1}{2}$. Additionally, when Assumption~\ref{as:weak} is satisfied, the following a~posteriori error estimate holds:
\begin{equation} \label{eq:a_posteriori_weak}
\|u-u_h\|_{V_h} \le \dfrac{1+\delta_w}{C_{\emph{sta}}} \|\varepsilon_h\|_{V_h}.
\end{equation} 
 }
\begin{rmrk}[$U_h=V_h$]
In the particular case where $U_h=V_h$, one readily verifies that the unique solution to the saddle-point problem~\eqref{eq:mix_form} is the pair $(0,\theta_h)$, where $\theta_h\in V_h$ is the unique solution to~\eqref{eq:vf_discrete}. In this situation, the bound~\eqref{eq:bnd_uh-thetah} is not informative.
\end{rmrk}

\begin{rmrk}[LS-FEM]
Assume that $V_h:=\Pol^p(\D_h)$, $p\ge1$, is the broken polynomial space defined in~\eqref{eq:def_bk_pol} and that $U_h$ is the $H^1$-conforming subspace $U_h:=V_h\cap H^1(\D)$. Assume that $V_h$ is equipped with the $L^2$-norm and that the inf-sup condition~\eqref{eq:infsup_h} holds true (in the examples we have in mind, for example, a first-order PDE such as the advection-reaction equation described in Section~\ref{sec:AR}, the discrete bilinear form is $L^2$-coercive). Then, the residual minimization problem~\eqref{eq:min_prob} coincides with the Least-Squares Finite Element Method (LS-FEM) set in $L^2(\D)$, and the residual representative $\varepsilon_h\in V_h$ is the $L^2$-projection onto $V_h$ of the finite element residual. Unfortunately, the $L^2$-norm is too weak, leading to an error estimate~\eqref{eq:apriori} that is suboptimal, typically by one order in the mesh size (i.e., of the form $Ch^p|u|_{H^{p+1}(\D)}$). To remedy this difficulty, we shall equip $V_h$ with a stronger norm inspired by the DG method, thereby leading to (asymptotic) quasi-optimality in~\eqref{eq:apriori}, that is, an upper bound of the form $Ch^{p+\frac12}|u|_{H^{p+1}(\D)}$.
\end{rmrk}

\section{Model problem: Advection-reaction equation}
\label{sec:AR}

In this section, we present examples of the above formulations in the context of the advection-reaction model problem. 

\subsection{Continuous setting} \label{sec:AR_setting}

Let $\D \subset \R^d$, with $d = 1,2,3$, be an open, bounded Lipschitz polyhedron with boundary $\Gamma :=\partial \D$ and outward unit normal $\Vn$. Let $\gamma \in L^\infty(\D)$ denote a bounded reaction coefficient, and let $\Vb \in L^\infty(\D;\R^d)$ denote velocity field such that $\nabla\cdot\Vb \in L^\infty(\D)$.  Assume that the boundary $\Gamma$ can be split into the three subsets $\Gamma^{\pm} := \{ \Vx \in \Gamma \, : \, \pm \Vb(\Vx) \cdot \Vn(\Vx) < 0\}$ (inflow/outflow) and $\Gamma^0 := \{ \Vx \in \Gamma \, : \, \Vb(\Vx) \cdot \Vn(\Vx) = 0\}$ (characteristic boundary).  Finally, let $f \in L^2(\D)$ denote a source term and $g \in L^2\left(|\Vb\cdot \Vn| ; \Gamma\right)$ denote a boundary datum, where
\begin{equation}\label{eq:L2border}
  L^2\left(|\Vb\cdot \Vn| ; \Gamma\right) := \left\{ v \text{ is
      measurable on } \Gamma \, : \, \int_\Gamma |\Vb\cdot \Vn| v^2 \,
    d\Gamma < \infty \right\}.
\end{equation}
The advection-reaction problem reads:
\begin{align}\label{eq:adv-reac}
  \left\{\begin{array}{l}
           \text{Find } u  \text{ such that:} \smallskip \\
           \begin{array}{rl}
             \Vb \cdot \nabla u + \gamma\, u = f & \text{ in } \D, \smallskip\\
             u  = g & \text{ on } \Gamma^-.
           \end{array}
         \end{array}\right.
\end{align}
Consider the graph space
$V:=\left\{v \in L^2(\D) \, : \, \Vb \cdot \nabla v \in L^2(\D) \right\}$
equipped with the inner product 
$(z,v)_V := (z,v)_{\D} + \left(\Vb \cdot \nabla z, \Vb \cdot \nabla v \right)_{\D}$ leading to the norm such that $\|v\|^2_{V} := (v,v)_V$. Then, $V$ is a Hilbert space. Moreover, assuming that $\Gamma^-$ and $\Gamma^+$ are well separated, that is, $d(\Gamma^-,\Gamma^+)>0$, traces are well-defined over $V$, in the sense that the operator:
\begin{equation}
  \begin{array}{rcl}
    \tau &:& C^0(\overline{\D}) \rightarrow L^2\left(|\Vb\cdot \Vn| ; \Gamma\right)\\
         && v \rightarrow \tau(v):=v|_{\Gamma},
  \end{array}
\end{equation}
extends continuously to $V$ (cf.,~\cite{ Ern2006}).

One possible weak formulation of~\eqref{eq:adv-reac} with weak imposition of boundary condition reads:
\begin{equation}\label{eq:cont_VF_adv}
\left\{\begin{array}{l}
\text{Find } u \in V, \text{ such that:} \smallskip\\ 
b(u,v) = (f,v)_{\D} + \left<(\Vb \cdot \Vn)^\ominus g\, , \, v \right>_\Gamma, \quad \forall \, v \in V,
\end{array} \right.
\end{equation}
with $(\cdot)^{\ominus}$ denoting the negative part such that $x^\ominus := \frac{1}{2}\left(|x|-x\right)$ for any real number $x$, and the continuous bilinear form $b$ is such that:
\begin{equation}\label{eq:cont_bil_adv}
  \displaystyle b(z,v) := \left(\Vb \cdot \nabla z + \gamma\, z \, , \, v \right)_{\D} + \left< (\Vb \cdot \Vn)^\ominus \, z \, , \, v\right>_{\Gamma}.
\end{equation}
The model problem~\eqref{eq:cont_VF_adv} is well-posed under some mild assumptions on $\gamma$ and $\Vb$ (see~\zzz{\cite{ Ern2006, CanCR2017} for details}).  \zzz{The problem is well-posed for specific right-hand sides of the form~\eqref{eq:cont_VF_adv} (and not for any right-hand side in $V^*$). That is, the} well-posedness is not established by reasoning directly on~\eqref{eq:cont_VF_adv}, but on an equivalent weak formulation where the boundary condition is strongly enforced by invoking a surjectivity property of the trace operator.

\subsection{DG discretization}

Recall the discrete setting from Section~\ref{sec:discrete}. For a~polynomial degree $p\ge0$, the standard DG discretization of the model problem~\eqref{eq:cont_VF_adv} is of the form:
\begin{equation}\label{eq:VF_adv-reac}
  \left\{\begin{array}{l}
           \text{Find } \theta_h \in V_h := \Pol^p(\D_h), \text{ such that:} \\
           % \text{such that:} 
           \displaystyle b^{\textrm{dg}}_h(\theta_h,v_h):= b_h(\theta_h,v_h) + p_h(\theta_h,v_h) = l_h(z_h), \quad \forall \, v_h \in V_h,
         \end{array} \right.
\end{equation}
where $b_h(\cdot,\cdot)$ and $p_h(\cdot,\cdot)$ are the bilinear forms over $V_h \times V_h$ such that
\begin{align}
    b_h(z_h,v_h) :=&   \sum_{m=1}^N(\Vb \cdot \nabla  z_h + \gamma \,  z_h\, , \, v_h)_{K_m} + \sum_{e \in \Sk^\partial_h}\left< ( \Vb \cdot \Vn_e)^\ominus z_h \, , \, v_h \right>_e \\ &  - \sum_{e \in \Sk^0_h} \left<(\Vb \cdot \Vn_e) \, \llbracket z_h \rrbracket \, , \, \llbrace v_h \rrbrace \right>_e,
\end{align}
and
\begin{equation}
\displaystyle p_h(z_h,v_h) :=  \dfrac{\eta}{2}\sum_{e \in \Sk^0_h} \left<\left|\Vb \cdot \Vn_e \right|  \llbracket z_h\rrbracket, \llbracket v_h \rrbracket \right>_e,
\end{equation}
where $\eta\ge0$ is the penalty parameter, and $l_h(\cdot)$ is the linear form over $V_h$ such that
\begin{equation}
l_h(v_h) := \sum_{m=1}^N (f, v_h)_{K_m} + \sum_{e \in \Sk_h^\partial}\left<(\Vb \cdot \Vn_e)^\ominus\, g \, , \, v_h\right>_e .
\end{equation}
The choice $\eta=0$ corresponds to the use of centered fluxes, and the choice $\eta>0$ (typically $\eta=1$) to the use of upwind fluxes, see~\cite{ BrMaS:04, di2011mathematical}.  To verify that the assumptions~\ref{as:inf-sup}-\ref{as:bound} are satisfied, we need to make a reasonable regularity assumption on the exact solution.

\begin{ssmptn}(Partition of ~$\D$ and regularity of exact solution $u$) \label{as:reg_u} 
There is a~partition $\Par_\D = \bigcup_{i = 1}^{N_\D} \D_i $ of $\D$ into open disjoint polyhedra $\D_i$ such that: 
\begin{enumerate}[(i)]
\item The inner part of the boundary of $\D_i$ is characteristic with respect to the advective field, that is, $\Vb(\Vx) \cdot \Vn_{\D_i}(\Vx) = 0$ for all $\Vx\in \partial\D_i\cap\D$ and all $i\in\{1,\ldots,N_\D\}$, where $\Vn_{\D_i}$ is the unit outward normal to $\D_i$; 
\item The exact solution is such that
  \begin{equation} \label{eq:def_Xsharp_AR}
    u \in X_\# := V \cap H^s(\Par_\D), \quad s>\frac12.
  \end{equation}
\item The mesh is aligned with this partition, that is, any mesh cell belongs to one and only one subset $\D_i$.
\end{enumerate}

\end{ssmptn}
Assumption~\ref{as:reg_u} \zzz{ensures} that the trace of $u$ is meaningful on the boundary of each mesh cell, and that $u$ can only jump across those interfaces that are subsets of some $\partial \D_i\cap \D$ with $i\in\{1,\ldots,N_\D\}$. In the case where $u$ has no jumps, we can broadly assume that $\Par_\D=\D$, in which case the statement of Assumption~\ref{as:reg_u}-(i) is void.

\zzz{We consider the following two norms with which to equip the broken polynomial space $V_h$:} 
\begin{equation}\label{eq:adv_norms}
  \begin{aligned}
    \| w \|^2_{\textrm{cf}} &:=  \|w\|^2_{\D} + \frac12 \left\||\Vb \cdot \Vn|^{\frac12} \, w\right\|_{\Gamma}^2 \, , \\
    \| w \|^2_{\textrm{up}} &:=  \| w \|^2_{\textrm{cf}} + \sum_{e \in \Sk_h^0} \dfrac{\eta}{2} \left\| \, |\Vb \cdot \Vn_e|^{\frac12} \llbracket w \rrbracket \right\|_e^2 + \sum_{m=1}^N  h_{K_m} \|\ \Vb \cdot \nabla w\|^2_{K_m}\, ,
  \end{aligned}
\end{equation}
\zzz{where $\| \cdot\|_{\textrm{cf}}$ relates to the central fluxes and $\| \cdot\|_{\textrm{up}}$ relates to the upwind fluxes.} In view of Assumption~\ref{as:bound}, we define $V_{h,\#}:=X_\#+V_h$ with $X_\#$ \zzz{as} in~\eqref{eq:def_Xsharp_AR}, and we consider the following extensions of the above norms:
\begin{equation}
\begin{aligned}
\| w \|^2_{\textrm{cf},\#} &:=  \| w \|^2_{\textrm{cf}} + \sum_{m=1}^N \left(\|\Vb \cdot \nabla w\|^2_{K_m} +  h_{K_m}^{-1}\|w\|^2_{\partial K_m} \right)\,,  \\
\| w \|^2_{\textrm{up},\#} &:=  \| w \|^2_{\textrm{up}} + \sum_{m=1}^N \left( h_{K_m}^{-1}\|w\|^2_{K_m} + \|w\|^2_{\partial K_m} \right)\,.
\end{aligned}
\end{equation}
For the proof of the following result, we refer the reader to~\cite{ di2011mathematical}.

\begin{prpstn} \label{prop:adv_coer} (\zzz{Assumption verification}) 
In the above framework, Assumptions~\ref{as:inf-sup}-\ref{as:bound} hold true in the following situations: \textup{(i)} $\eta=0$ (centered fluxes) and the norms $\| \cdot \|_{\emph{cf}}$ and $\| \cdot \|_{\emph{cf},\#}$;
\textup{(ii)} $\eta>0$ (upwind fluxes) and the norms $\| \cdot \|_{\emph{up}}$ and $\| \cdot \|_{\emph{up},\#}$.
\end{prpstn}

\begin{rmrk}(Coercivity) 
The case of centered fluxes in Proposition~\ref{prop:adv_coer} leads to stability in the form of coercivity, whereas the case of upwind fluxes leads to an inf-sup condition, \zzz{which implies a stronger norm control} than with centered fluxes \zzz{since $\| \cdot\|_{\textrm{up}} \geq \| \cdot\|_{\textrm{cf}}$}.
\end{rmrk}

An immediate consequence of Theorem~\ref{thrm:well-posedness} is the bound
\begin{equation}\label{eq:adv_apriori}
\inf_{y_h \in V_h} \| u-y_h\|_{V_h} \le \| u - \theta_h\|_{V_h} \leq C \inf_{y_h \in V_h} \| u-y_h\|_{V_h,\#},
\end{equation}
with $C$ uniform with respect to the mesh size, where the lower bound is trivial \zzz{since} $\theta_h\in V_h$. \zzz{The following terminology is useful to compare} the decay rates with respect to the mesh size for the best-approximation errors of smooth solutions in the norms $\|\cdot\|_{V_h}$ and $\|\cdot\|_{V_h,\#}$. The estimate~\eqref{eq:adv_apriori} is suboptimal if the left-hand side \zzz{decays faster than} the right-hand side\zzz{. Also, the estimate is} (asymptotically) quasi-optimal when both sides decay at the same rate. The inequality in~\eqref{eq:adv_apriori} is suboptimal \zzz{for the} centered fluxes (\zzz{i.e.,} the left-hand side decays at a rate of order $h^{p+\frac12}$ while the right-hand side decays at rate $h^p$), and it is (asymptotically) quasi-optimal \zzz{for the} upwind fluxes (\zzz{i.e.,} both sides decay at a rate of order $h^{p+\frac12}$). \zzz{For further details see~\cite{ di2011mathematical}}.

\begin{rmrk}[Continuous trial space]
  If we set $U_h:=V_h\cap H^1(\D)$, that is, if $U_h$ is composed of continuous, piecewise polynomial functions, the discrete bilinear form $b_h(\cdot,\cdot)$ reduces on $U_h\times V_h$ to 
  \[
    b_h(z_h,v_h) = (\Vb \cdot \nabla  z_h + \gamma \,  z_h\, , \, v_h)_{\D} + \sum_{e \in \Sk^\partial_h}\left< ( \Vb \cdot \Vn_e)^\ominus z_h \, , \, v_h \right>_e,
  \]
  since functions in $U_h$ have \zzz{no} jumps across the mesh interfaces, and their piecewise gradient coincides with their weak gradient over $\D$.
\end{rmrk}

\begin{rmrk}[Best approximation] \label{rem:best_app}
  An interesting property of the present setting regarding the approximation capacity of the discrete spaces $U_h$ and $V_h$ is that, for all $v\in H^s(\D)$, $s>\frac12$,
  \[
    \inf_{v_h\in U_h} \|v-v_h\|_{\textrm{up},\#} \le C_{\rm app} \, \inf_{v_h\in V_h} \|v-v_h\|_{\textrm{up},\#},
  \]
  with $C_{\rm app}$ uniform with respect to the mesh size (see~\ref{ap:best_proof}). The converse bound is trivially satisfied with constant equal to $1$ since $U_h\subset V_h$.
\end{rmrk}

\section{Numerical examples} \label{sec:num}

In this section, we present the 2D and 3D test cases, cover some implementation aspects, and discuss \zzz{some} numerical results.
 
\subsection{Model problems}
\subsubsection{2D model problem}
We consider \tred{a~pure advection}  problem~\eqref{eq:adv-reac} ($\gamma = 0$) over the unit square $\D = (0 \, , \, 1)^2 \subset \R^2$, with a constant velocity field $\Vb = (3 \, , \, 1)^T$ (see Figure~\ref{fig:adv_domain}). We consider the source term $f = 0$ in $\D$ and, for a~parameter $M>0$, an inflow boundary datum $g=u_{M|\Gamma^-}$, where
$\Gamma^-=\{(0,y),y\in(0,1)\}\cup\{(x,0),x\in(0,1)\}$, and the exact solution $u_M$ is 
\begin{equation}\label{eq:adv_anal}
  u_M(x_1,x_2) = 1+\tanh\left(M \left(x_2-\dfrac{x_1}{3} - \dfrac{1}{2}\right) \right).
\end{equation}
The parameter $M$ tunes the width of the inner layer along the line of equation $x_2 = \dfrac{x_1}{3} + \dfrac{1}{2}$. The limit value as $M$ grows becomes
\begin{equation}\label{eq:u_infty}
  u_\infty(x_1,x_2) := \lim_{M\rightarrow +\infty} u_M(x_1,x_2) = 1+\sign\left(x_2 - \dfrac{x_1}{3} - \dfrac{1}{2}\right).
\end{equation}
\begin{figure}[t!]
    \centering
    \begin{subfigure}[b]{0.47\textwidth}
        \includegraphics[width=0.8\textwidth]{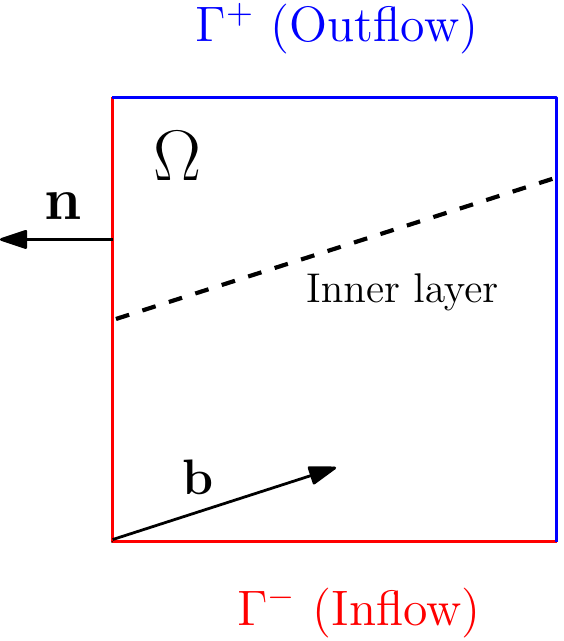}
        \caption{2D problem setting}
        \label{fig:adv_domain}
    \end{subfigure}
    \hspace{0.3cm}
    \begin{subfigure}[b]{0.37\textwidth}
    \includegraphics[width=0.8\textwidth]{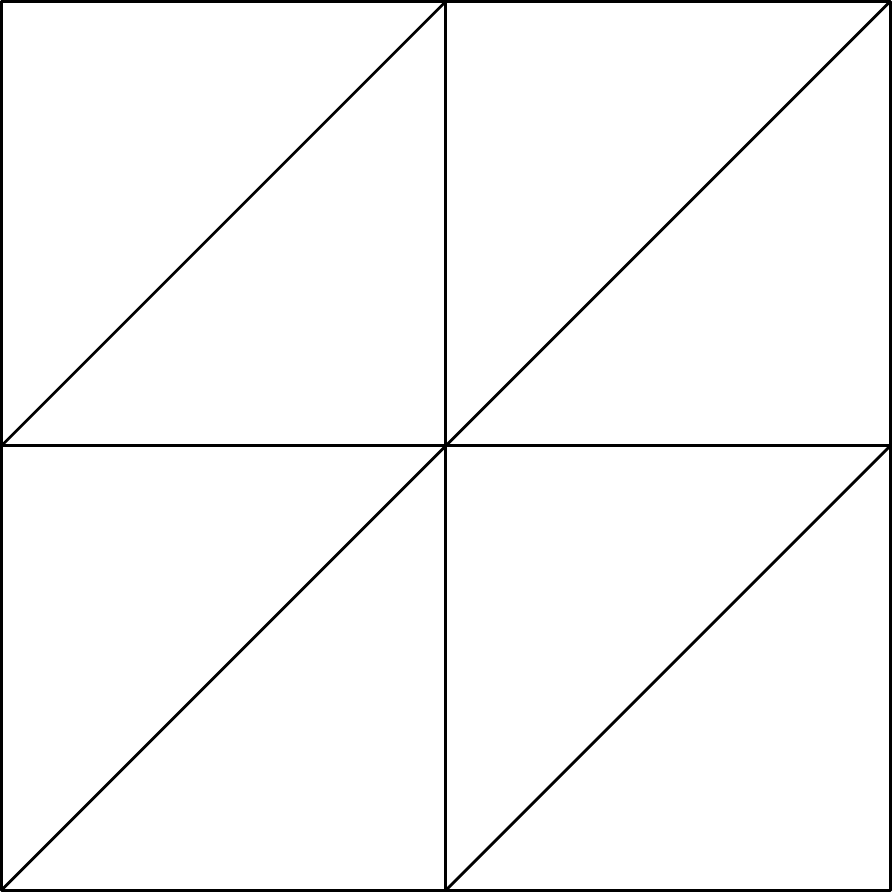}
    \vspace{0.8cm}
        \caption{2D initial mesh}
        \label{fig:ini_mesh}
    \end{subfigure} 
    \caption{2D model problem and initial mesh}
    \label{fig:data_set}
\end{figure}
\begin{figure}[h!]
     \begin{subfigure}[b]{0.48\textwidth}
        \includegraphics[angle=0,width=\textwidth]{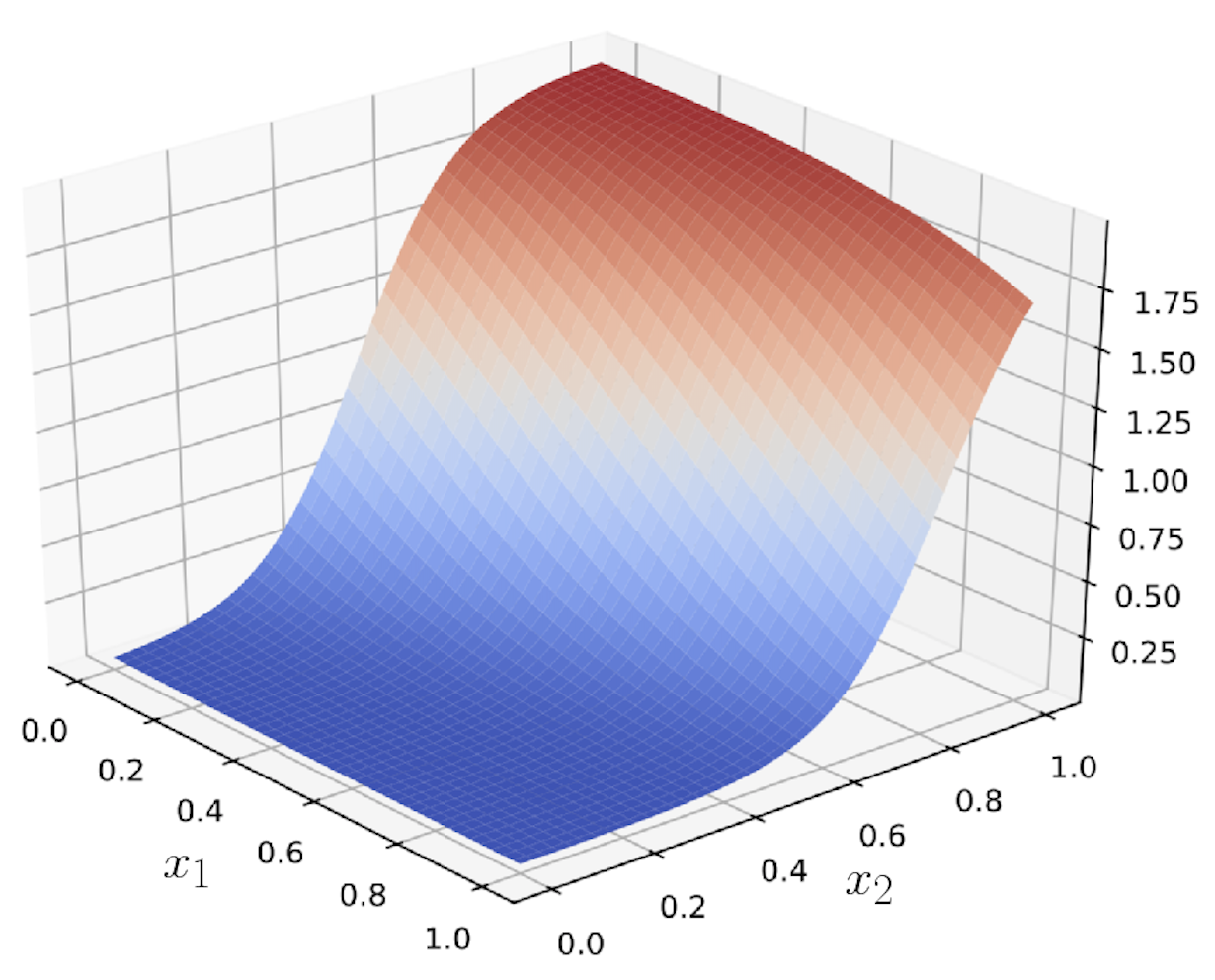}
        \caption{2D exact solution for $M=5$}
        \label{fig:anal_sol_M5}
     \end{subfigure}
     %
      %\hspace{cm}   
    \begin{subfigure}[b]{0.48\textwidth}
        \includegraphics[width=\textwidth]{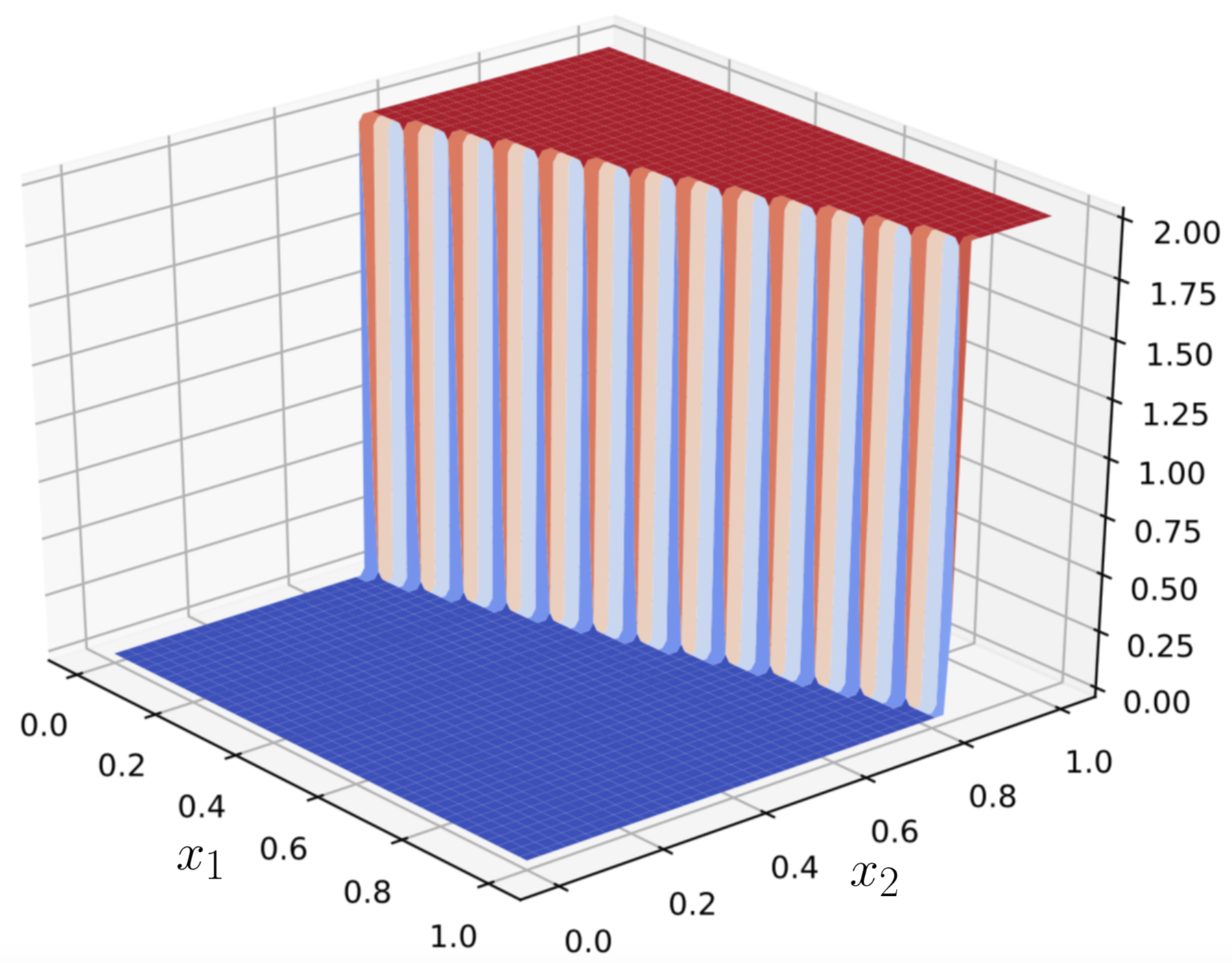}
        \caption{2D exact solution for $M=500$}
        \label{fig:anal_sol_M500}
    \end{subfigure}
  \caption{2D exact solutions for $M=5$ and $M=500$}  
  \label{fig:anal_sol}
\end{figure}
The inner layer can be seen in Figure~\ref{fig:anal_sol_M500} for the case $M=500$ (compare with Figure~\ref{fig:anal_sol_M5} for the case $M=5$). Although the inflow and outflow boundaries are not well-separated here, the exact solution matches the partition and regularity assumption~\ref{as:reg_u}. Indeed, $\Par_\D = \Omega$ when $M<\infty$, whereas $\Par_\D$ is composed of two subsets with the common characteristic interface $\{(x_1,x_2)\in \D\;|\; x_2-\frac{x_1}{3}-\frac12=0\}$ when $M=\infty$ (see Figure~\ref{fig:adv_domain}). In addition, the absence of a reactive term precludes the straightforward derivation of $L^2$-stability by a coercivity argument. However, $L^2$-stability is recovered via an inf-sup argument which remains valid whenever the advective field is filling (which \zzz{is the} case for a constant field across a square domain); we refer the reader to~\cite{ devinatz1974asymptotic, azerad1996inegalite, cantin2017edge, ayuso2009discontinuous} for further insight. 

\subsubsection{3D model problem}

\begin{figure}[t!]
    \centering
    \begin{subfigure}[b]{0.45\textwidth}
        \includegraphics[width=\textwidth]{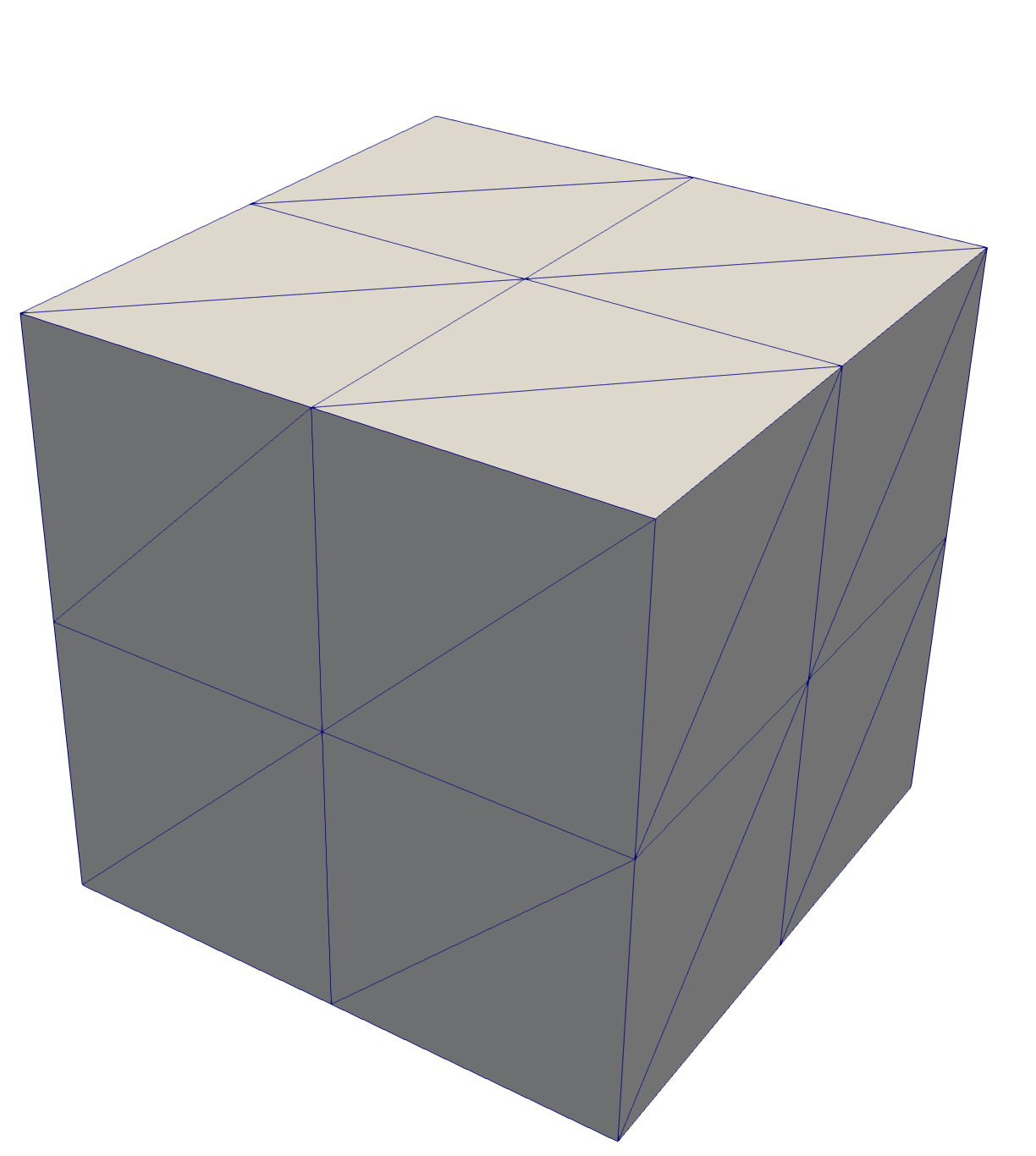}
        \caption{3D initial mesh}
        \label{fig:3d_mesh}
    \end{subfigure}
    \hspace{0.4cm}
   \begin{subfigure}[b]{0.47\textwidth}
    \includegraphics[width=\textwidth]{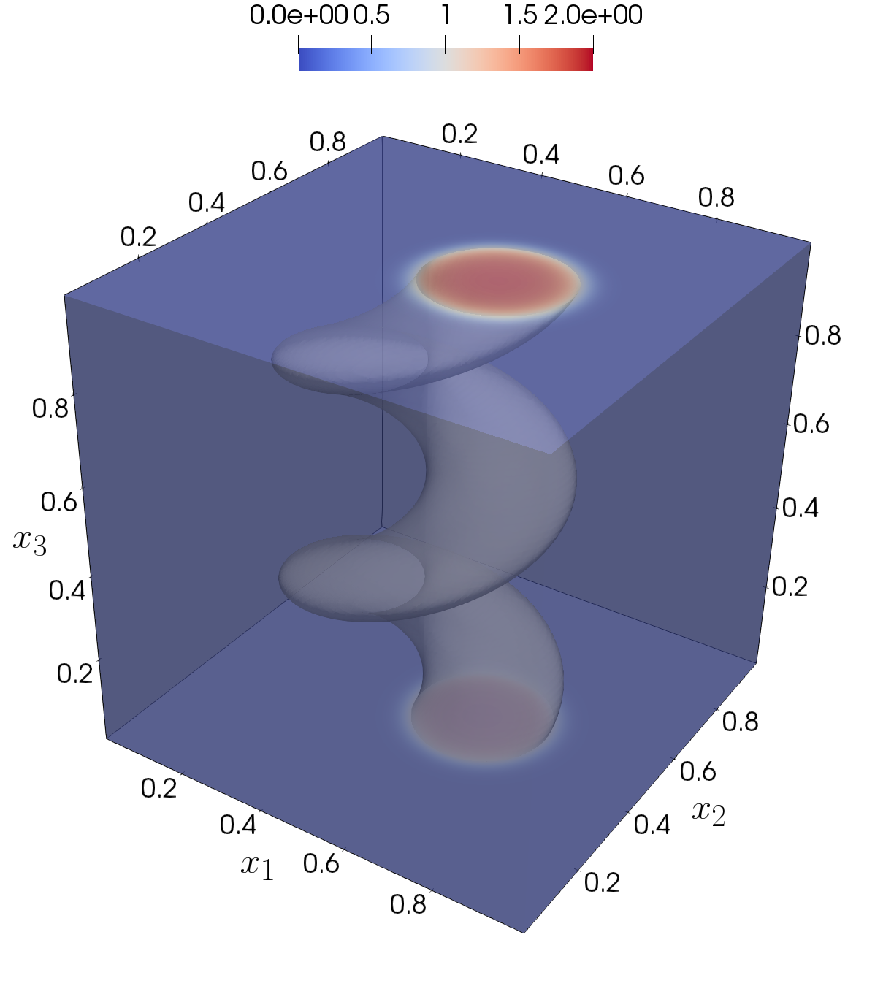}
        \caption{3D exact solution for $M=100$}
        \label{fig:3d_ref_sol}
    \end{subfigure} 
    \caption{3D initial mesh and exact solution}
    \label{fig:3d_data_set}
\end{figure}

We \zzz{again consider a~pure advection problem}~\eqref{eq:adv-reac}, this time over the unit cube $\D = (0 \, , \, 1)^3 \subset \R^3$, the source term $f=0$, the spiral-type velocity field $\Vb(x_1,x_2,x_3) =  (-0.15\,\sin(4\pi x_3) \, , \, 0.15\,\cos(4\pi x_3) \, , \, 1)^T$, and the inflow boundary datum $g = u_M|_{\Gamma^-}$, where the exact solution $u_M$ is given by
\begin{equation}\label{eq:ana_3D}
  u_M(x_1,x_2,x_3) = 1+\tanh\left(M \left( 0.15^2 - \big(x_1- X_1(x_3)\big)^2 - \big(x_2- X_2(x_3)\big)^2 \right) \right), 
\end{equation}
with $X_1(x_3) = 0.15 \, \cos(4\pi x_3) + 0.45$ and $X_2(x_3) = 0.15 \, \sin(4\pi x_3) + 0.5$.  \zzz{As $M$ grows, the limit value becomes:}
\begin{equation}
u_\infty(x_1,x_2,x_3) = 1+\sign\left( 0.15^2 - \big(x_1- X_1(x_3)\big)^2 - \big(x_2- X_2(x_3)\big)^2 \right).
\end{equation}
Figure~\ref{fig:3d_ref_sol} illustrates the inner layer for $M=100$. The inflow boundary corresponds to the following union \zzz{of plane portions:}
\begin{equation}\begin{aligned}
\Gamma^- = {}& \{(x_1,x_2)\in \Gamma_1\cup\Gamma_2,\ x_3\in [0,\tfrac18]\cup[\tfrac12,\tfrac58]\} \\
& \cup \{(x_1,x_2)\in \Gamma_2\cup\Gamma_3,\ x_3\in [\tfrac18,\tfrac14]\cup[\tfrac58,\tfrac34]\} \\
& \cup \{(x_1,x_2)\in \Gamma_3\cup\Gamma_4,\ x_3\in [\tfrac14,\tfrac38]\cup[\tfrac34,\tfrac78]\} \\
& \cup \{(x_1,x_2)\in \Gamma_4\cup\Gamma_1,\ x_3\in [\tfrac38,\tfrac12]\cup[\tfrac78,1]\}, 
\end{aligned}\end{equation}
where $\Gamma_1=\{x_1\in (0,1), x_2=0\}$, $\Gamma_2=\{x_1=1, x_2\in(0,1)\}$,
$\Gamma_3=\{x_1\in (0,1), x_2=1\}$, and $\Gamma_4=\{x_1=0, x_2=0\}$. As for the 2D model problem, the solution does not satisfy the regularity assumption~\ref{as:reg_u} since the inflow and outflow boundaries are not well separated; moreover, whenever $M=\infty$, the exact solution is piecewise smooth, but the subsets $\D_i$ in the corresponding partition are not polyhedra.

\subsection{Implementation aspects}

We consider the broken polynomial space $V_h := \Pol^p(\D_h)$, with $p=1,2$, defined in~\eqref{eq:def_bk_pol}. \zzz{The $H^1$-conforming} subspace $U_h:=V_h\cap H^1(\D)$\zzz{, which implies that $U_h$ contains the} continuous, piecewise polynomial functions of degree $p=1,2$. Since the trial space for the minimization problem~\eqref{eq:min_prob} is composed of continuous functions, we use the label ``CT'' in our figures. We equip the space $V_h$ with one of the two norms defined in~\eqref{eq:adv_norms}, leading to the labels ``CT-cf'' and ``CT-up''. For \zzz{comparison,} we also compute the DG solution $\theta_h\in V_h$ solving the primal problem~\eqref{eq:vf_discrete}. When reporting the corresponding error, we use the labels ``DT-cf'' and ``DT-up'', where "DT" means discontinuous trial space, and the labels ``cf'' and ``up'' indicate the use of centered fluxes and upwind fluxes (with $\eta=1$), respectively, as well as the norm in which \zzz{we evaluate the error.}  The solution ``CT-cf'' can be loosely interpreted as an LS-FEM solution. Indeed the residual minimization is performed over the $H^1$-conforming finite element space, and the minimizing functional is the $L^2$-norm of the residual in $\D$ supplemented by the $L^2$-norm in $\Gamma^-$ for the residual associated with the \zzz{inflow} boundary condition (see the norm $\|\cdot\|_{\textrm{cf}}$ in~\eqref{eq:adv_norms}).

In all 2D simulations, we start with the uniform triangular mesh shown in Figure~\ref{fig:ini_mesh}, whereas for 3D simulations, we start with the uniform tetrahedral mesh shown in Figure~\ref{fig:3d_mesh}. We produce subsequent mesh refinements under uniform and adaptive criteria. We obtain all the solutions of the saddle-point problem~\eqref{eq:mix_form} and the primal formulation~\eqref{eq:vf_discrete} by using FEniCS~\cite{ alnaes2015fenics}. We show convergence plots of the error measured in the chosen norm of $V_h$ as a function of the number of degrees of freedom (DOFs) (that is, $\dim(U_h) + \dim(V_h)$ for~\eqref{eq:mix_form} and $\dim(V_h)$ for~\eqref{eq:vf_discrete}). 

\subsubsection{Adaptive mesh refinement}

Adaptive mesh refinement is possible when solving~\eqref{eq:mix_form}, \zzz{and the} residual representative $\varepsilon_h\in V_h$ \zzz{drives that process}.  A standard adaptive procedure considers an iterative loop where each step consists of the following four modules:
$$
\text{ SOLVE } \rightarrow \text{ ESTIMATE } \rightarrow \text{ MARK } \rightarrow  \text{ REFINE. } 
$$
\zzz{We apply these four modules as} follows: We first solve the saddle-point problem~\eqref{eq:mix_form}. Then, we compute for each mesh cell $K$, the local error indicator $E_K$ \zzz{defined as:}
\begin{equation}
E_K ^2:= \begin{cases}
E_{\mathrm{cf},K}^2 := \|\varepsilon_h\|_K^2 + \frac12 \Big\||\Vb \cdot \Vn|^{\frac12} \, \varepsilon_h\Big\|_{\Gamma\cap \partial K}^2, 
&\text{if $\| \cdot \|_{V_h} = \| \cdot \|_{\textrm{cf}}$}, \\
E_{\mathrm{cf},K}^2 + \dfrac{\eta}{2} \Big\| \, |\Vb \cdot \Vn|^{\frac12} \llbracket \varepsilon_h \rrbracket \Big\|_{\Sk_h^0\cap\partial K}^2 + h_{K} \|\ \Vb \cdot \nabla \varepsilon_h\|^2_{K},&\text{if $\| \cdot \|_{V_h} = \| \cdot \|_{\textrm{up}}$}.
\end{cases}\end{equation}%
We mark using the D{\"o}rfler bulk-chasing criterion (see~\cite{ dorfler1996convergent}) that marks elements for which the cumulative sum of the local values $E_K$ in a decreasing order remains below a chosen fraction of the total estimated error $\|\varepsilon_h\|_{V_h}$. For the numerical examples, we consider this fraction to be one half and a quarter \zzz{for 2D and 3D,} respectively. Finally, we \zzz{employ} a bisection-type refinement criterion (see~\cite{ bank1983some}) to obtain the refined mesh \zzz{used} in the next step. 

\subsubsection{Iterative solver}

\zzz{The resulting algebraic form of~\eqref{eq:mix_form} becomes}
\begin{equation}
\left(
\begin{array}{cc}
G & B \\
B^\ast & 0
\end{array}
\right) 
\left(
\begin{array}{c}
{\bm \varepsilon}\\
{\bf u}
\end{array}
\right) =
\left(
\begin{array}{c}
\bf{l} \\
{\bf 0} 
\end{array}
\right).
\end{equation}
We \zzz{use} the iterative algorithm proposed in~\cite{ bank1989class}. Denoting by $\widehat{G}$ a~preconditioner for the Gram matrix $G$, and by $\widehat{S}$ a~preconditioner for the reduced Schur complement $B^\ast \widehat{G}^{-1} B$, the iterative scheme \zzz{is}
\begin{equation}\label{eq:iter_solver}
\left(
\begin{array}{c}
{\bm \varepsilon}_{i+1}\\
{\bf u}_{i+1}
\end{array}
\right)
= 
\left(
\begin{array}{c}
{\bm \varepsilon}_i\\
{\bf u}_i
\end{array}
\right)
+ 
\left(
\begin{array}{cc}
\widehat{G} & B \\
B^\ast & \widehat{C}
\end{array}
\right)^{-1} 
\left\{
\left(
\begin{array}{c}
\bf{l} \\
{\bf 0} 
\end{array}
\right)
-
\left(
\begin{array}{cc}
G & B \\
B^\ast & 0
\end{array}
\right) 
\left(
\begin{array}{c}
{\bm \varepsilon}_i\\
{\bf u}_i
\end{array}
\right)
\right\},
\end{equation}
with $\widehat{C} = B^\ast \widehat{G}^{-1} B - \widehat{S}$. Denoting by ${\bf r}_i= {\bf l} - G \, {\bm \varepsilon}_i - B\, {\bf u}_i$ and by ${\bf s}_i = - B^\ast {\bm \varepsilon}_i$ the residuals for ${\bm \varepsilon}$ and ${\bf u }$ at the outer iteration $i$, the scheme requires the resolution of two interior problems for the following increments:
\begin{equation}\label{eq:iter_solver_1}
\eta_{i+1}:={\bf u}_{i+1}-{\bf u}_i = \widehat{S}^{-1}\left(B^\ast\left(\widehat{G}^{-1} {\bf r}_i\right) - {\bf s}_i \right),
\end{equation}
and
\begin{equation}\label{eq:iter_solver_2}
\delta_{i+1}:={\bm\varepsilon}_{i+1}-{\bm\varepsilon}_{i} = \widehat{G}^{-1}\left({\bf r}_i - B \, {\bm\eta}_{i+1} \right).
\end{equation}
In~\cite{ bank1989class}, the authors \zzz{consider} $\widehat{G}$ as a relaxed approximation for the matrix $G$, for instance, a few iterations of the conjugate gradient method. However, in our context, the matrix $G$ plays \zzz{a vital} role in stabilizing the system. Therefore, less accurate representations worsen the conditioning of the reduced Schur complement in~\eqref{eq:iter_solver_1}. \zzz{As a consequence}, we consider one outer iteration \zzz{with} a sparse Cholesky factorization \zzz{of the matrix $G$ as preconditioner (using ``sksparse.cholmod'', see~\cite{ chen2008algorithm}). The} conjugate gradient method (\zzz{from} the Scipy sparse linear algebra package) is the preconditioner $\widehat{S}$. On the coarsest mesh,  we \zzz{set} the initial guess to be zero, whereas, on the subsequent adaptive meshes, this guess becomes the solution obtained in the previous level of refinement. 

\subsection{Discussion of the numerical results}

\subsubsection{Discussion of the 2D results}
As a first example, we \zzz{set} $M=5$ in the 2D exact solution~\eqref{eq:adv_anal} to obtain a sufficiently smooth solution that allows us to appreciate the expected convergence rates.  Figure~\ref{fig:p1_M5} reports the results obtained with uniform mesh refinement and the polynomial degree $p=1$. Figure~\ref{fig:un_L2_p1_M5} shows the error measured in the $L^2$-norm vs.~DOFs in log-scale, whereas Figure~\ref{fig:un_gr_p1_M5} shows the error measured in the corresponding $V_h$-norm vs.~DOFs in log-scale. The main point is that the ``CT-up'' solution converges at the same rate as the ``DT-up'' solution and with \zzz{similar errors measured using} both the $L^2$- and $\|\cdot\|_{\rm up}$-norms. Interestingly, the ``CT-cf'' solution converges with a higher rate when compared with the ``DT-cf'' solution in the $L^2$-norm, and that is also the case for this solution in the $\|\cdot\|_{\rm cf}$-norm.  In Figure~\ref{fig:p2_M5}, we consider the same type of results but using the polynomial degree $p=2$. The conclusions are similar. Again\zzz{, both} ``CT-up'' and ``DT-up'' solutions converge at the same rate for both the $L^2$- and $\|\cdot\|_{\rm up}$-norms, \zzz{for which both methods deliver almost identical} error values. Incidentally, the ``DT-cf'' solution also converges at the same rate for $p=2$, again in both norms. 
\begin{figure}[t!]
  \centering
  \begin{subfigure}[b]{0.49\textwidth}
    \includegraphics[width=\textwidth]{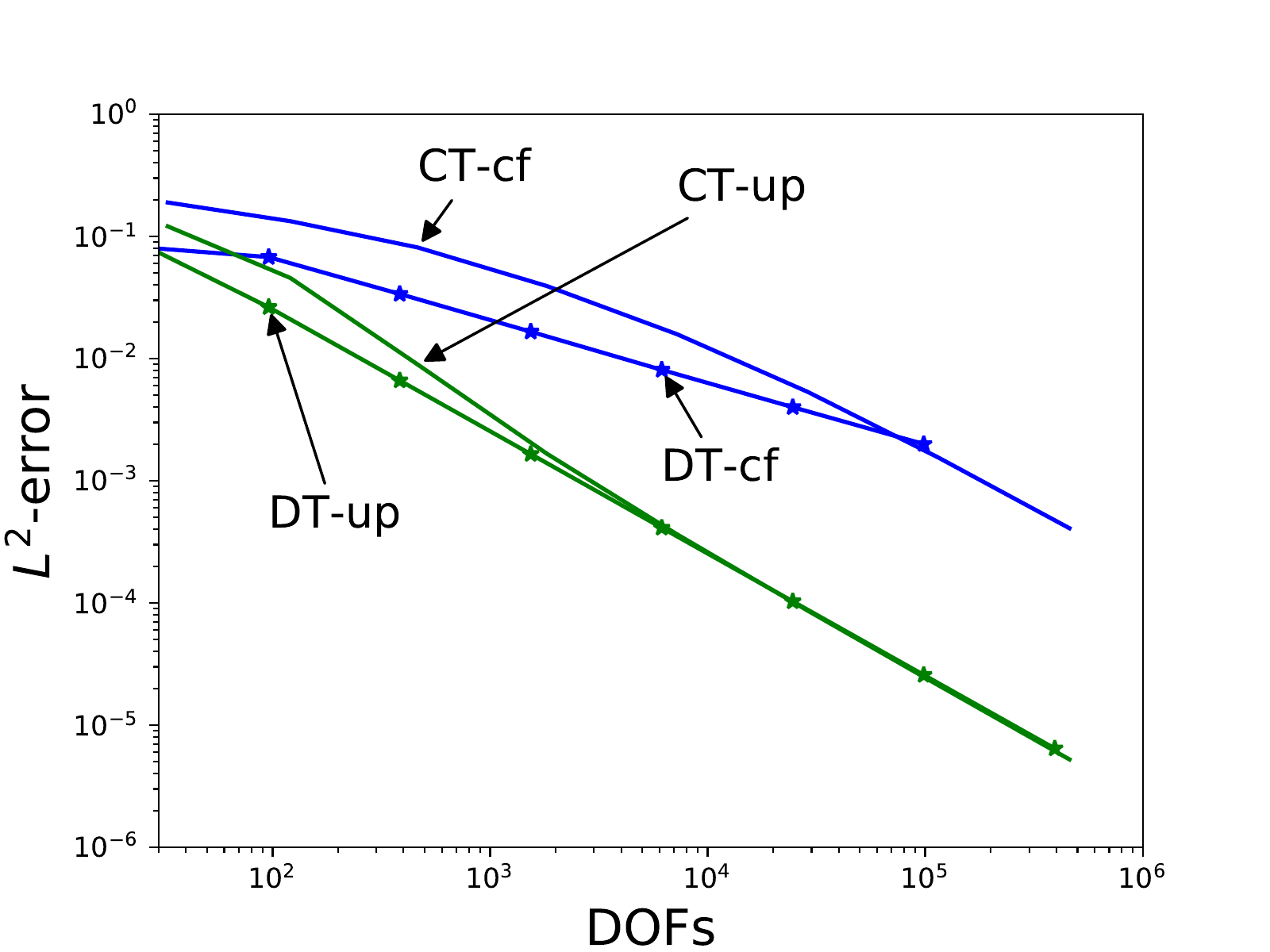}
    \caption{$L^2$ error vs.~DOFs}
    \label{fig:un_L2_p1_M5}
  \end{subfigure}
  \begin{subfigure}[b]{0.49\textwidth}
    \includegraphics[width=\textwidth]{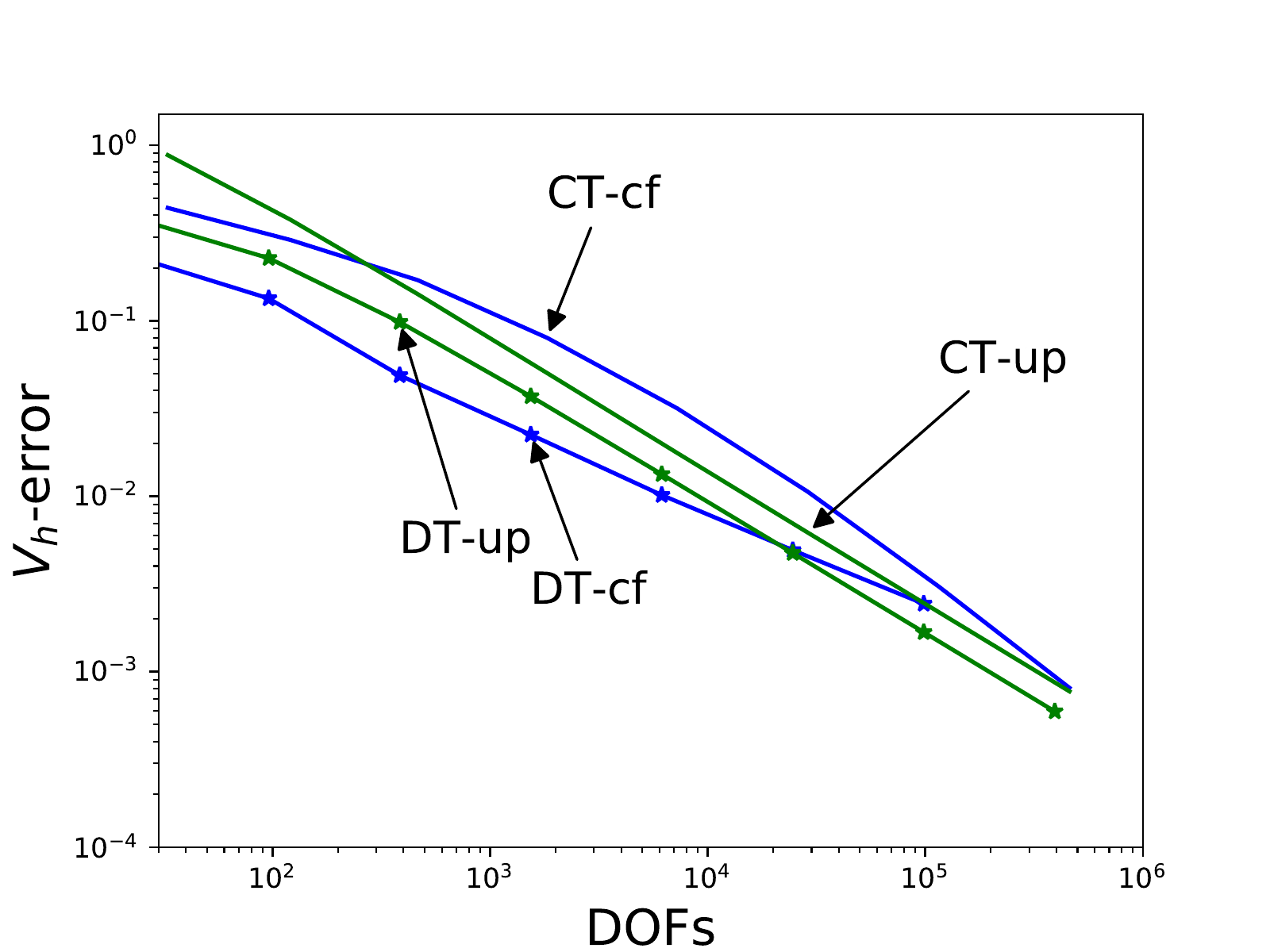}
    \caption{$V_h$ error vs.~DOFs}
    \label{fig:un_gr_p1_M5}
  \end{subfigure}
  \caption{2D model problem: $L^2$-error and $V_h$-error vs.~DOFs using uniform mesh refinement for $p=1$ and $M=5$.}
  \label{fig:p1_M5}
\end{figure}
\begin{figure}[h!]
  \centering
  \begin{subfigure}[b]{0.49\textwidth}
    \includegraphics[width=\textwidth]{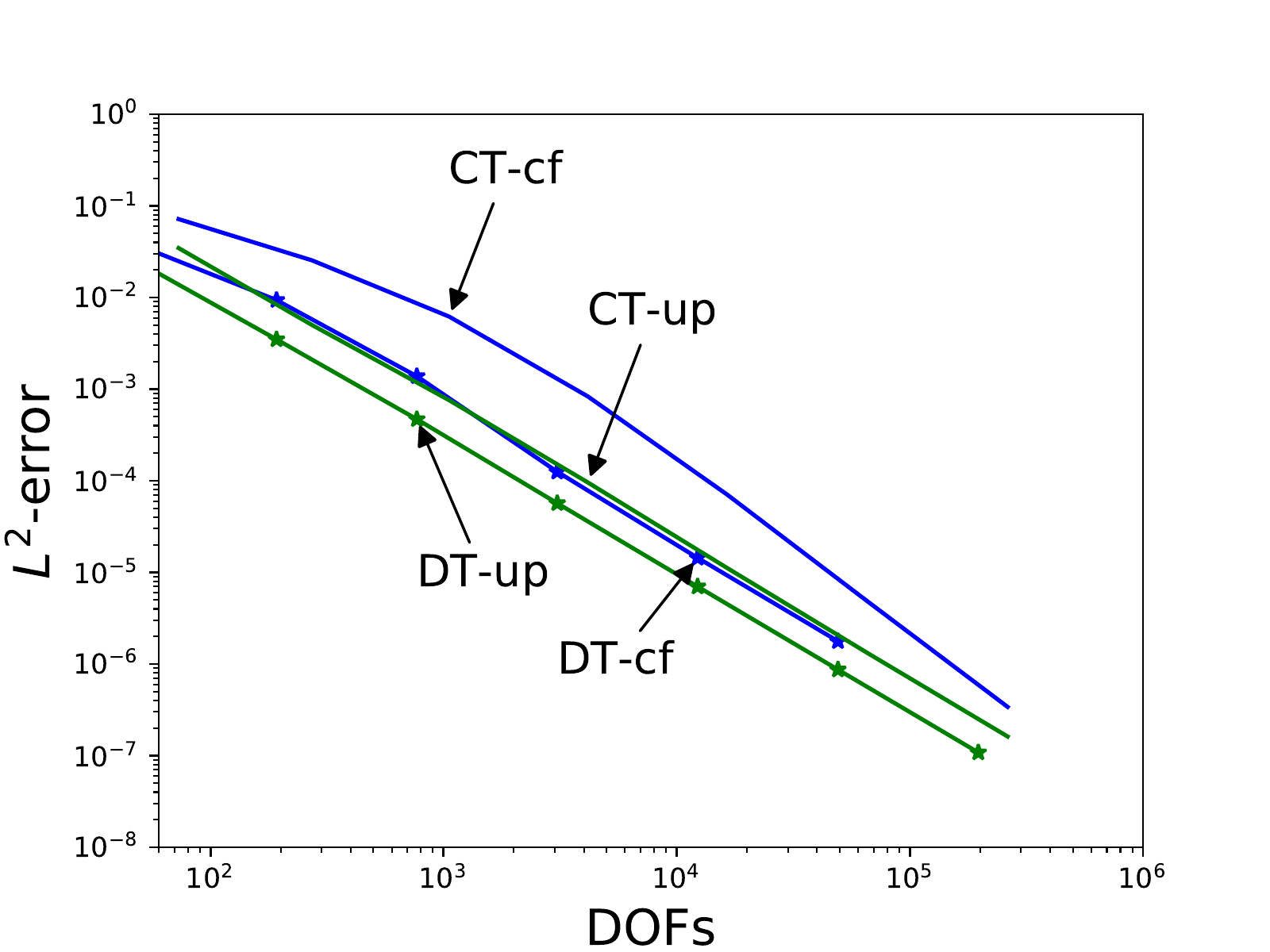}
    \caption{$L^2$ error vs.~DOFs}
    \label{fig:un_L2_p2_M5}
  \end{subfigure}
  % 
  % \hspace{0.4cm}
  \begin{subfigure}[b]{0.49\textwidth}
    \includegraphics[width=\textwidth]{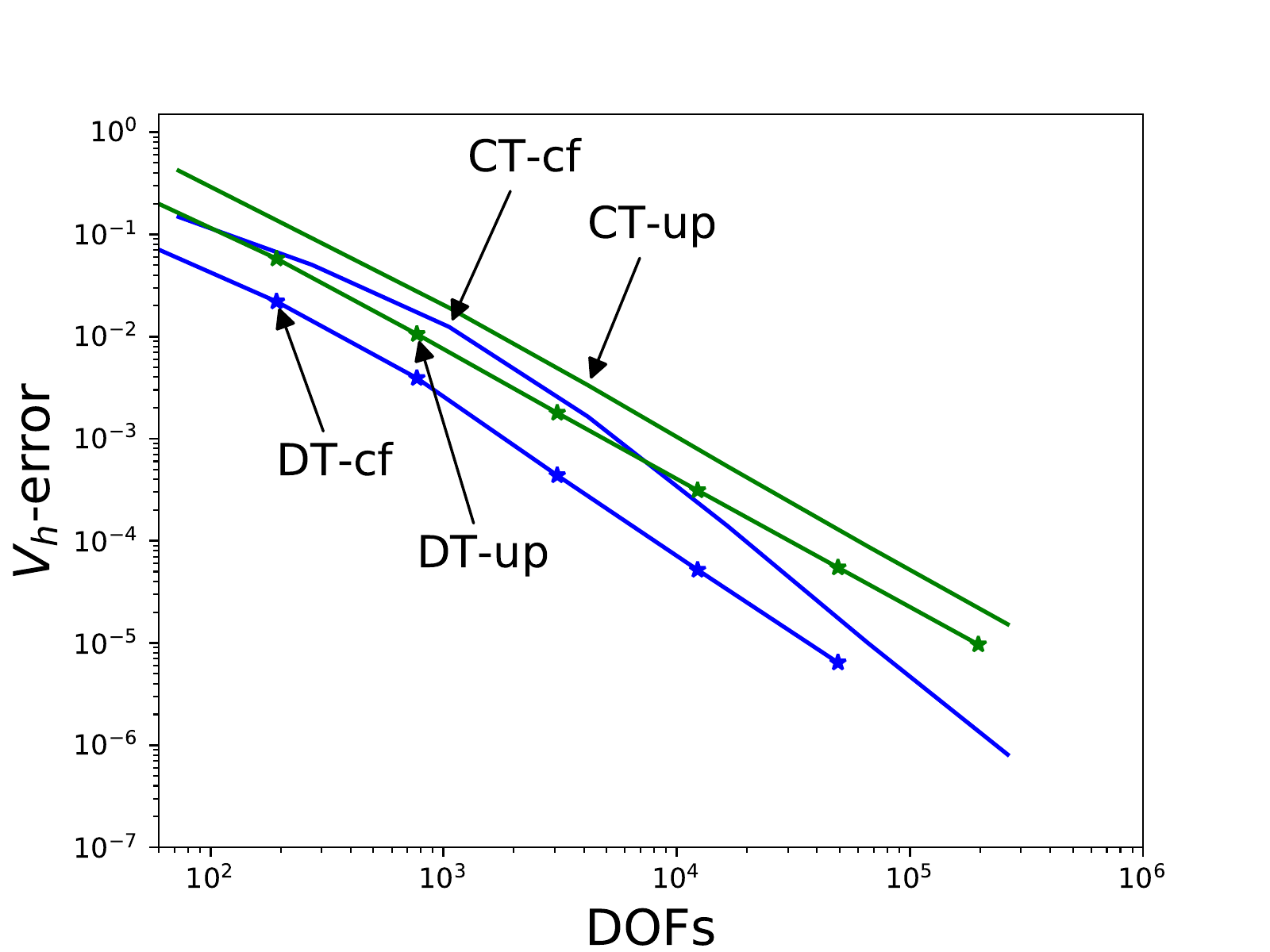}
    \caption{$V_h$ error vs.~DOFs}
    \label{fig:ad_gr_p2_M5}
  \end{subfigure}
  \caption{2D model problem: $L^2$-error and $V_h$-error vs.~DOFs using uniform mesh refinement for $p=2$ and $M=5$.}
  \label{fig:p2_M5}
\end{figure}

\tred{To study the validity of the saturation assumptions (see Assumptions~\ref{as:saturation} and \ref{as:weak}), we consider the $\|\cdot\|_{\textrm{up}}$-norm defined in~\eqref{eq:adv_norms} and, on each mesh, we introduce the following ratios:
\begin{equation}\label{eq:ratio}
\mathcal{S}:= \dfrac{\|u-\theta_h\|_{\textrm{up}}}{\|u-u_h\|_{\textrm{up}}}, \qquad  \mathcal{W}:= \dfrac{\|u-\theta_h\|_{\textrm{up}}}{\|\theta_h-u_h\|_{\textrm{up}}}.
\end{equation}
When $\mathcal{S}<1$, the approximation satisfies the saturation Assumption~\ref{as:saturation}, and thus also the weaker Assumption~\ref{as:weak}. Instead, the weaker Assumption~\ref{as:weak} is satisfied if there exists an upper bound, uniform with respect to the mesh size, for $\mathcal{W}$. Figures~\ref{fig:normp1u} and~\ref{fig:normp2u} show the quantities $\|u-u_h\|_{\textrm{up}}$, $\|u-\theta_h\|_{\textrm{up}}$, $\|\theta_h-u_h\|_{\textrm{up}}$, and $\|\varepsilon_h\|_{\textrm{up}}$ vs DOFs, for $p=1$ and $p=2$, respectively. All the quantities exhibit the same convergence rates, in agreement with the a~posteriori error estimates~\eqref{eq:bnd_uh-thetah} and~\eqref{eq:a_posteriori}.Figure~\ref{fig:satu} displays the ratios $\mathcal{S}$ and $\mathcal{W}$ defined in~\eqref{eq:ratio} vs DOFs, in log-scale, for $p=1$ and $p=2$. All meshes and polynomial orders, satisfy the saturation Assumption~\ref{as:saturation} and, a fortiori, the weaker Assumption~\ref{as:weak}.}
\begin{figure}[t!]
  \centering
  \begin{subfigure}[b]{0.49\textwidth}
    \centering
    \includegraphics[width=\textwidth]{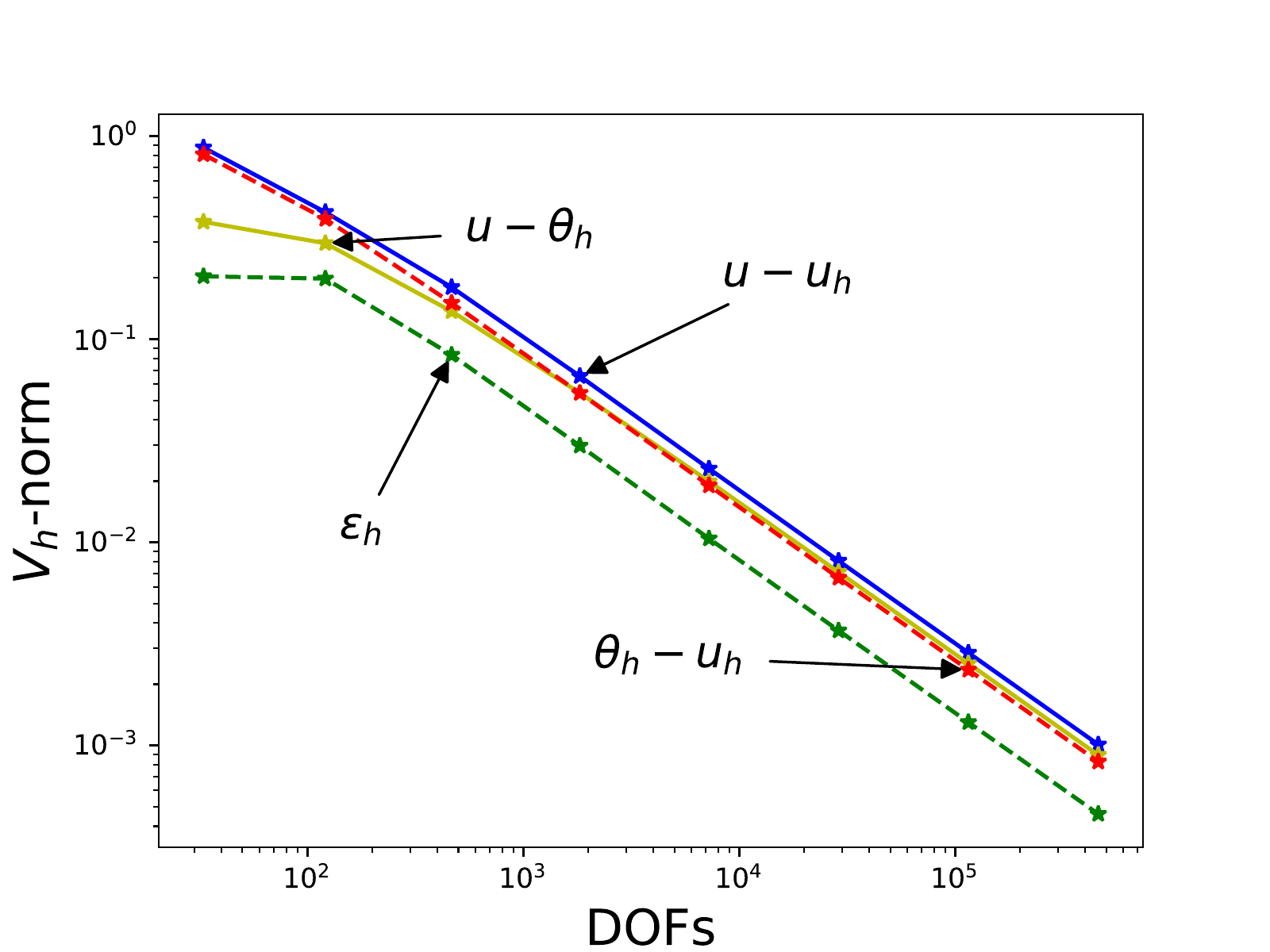}
    \caption{$p=1$}
    \label{fig:normp1u}
  \end{subfigure}
  % 
  % \hspace{0.4cm}
  \begin{subfigure}[b]{0.49\textwidth}
    \centering
    \includegraphics[width=\textwidth]{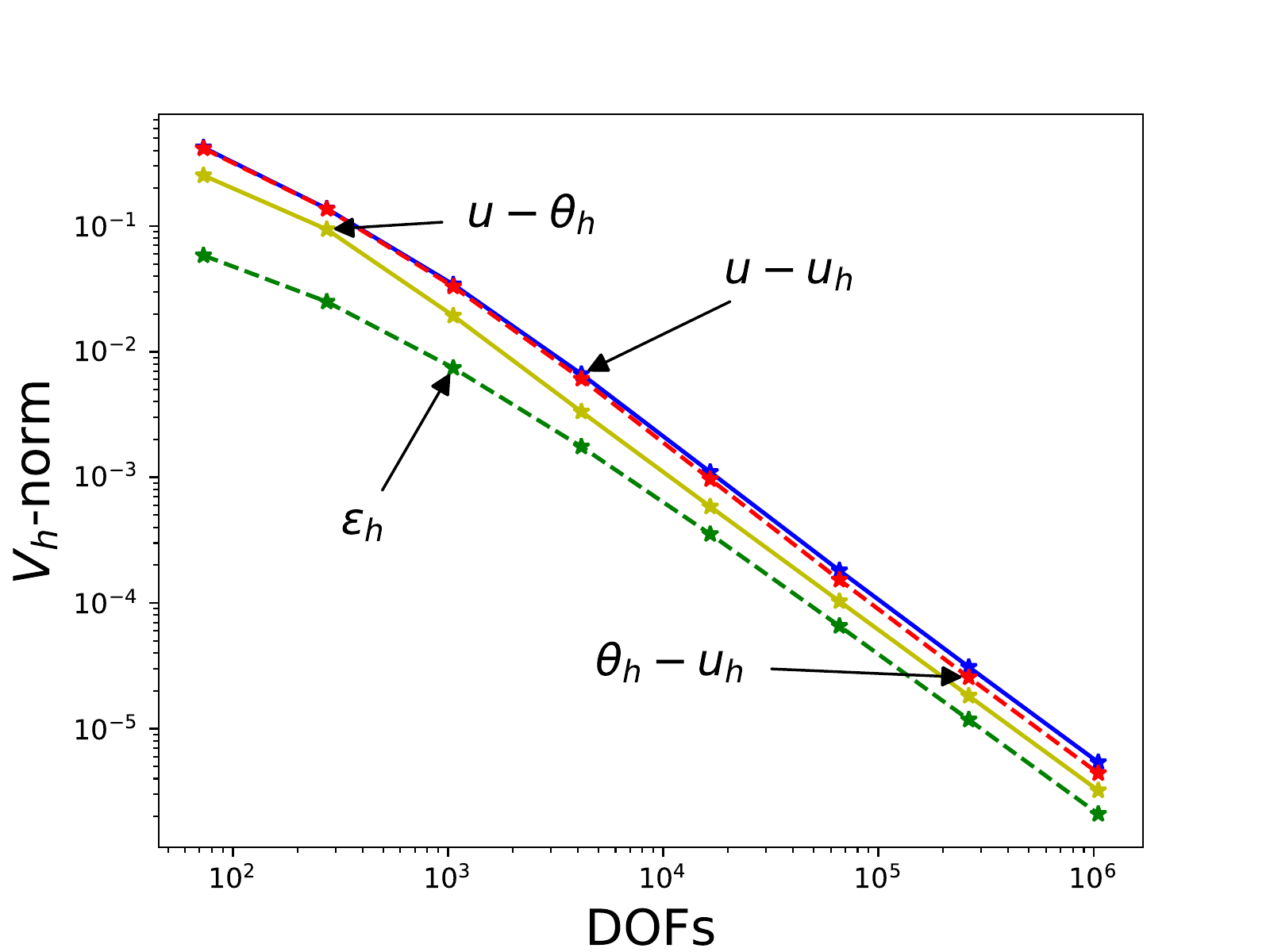}
    \caption{$p=2$}
    \label{fig:normp2u}
  \end{subfigure}
  \caption{$V_h$-norm vs DOFs for uniform meshes; CT-up, $p=1,2$ and $M=5$}
  \label{fig:norm_uniform}
\end{figure}
 \begin{figure}[t!]
 \vspace{0.4cm}
    \centering
    \begin{subfigure}[b]{0.49\textwidth} 
    \begin{picture}(100,100)
    \centering
    \put(0,0){\includegraphics[width=\textwidth]{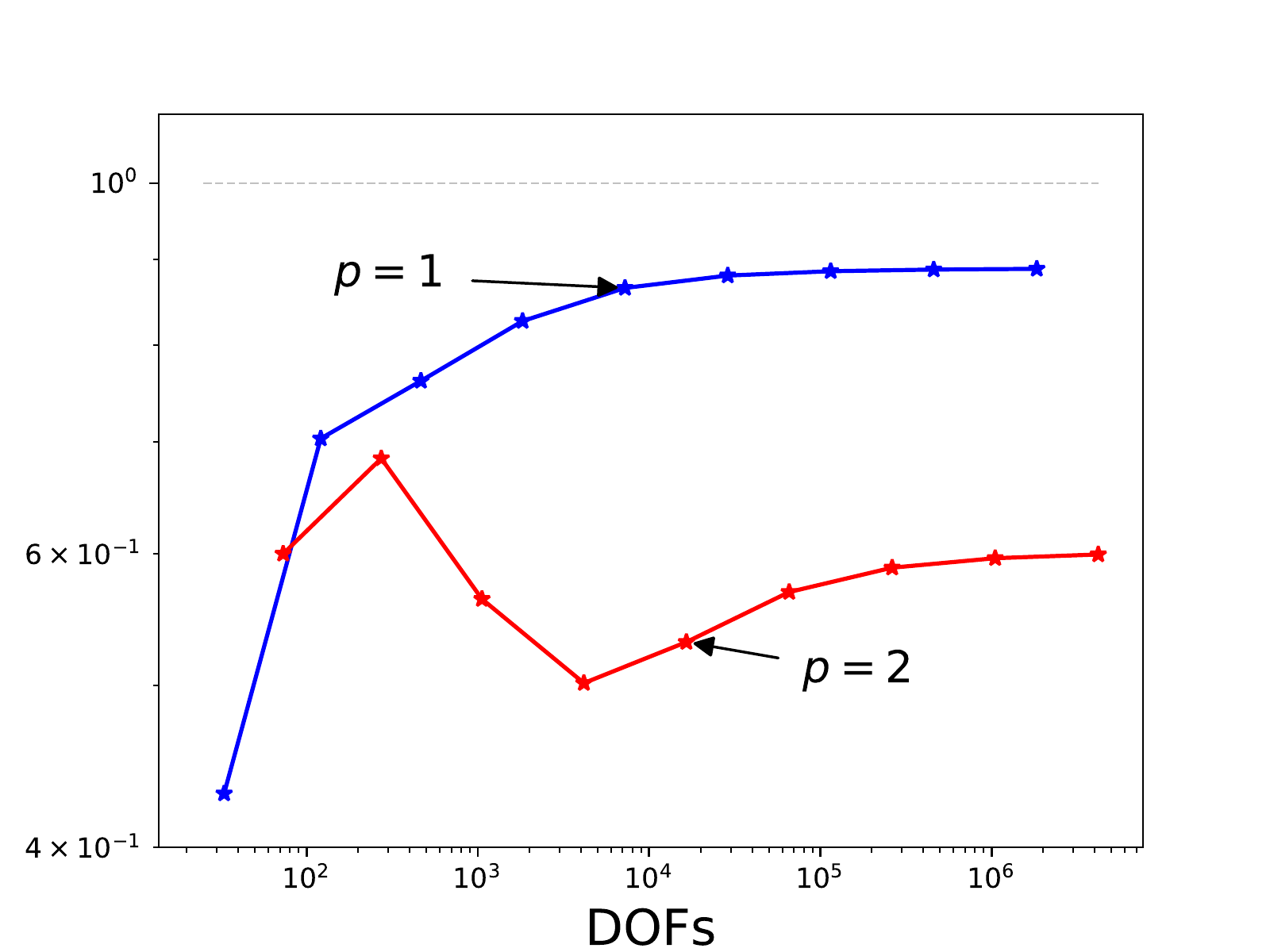}}
    \put(0,65){$\mathcal{S}$}
    \end{picture}
        \caption{Ratio $\mathcal{S}$}
        \label{fig:sat_strongu}
    \end{subfigure}
    %
   % \hspace{0.4cm}
    \begin{subfigure}[b]{0.49\textwidth}
    \begin{picture}(100,100)
    \centering
    \put(0,0){\includegraphics[width=\textwidth]{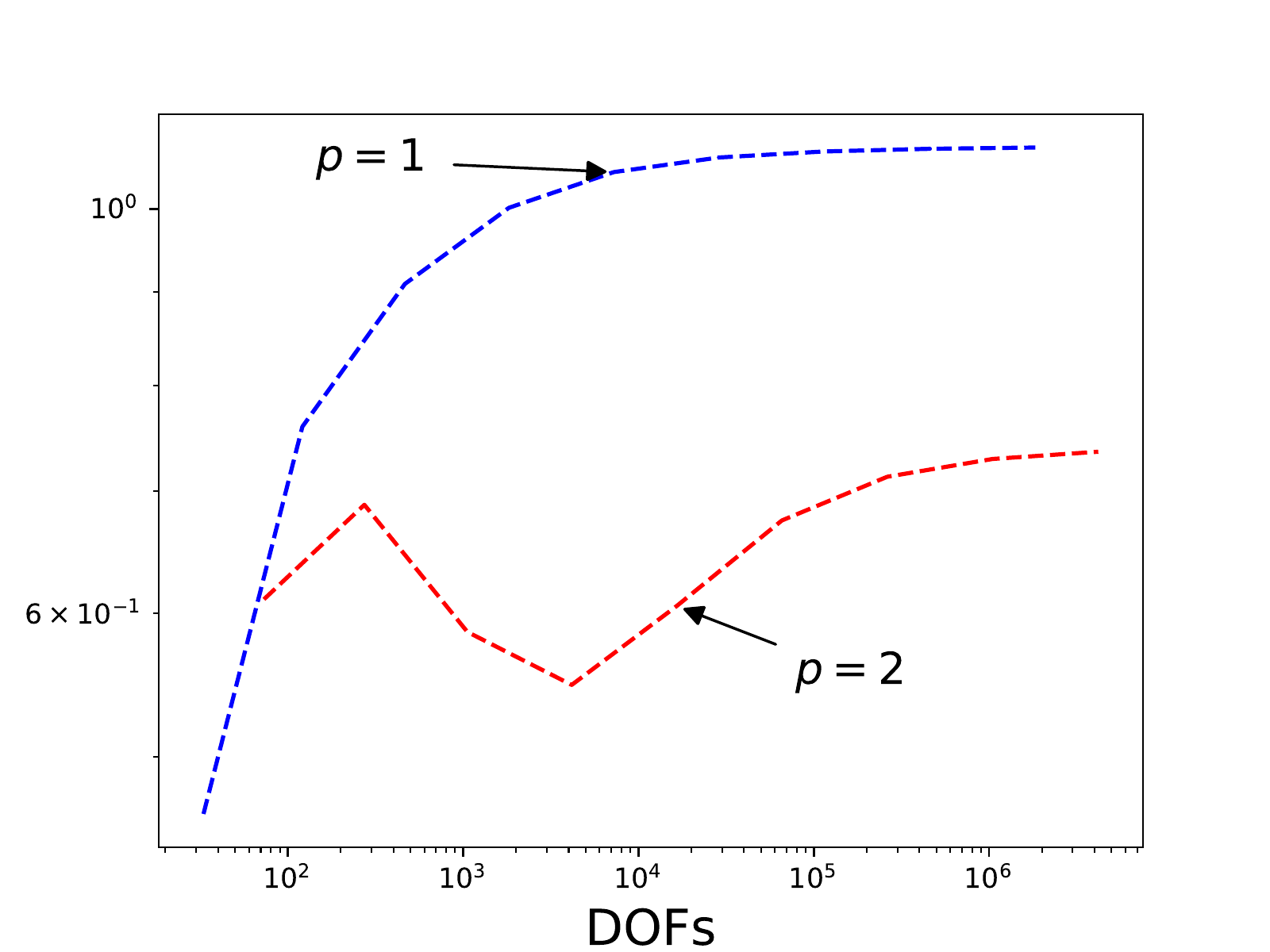}}
    \put(0,65){$\mathcal{W}$}
    \end{picture}
        \caption{Ratio $\mathcal{W}$}
        \label{fig:sat_weaku}
    \end{subfigure}
        \caption{Ratios $\mathcal{S}$ and $\mathcal{W}$ defined in~\eqref{eq:ratio} vs DOFs for uniform meshes; CT-up, $p=1,2$ and $M=5$}
        \label{fig:satu}
 \end{figure}  

As a second example, we consider the value $M=500$ for the 2D exact solution $u_M$ in~\eqref{eq:adv_anal} \zzz{such} that $u_M$ is very close to the discontinuous function $u_\infty$ defined in~\eqref{eq:u_infty} (see Figure~\ref{fig:anal_sol_M500}). In this case, \zzz{uniform refinements do not deliver optimal convergence as Figures~\ref{fig:un_L2_p1_M500}-\ref{fig:un_gr_p1_M500}, and~\ref{fig:un_L2_p2_M500}-\ref{fig:un_gr_p2_M500} show} for the polynomial degrees $p=1$ and $p=2$, respectively. However, as Figures~\ref{fig:ad_L2_p1_M500}-\ref{fig:ad_gr_p1_M500}, and~\ref{fig:ad_L2_p2_M500}-\ref{fig:ad_gr_p2_M500} show, the convergence improves for the ``CT-up'' solution by resorting to mesh adaptation. Figures~\ref{fig:cfmeshcut} and~\ref{fig:upmeshcut} compare \zzz{the adaptive refinement processes at a cut} of the solution over the line $\{x_1 = 1-x_2\}$ requiring a similar number of total DOFs. \zzz{The} choice of the norm in $V_h$ has a significant impact on the convergence of the adaptive process. \zzz{The method achieves a significantly faster convergence rate} if one uses the $\|\cdot\|_{\rm up}$-norm rather than the $\|\cdot\|_{\rm cf}$-norm. This \zzz{shows} that the residual \zzz{representative defined in}~\eqref{eq:e_h} can \zzz{guide the adaptive process. Also,} the use of the stronger norm $\|\cdot\|_{\rm up}$ drives \zzz{adaptivity} more efficiently than the $\|\cdot\|_{\rm cf}$-norm which would be considered in LS-FEM.  
\tred{In this case, $p=2$ satisfies the saturation Assumption~\ref{as:saturation}, measured using the up-norm, for all meshes. In contrast, for the $p =1$ case only on the finer meshes (and a few coarser meshes) satisfy this assumption as Figure~\ref{fig:sat} shows.
%\tred{Contrary to the smooth case, the saturation Assumption~\ref{as:saturation}, measured using the $\|\cdot\|_{\textrm{up}}$-norm, is always satisfied for $p=2$, whereas it is satisfied only on the finer meshes (and a few coarser meshes) for $p=1$, as can be seen in Figure~\ref{fig:sat}. 
In spite of this, Figures~\ref{fig:normp1} and ~\ref{fig:normp2} show similar convergence rates for the quantities $\|u-u_h\|_{\textrm{up}}$ and $\|\varepsilon_h\|_{V_h}$ for both polynomial degrees. 
Moreover, Figure~\ref{fig:sat_weak} shows that there is indeed an upper bound for the ratio $\mathcal{W}$, indicating that both cases satisfy the weaker Assumption~\ref{as:weak}.}
\begin{figure}[h!]
    \centering
    \begin{subfigure}[b]{0.49\textwidth}
    \centering
        \includegraphics[width=\textwidth]{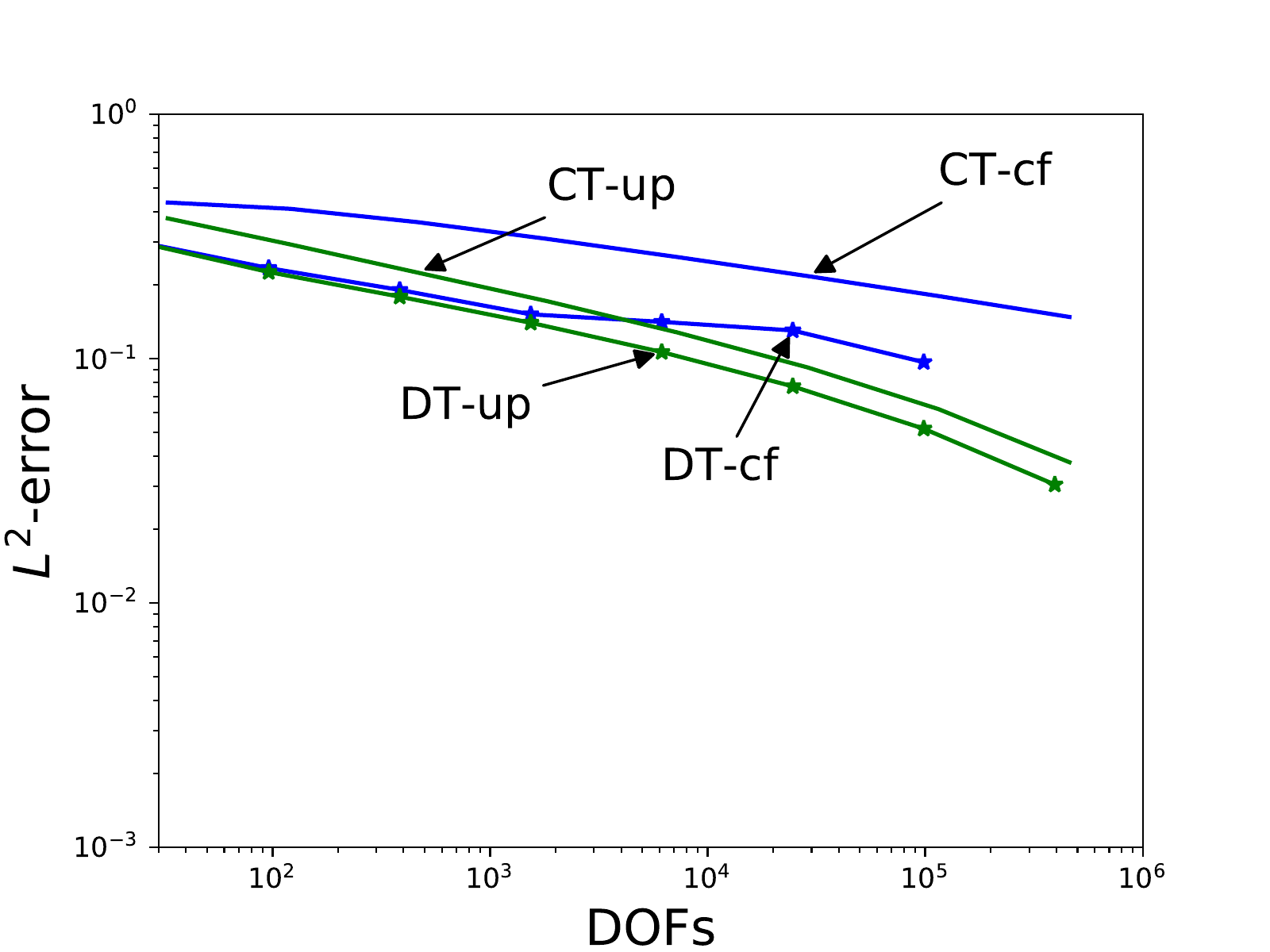}
        \caption{$L^2$ error vs.~DOFs; uniform}
        \label{fig:un_L2_p1_M500}
    \end{subfigure}
    %\hspace{0.4cm}
    \begin{subfigure}[b]{0.49\textwidth}
    \centering
        \includegraphics[width=\textwidth]{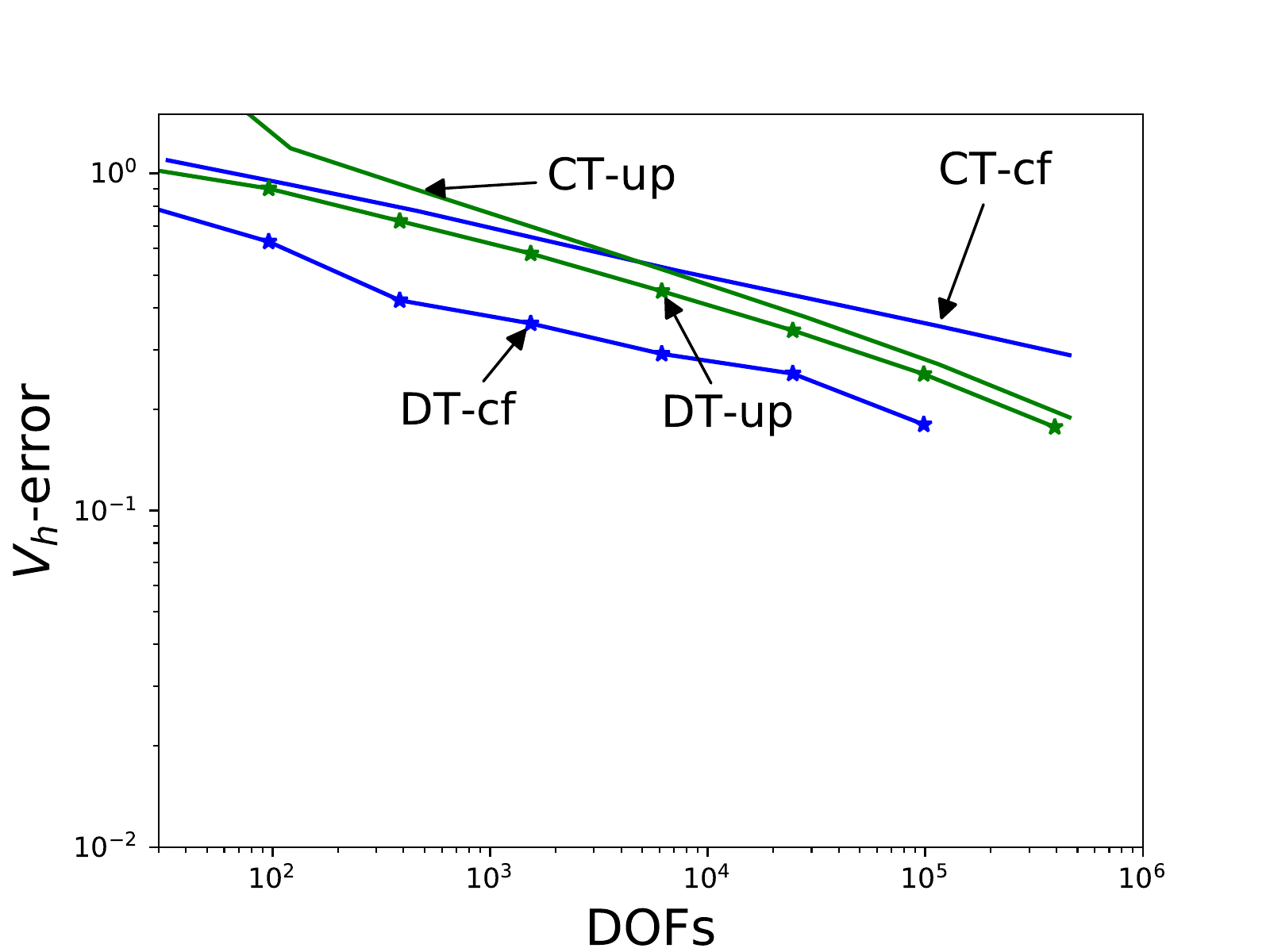}
        \caption{$V_h$ error vs.~DOFs; uniform}
        \label{fig:un_gr_p1_M500}
    \end{subfigure} 
    \caption{2D model problem: errors in the $L^2$- and $\|\cdot\|_{V_h}$-norms vs.~DOFs for uniform meshes, $p=1$, and $M=500$.}
 \end{figure}
    \begin{figure}[t!]
    \centering
    \begin{subfigure}[b]{0.49\textwidth}
    \centering
        \includegraphics[width=\textwidth]{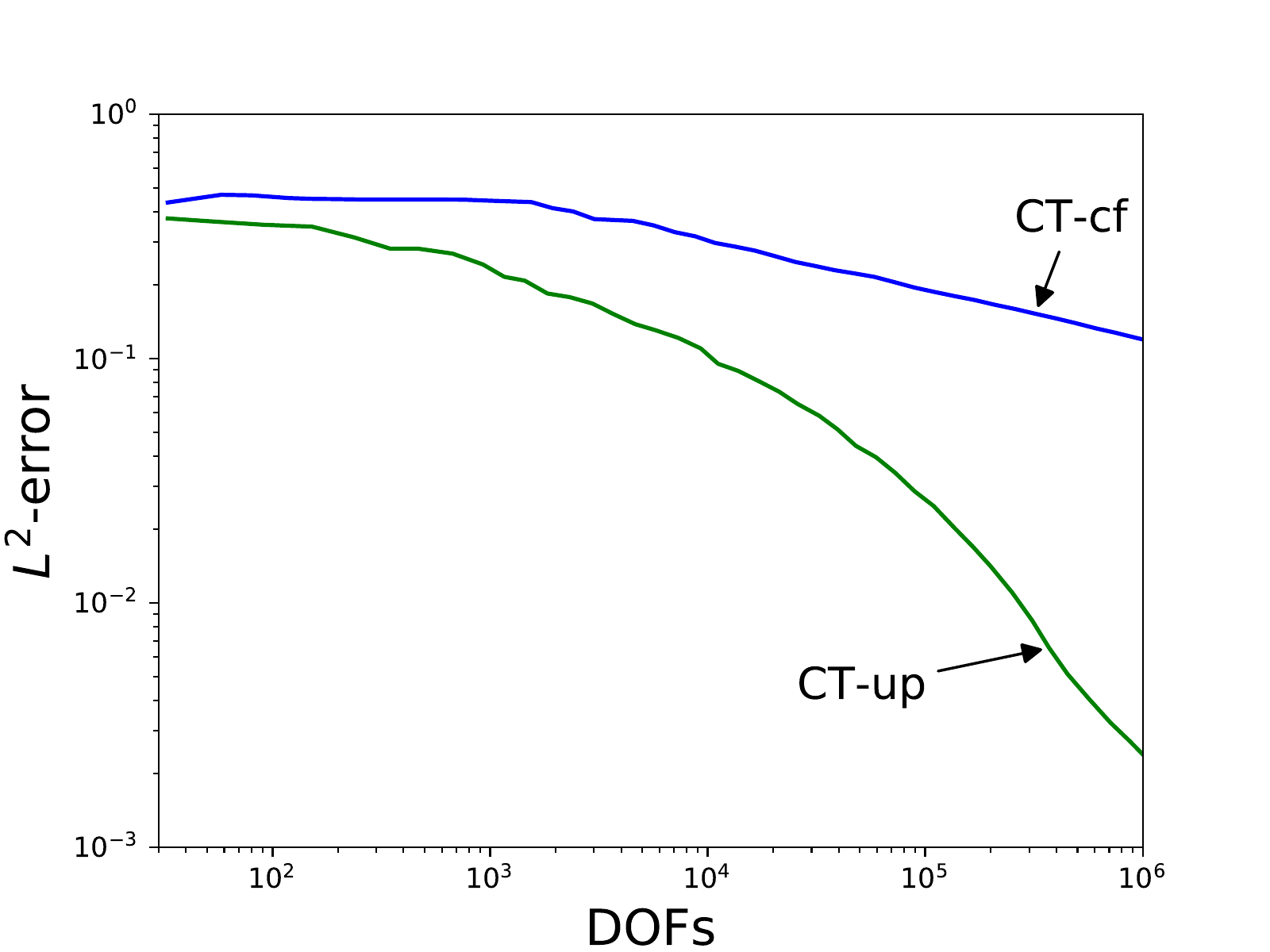}
        \caption{$L^2$ error vs.~DOFs; adaptive}
        \label{fig:ad_L2_p1_M500}
    \end{subfigure}
    %\hspace{0.4cm}
    \begin{subfigure}[b]{0.49\textwidth}
    \centering
        \includegraphics[width=\textwidth]{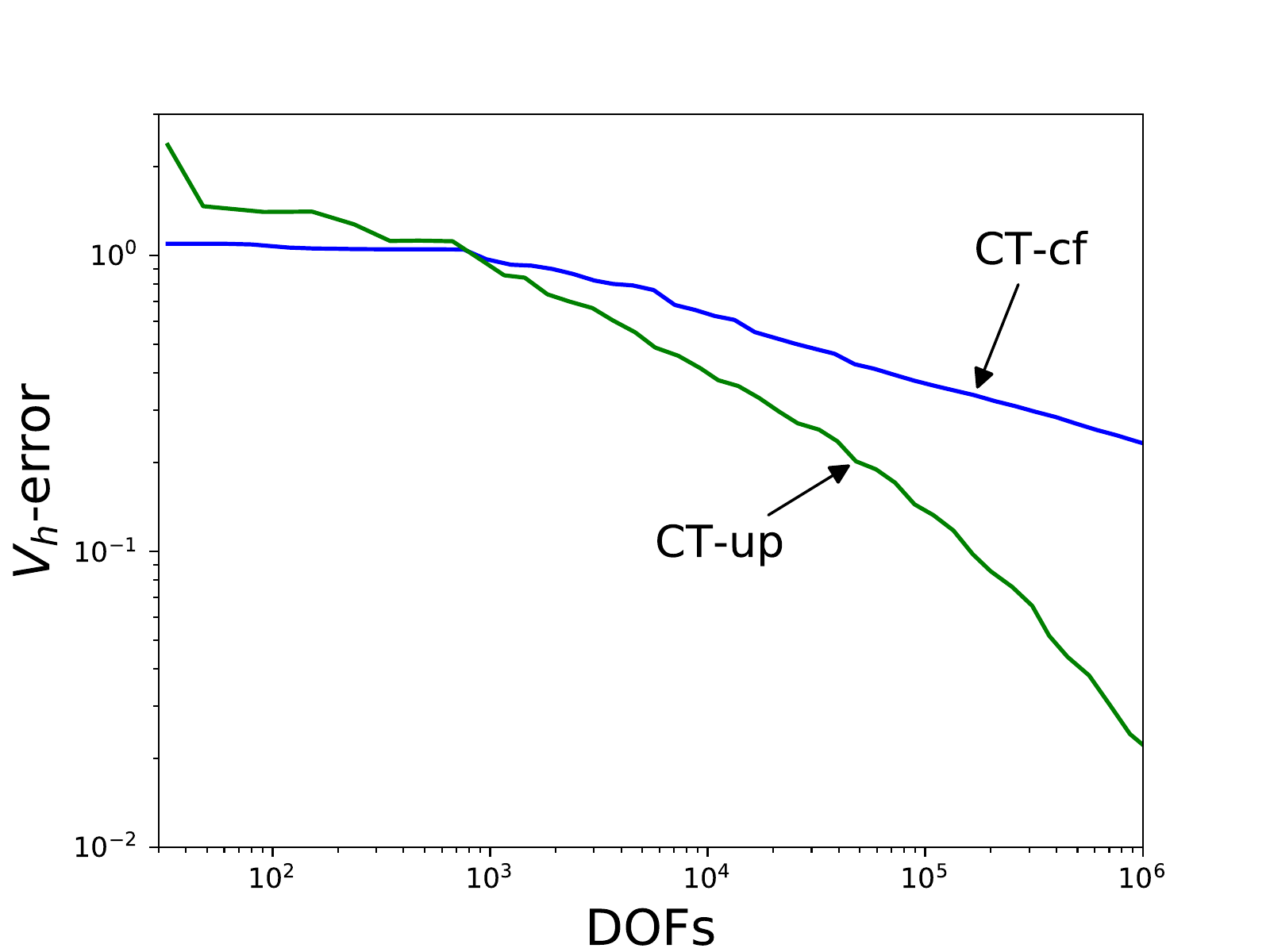}
        \caption{$V_h$ error vs.~DOFs; adaptive}
        \label{fig:ad_gr_p1_M500}
    \end{subfigure}
   \caption{2D model problem: errors in the $L^2$- and $\|\cdot\|_{V_h}$-norms vs.~DOFs for adaptive meshes, $p=1$, and $M=500$.}
 \end{figure}
\begin{figure}[t!]
    \centering
    \begin{subfigure}[b]{0.49\textwidth}
    \centering
        \includegraphics[width=\textwidth]{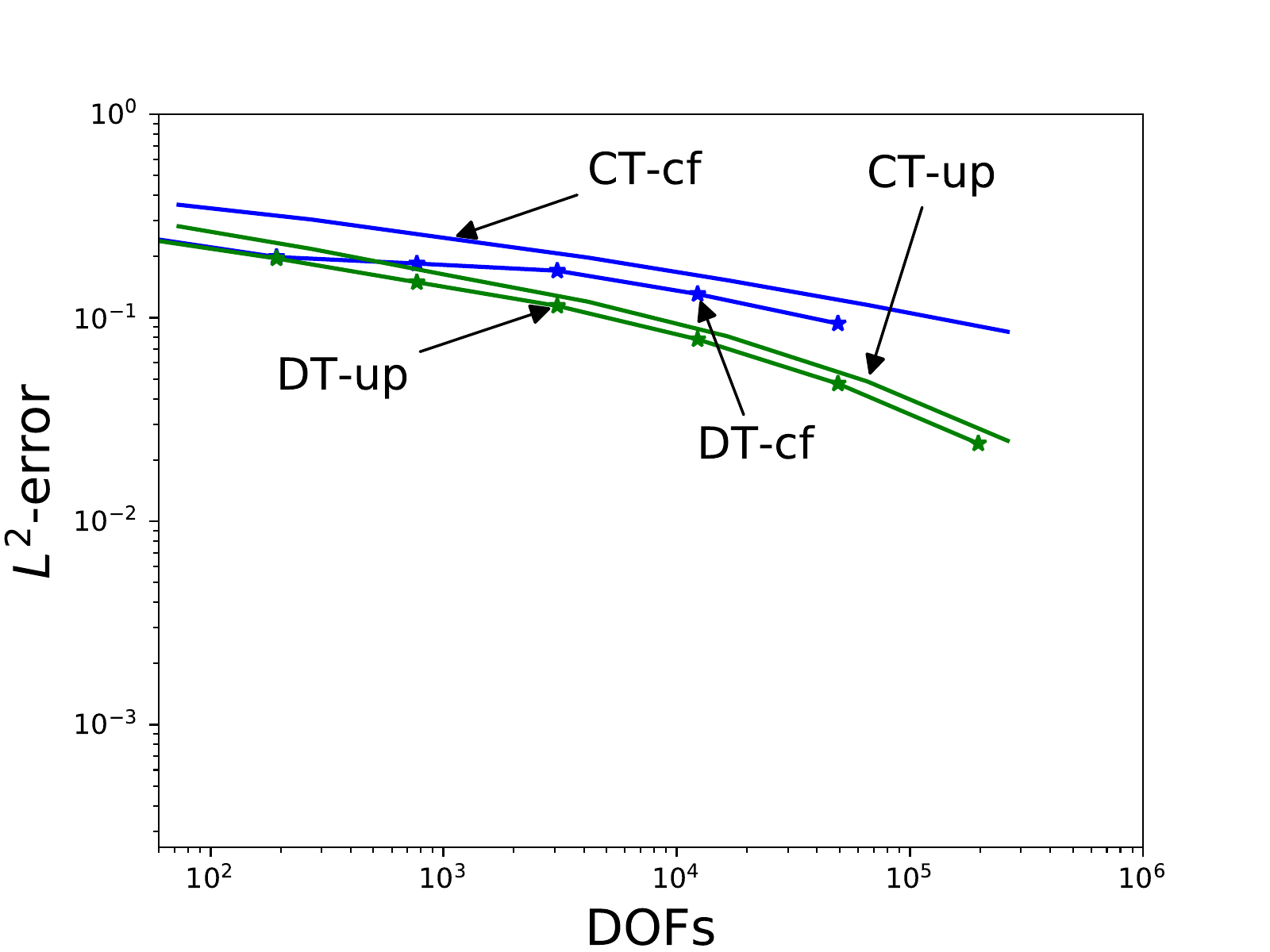}
        \caption{$L^2$ error vs.~DOFs; uniform}
        \label{fig:un_L2_p2_M500}
    \end{subfigure}
    %\hspace{0.4cm}
    \begin{subfigure}[b]{0.49\textwidth}
    \centering
        \includegraphics[width=\textwidth]{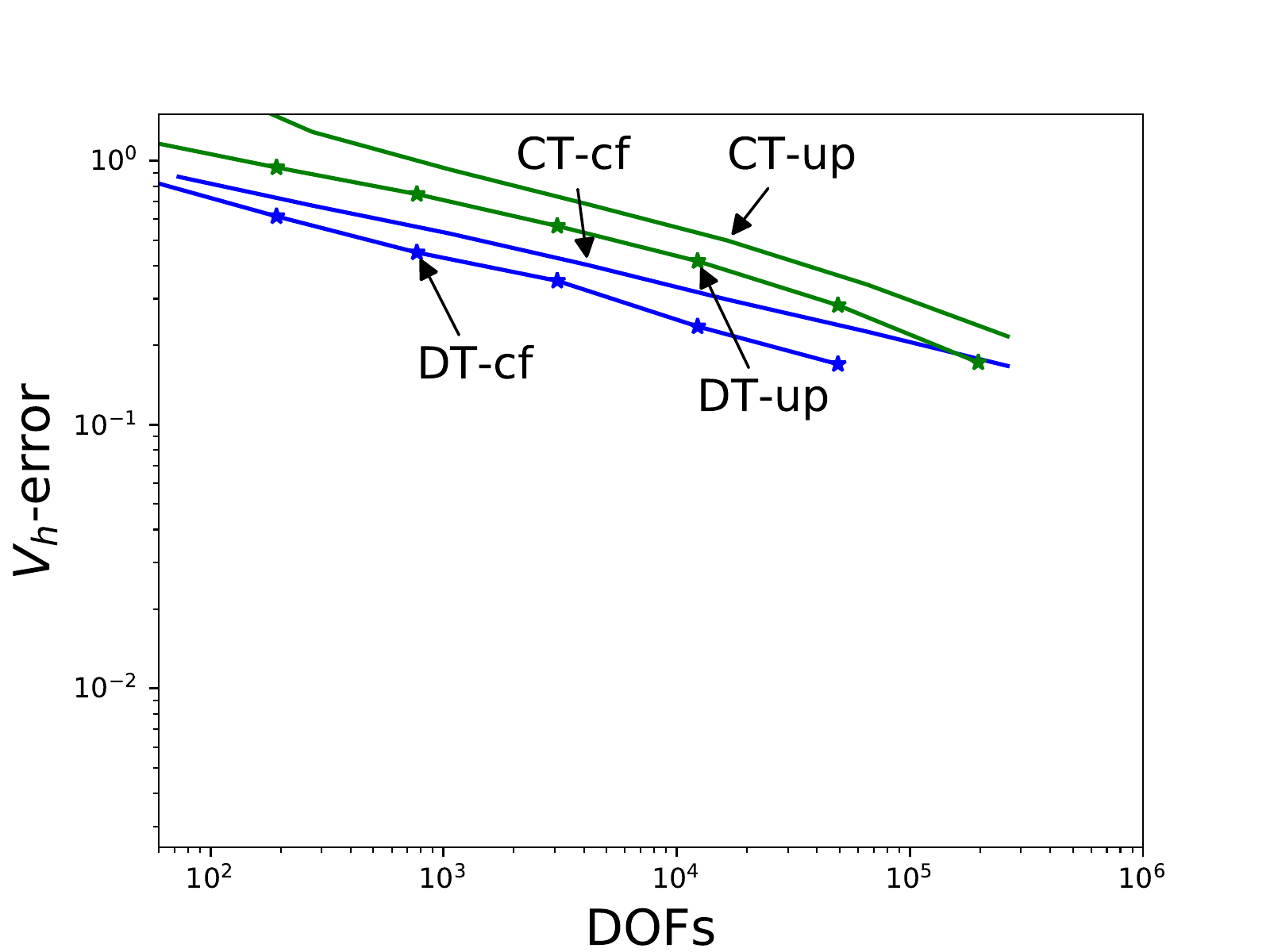}
        \caption{$V_h$ error vs.~DOFs; uniform}
        \label{fig:un_gr_p2_M500}
    \end{subfigure}
   \caption{2D model problem: errors in the $L^2$- and $\|\cdot\|_{V_h}$-norms vs.~DOFs for uniform meshes, $p=2$, and $M=500$.}
\end{figure}
\begin{figure}[h!]
    \centering
    \begin{subfigure}[b]{0.49\textwidth}
    \centering
        \includegraphics[width=\textwidth]{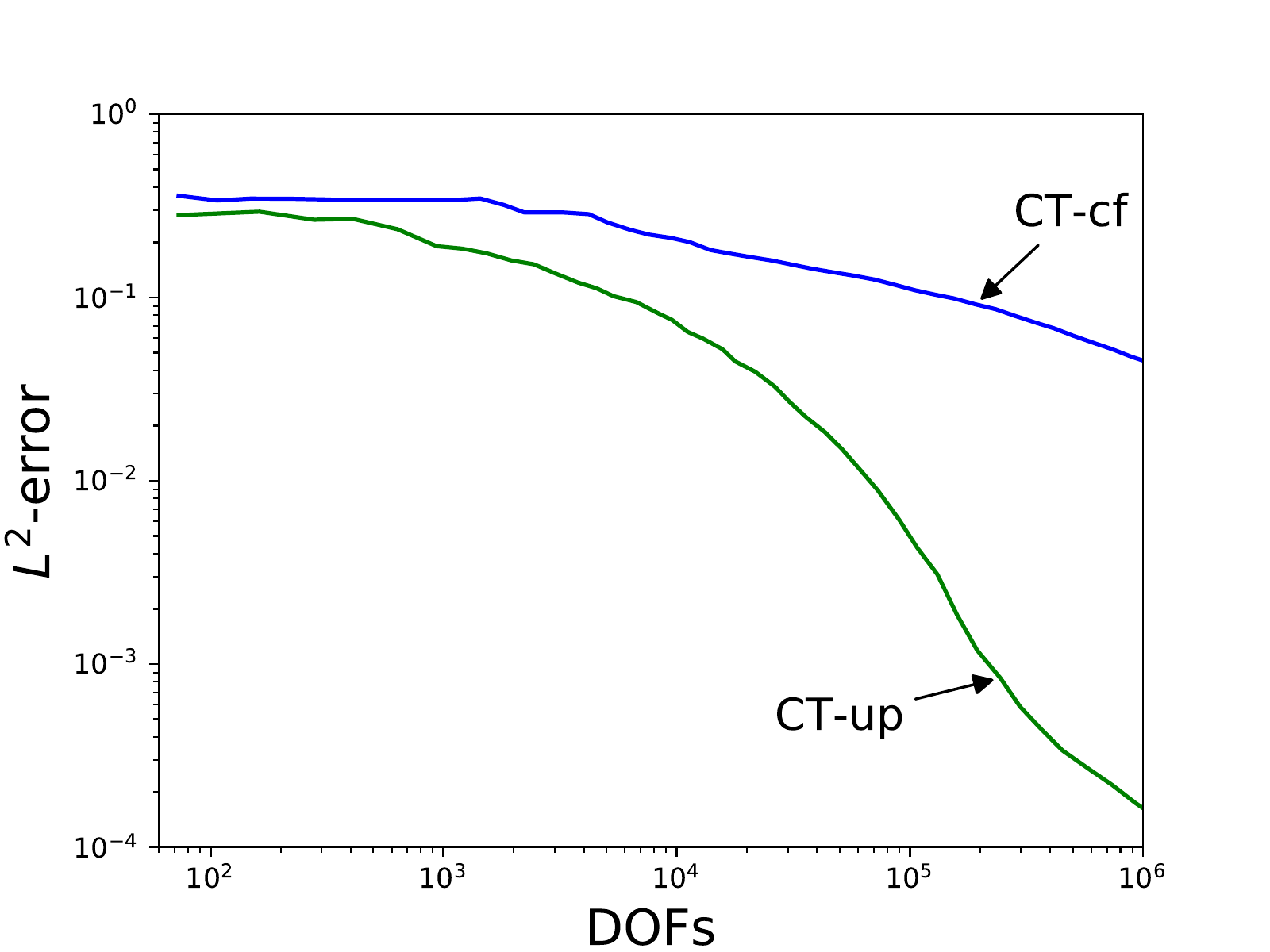}
        \caption{$L^2$ error vs.~DOFs; adaptive}
        \label{fig:ad_L2_p2_M500}
    \end{subfigure}
    %\hspace{0.4cm}
    \begin{subfigure}[b]{0.49\textwidth}
    \centering
        \includegraphics[width=\textwidth]{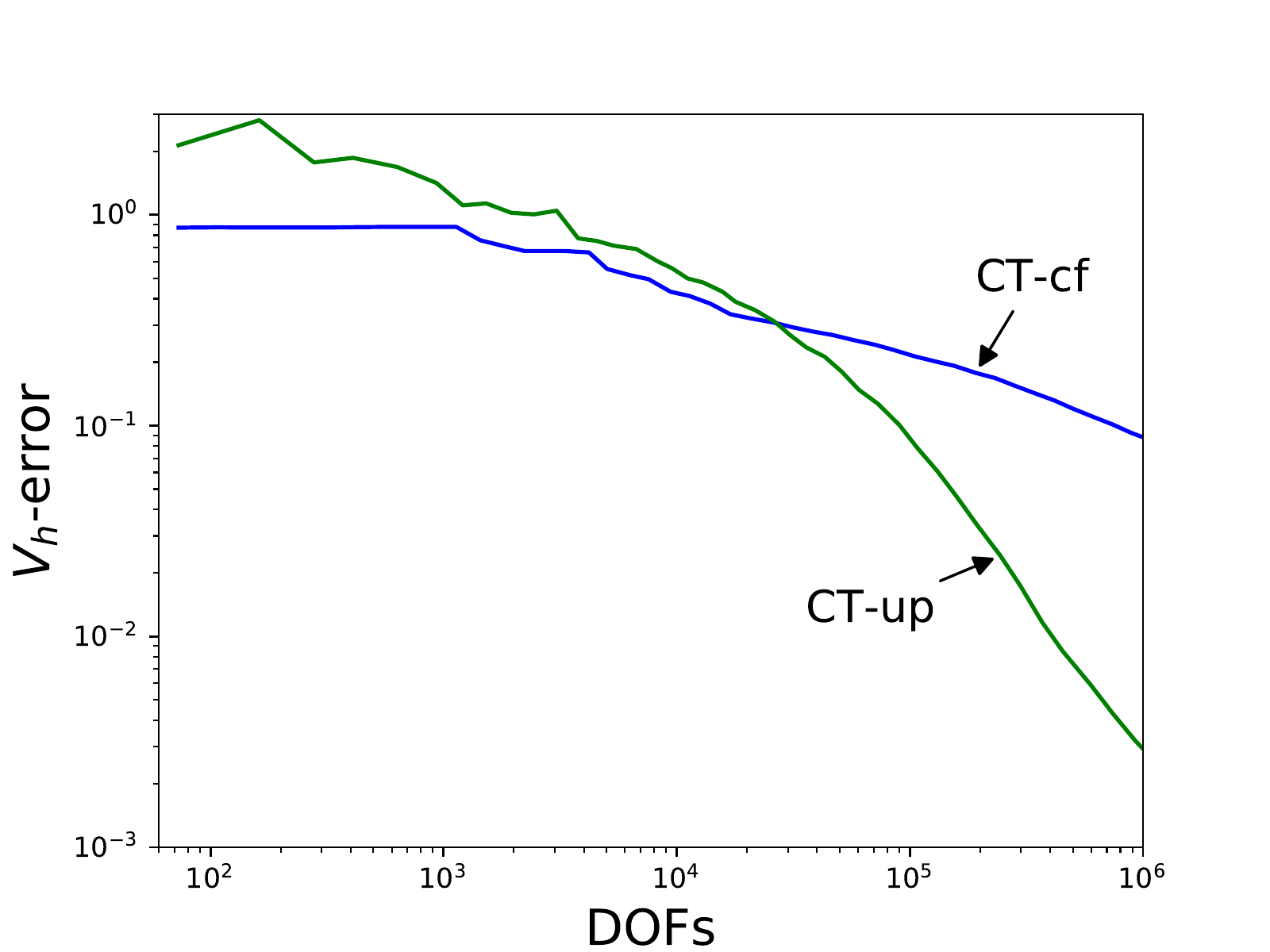}
        \caption{$V_h$ error vs.~DOFs; adaptive}
        \label{fig:ad_gr_p2_M500}
    \end{subfigure}
   \caption{2D model problem: errors in the $L^2$- and $\|\cdot\|_{V_h}$-norms vs.~DOFs for adaptive meshes, $p=2$, and $M=500$.}
\end{figure}
\begin{figure}[t!]
    \centering
    \begin{subfigure}[b]{0.49\textwidth}
    \centering
   % Uniform\\
        \includegraphics[width=\textwidth]{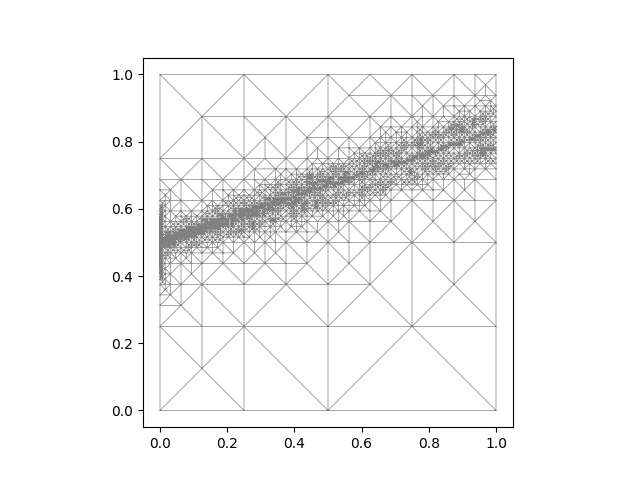}
        \caption{Adapted mesh}
        \label{fig:cfmesh}
    \end{subfigure}
    %
    %\hspace{0.4cm}
    \begin{subfigure}[b]{0.49\textwidth}
    \centering
    %Adaptive\\
        \includegraphics[width=\textwidth]{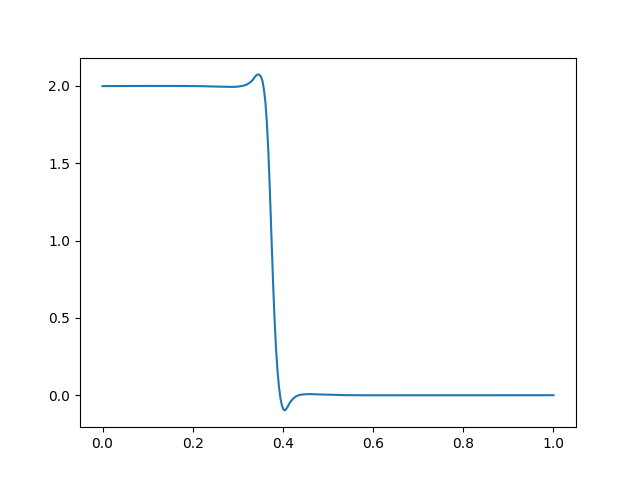}
        \caption{Solution cut over the line $x_1 = 1-x_2$}
        \label{fig:cfcut}
    \end{subfigure}
   \caption{2D model problem: adaptively refined mesh and transversal cut of the discrete solution for $p=2$, $M=500$ at level 35 of refinement; CT-cf, 191,546 DOFs.}
 \label{fig:cfmeshcut}
 \end{figure}
  \begin{figure}[h!]
    \centering
    \begin{subfigure}[b]{0.49\textwidth}
    \centering
  
        \includegraphics[width=\textwidth]{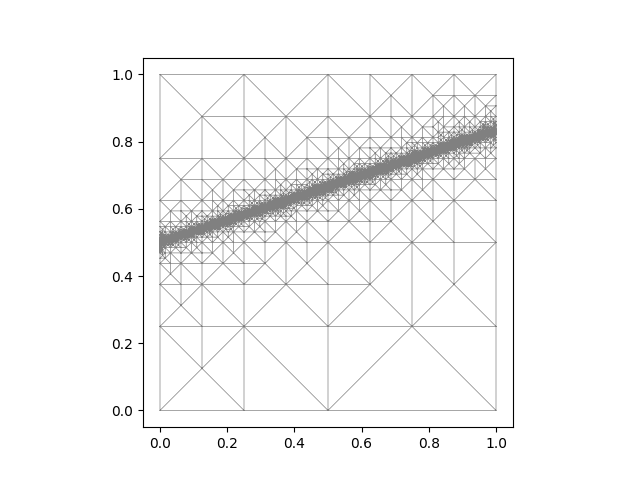}
        \caption{Adapted mesh}
        \label{fig:upmesh}
    \end{subfigure}
    %
   % \hspace{0.4cm}
    \begin{subfigure}[b]{0.49\textwidth}
    \centering

        \includegraphics[width=\textwidth]{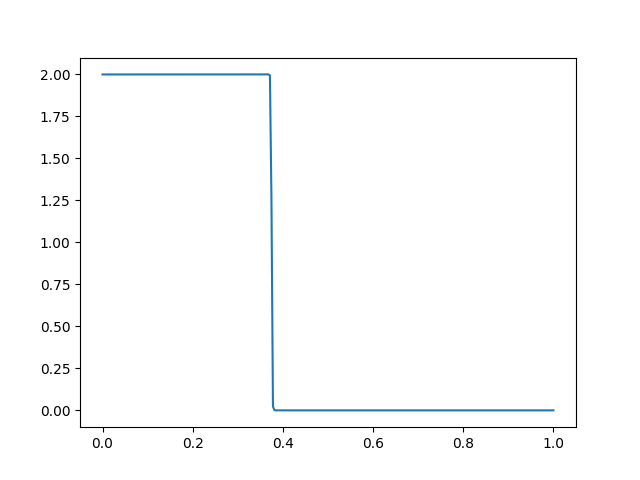}
        \caption{Solution cut over the line $x_1 = 1-x_2$}
        \label{fig:upcut}
    \end{subfigure}
   \caption{2D model problem: adaptively refined mesh and transversal cut of the discrete solution for $p=2$, $M=500$; CT-up, 159,255 DOFs.}
   \label{fig:upmeshcut}
 \end{figure}
\begin{figure}[h!]
    \centering
    \begin{subfigure}[b]{0.49\textwidth}
    \centering
       \includegraphics[width=\textwidth]{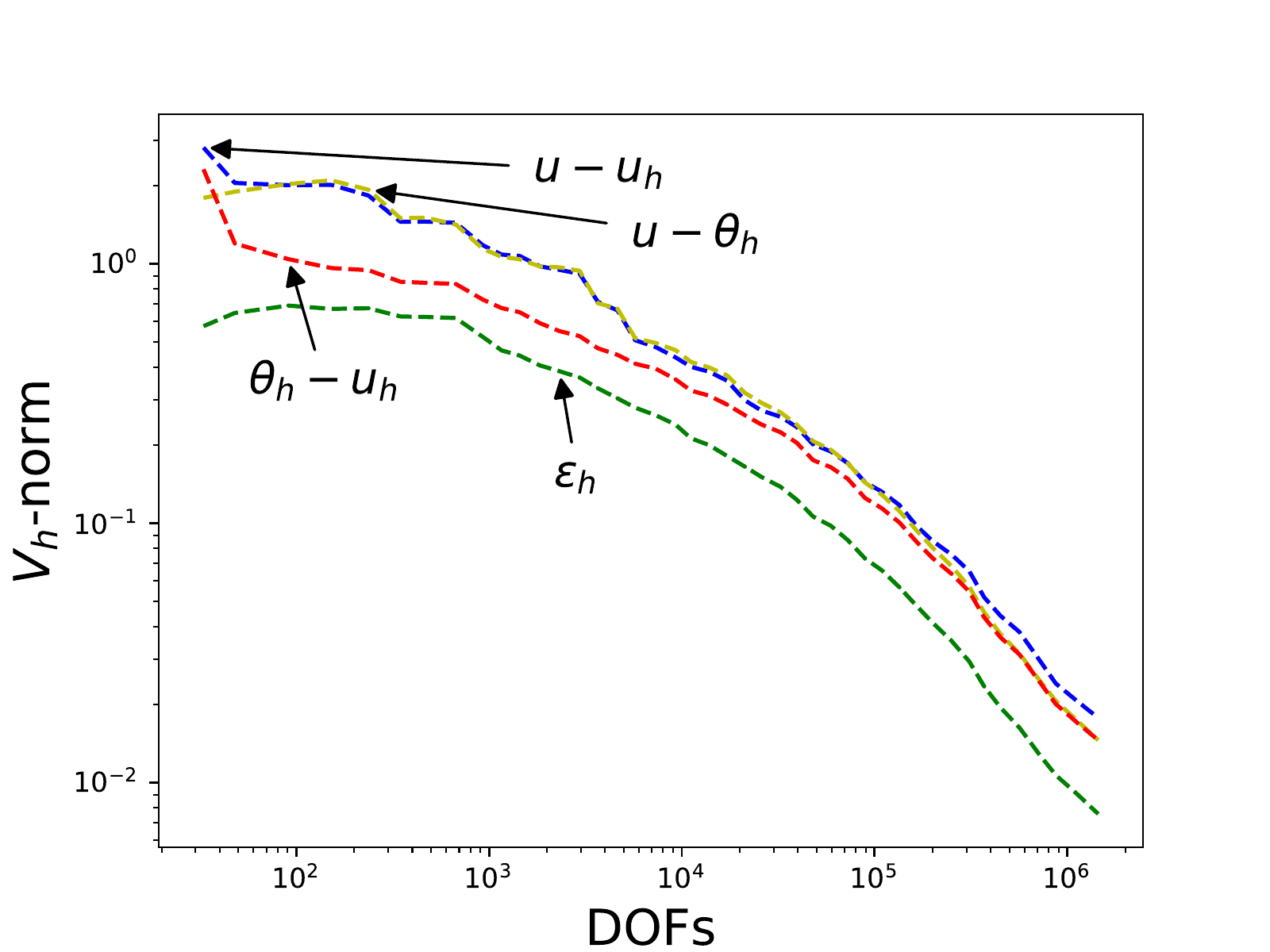}
        \caption{$p=1$}
        \label{fig:normp1}
    \end{subfigure}
    %
   % \hspace{0.4cm}
    \begin{subfigure}[b]{0.49\textwidth}
    \centering
 \includegraphics[width=\textwidth]{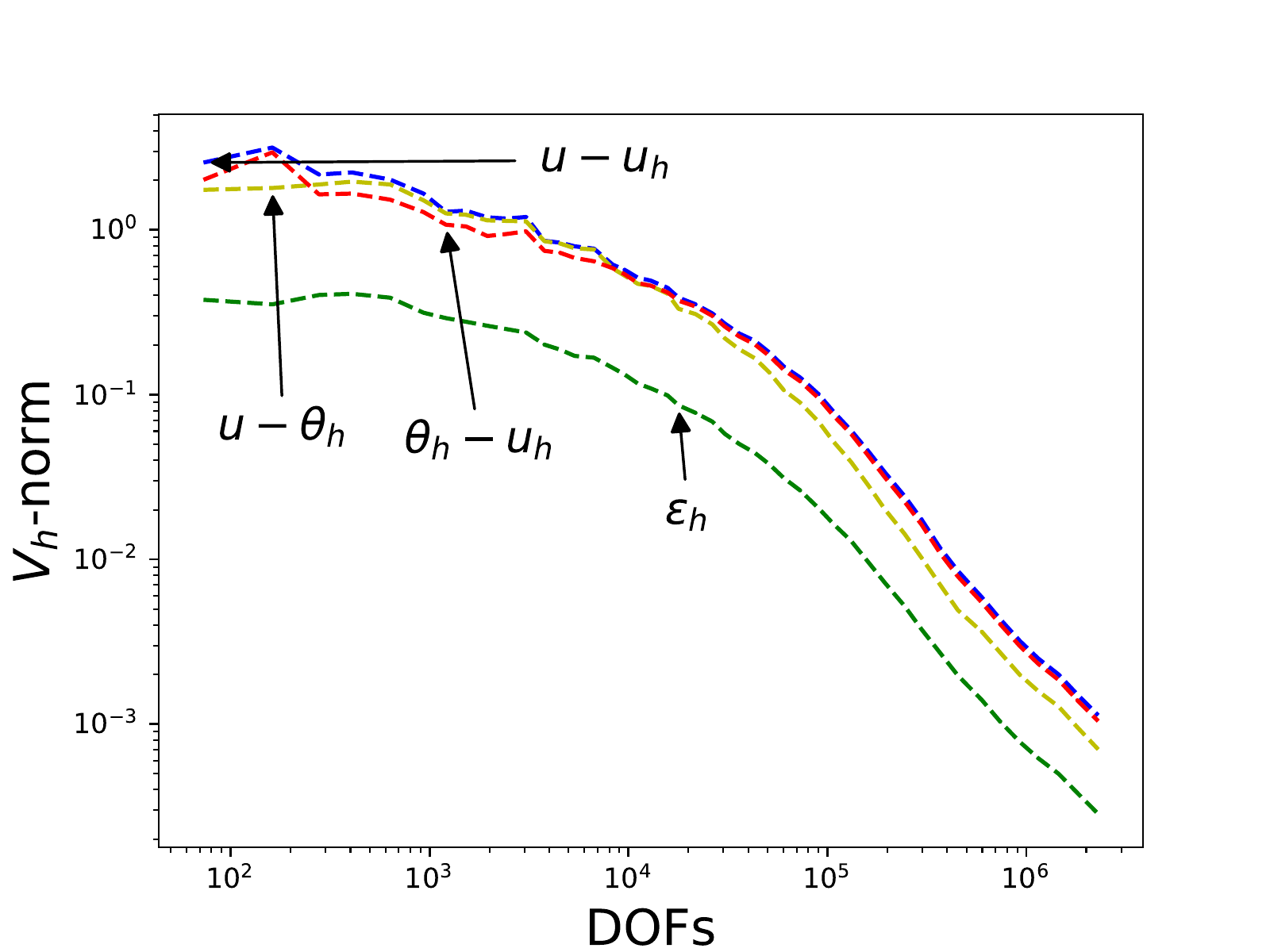}
        \caption{$p=2$}
        \label{fig:normp2}
    \end{subfigure}
   \caption{$V_h$-norm vs DOFs for adaptive meshes; CT-up, $p=1,2$ and $M=500$.}
   \label{fig:norm_adapt}
 \end{figure}
 \begin{figure}[t!]
 \vspace{0.4cm}
    \centering
    \begin{subfigure}[b]{0.49\textwidth} 
    \begin{picture}(100,100)
    \centering
    \put(0,0){\includegraphics[width=\textwidth]{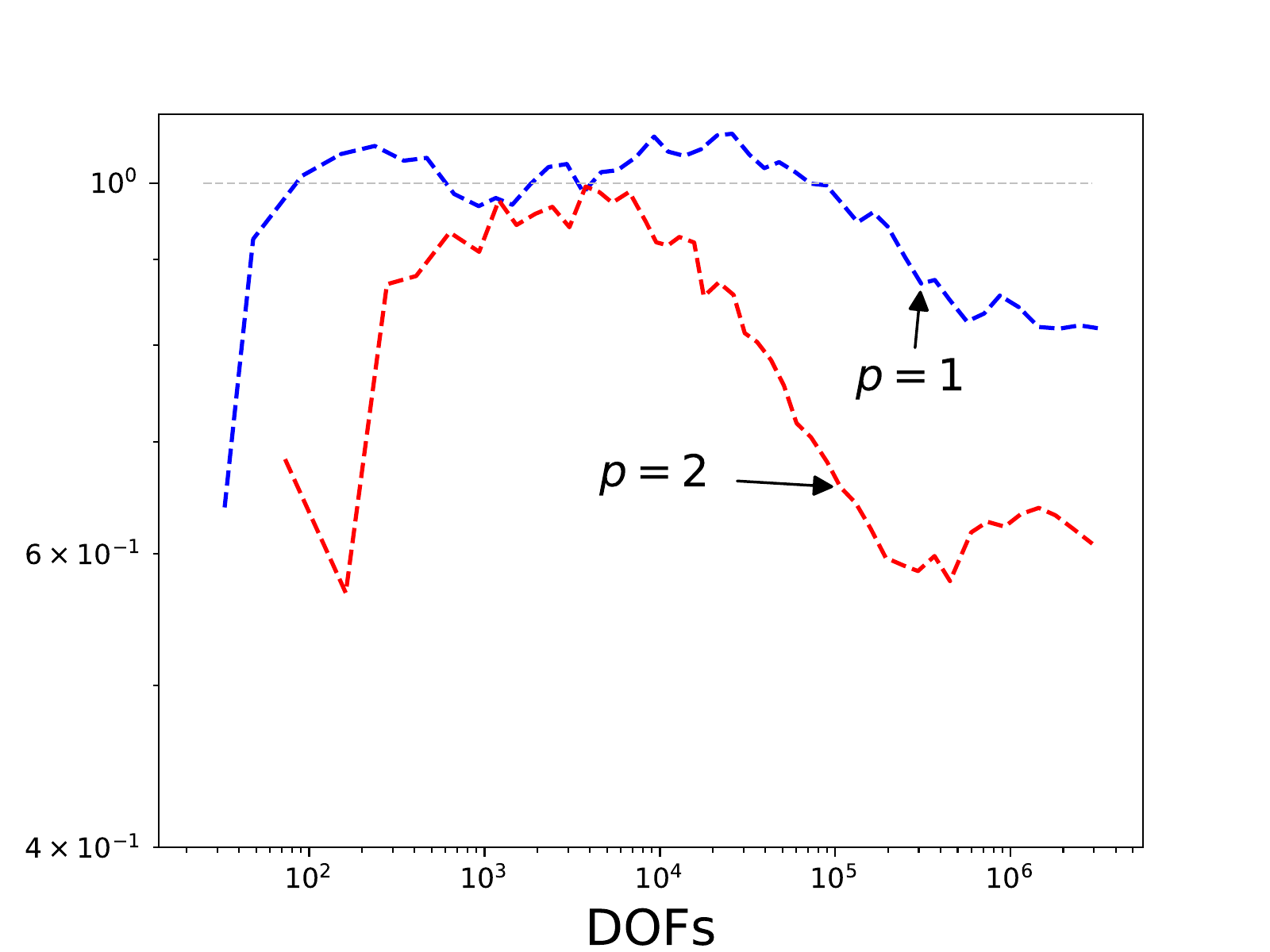}}
    \put(0,65){$\mathcal{S}$}
    \end{picture}
        \caption{Ratio $\mathcal{S}$}
        \label{fig:sat}
    \end{subfigure}
    \begin{subfigure}[b]{0.49\textwidth}
    \begin{picture}(100,100)
    \centering
    \put(0,0){\includegraphics[width=\textwidth]{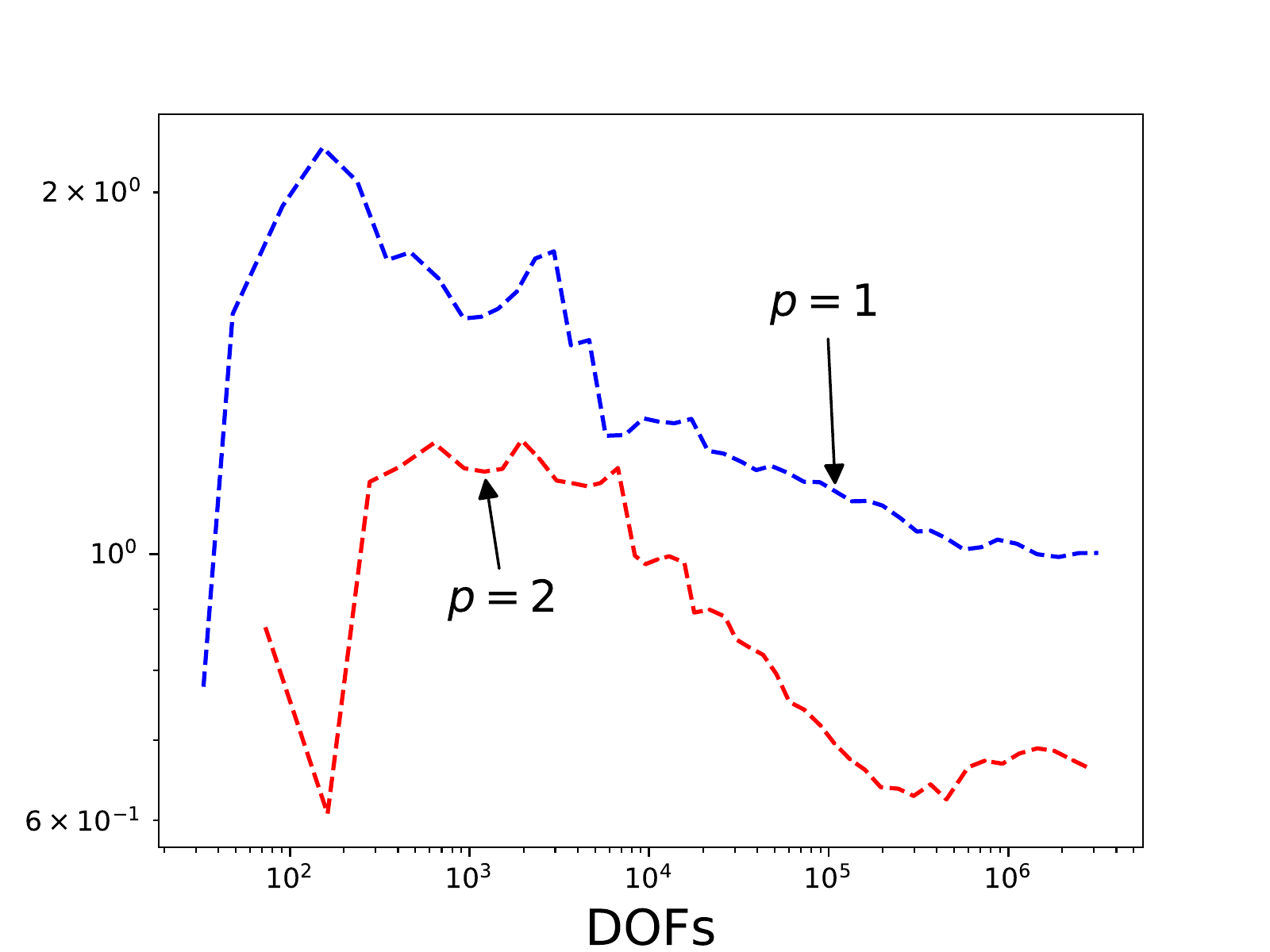}}
    \put(0,65){$\mathcal{W}$}
    \end{picture}
        \caption{Ratio $\mathcal{W}$}
        \label{fig:sat_weak}
    \end{subfigure}
        \caption{Ratios $\mathcal{S}$ and $\mathcal{W}$ defined in~\eqref{eq:ratio} vs DOFs for adaptive meshes; CT-up, $p=1,2$ and $M=500$}
        \label{fig:sat_all}
 \end{figure}  

%Moreover, Figure~\ref{fig:sat_weak} shows that there is indeed an upper bound for the ratio $\mathcal{W}$, indicating that the weaker Assumption~\ref{as:weak} is satisfied.}
%
\begin{figure}[t!]
  \centering
  \begin{subfigure}[b]{0.48\textwidth}
    \centering
    \includegraphics[width=\textwidth]{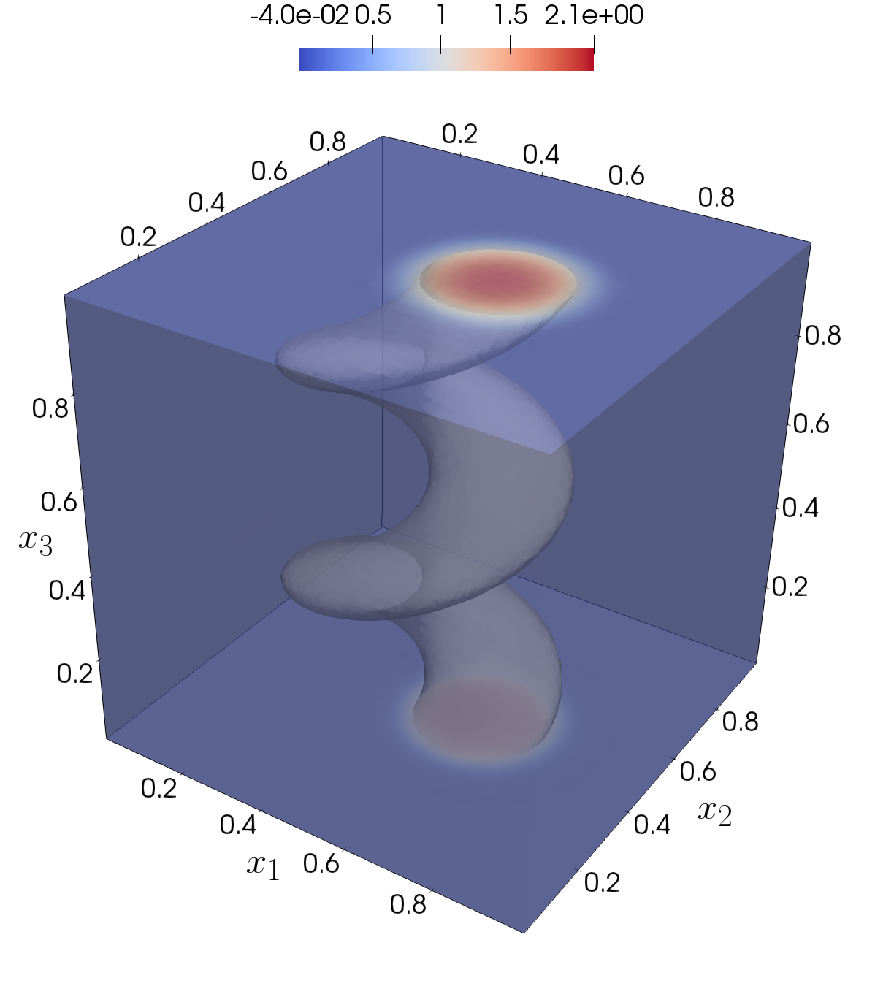}
    \caption{Contour representation of the solution}
    \label{fig:3D_sol}
  \end{subfigure}
    %\hspace{0.4cm}
  \begin{subfigure}[b]{0.48\textwidth}
    \centering
    \includegraphics[width=\textwidth]{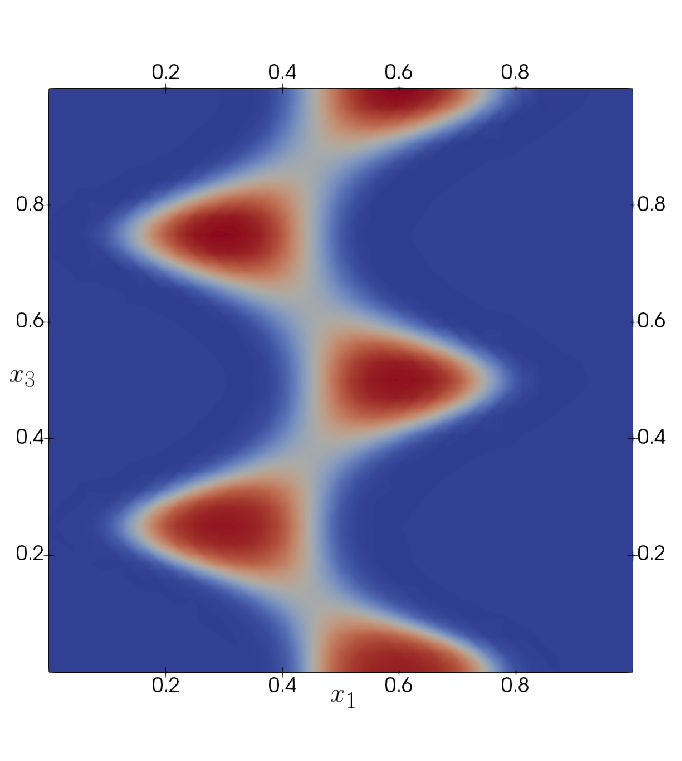}
    \caption{Solution cut over the plane $\{x_2 = 0.5\}$}
    \label{fig:3D_sol_cut_y}
  \end{subfigure}
  \vspace{0.4cm}
  \begin{subfigure}[b]{0.48\textwidth}
    \centering
    \includegraphics[width=\textwidth]{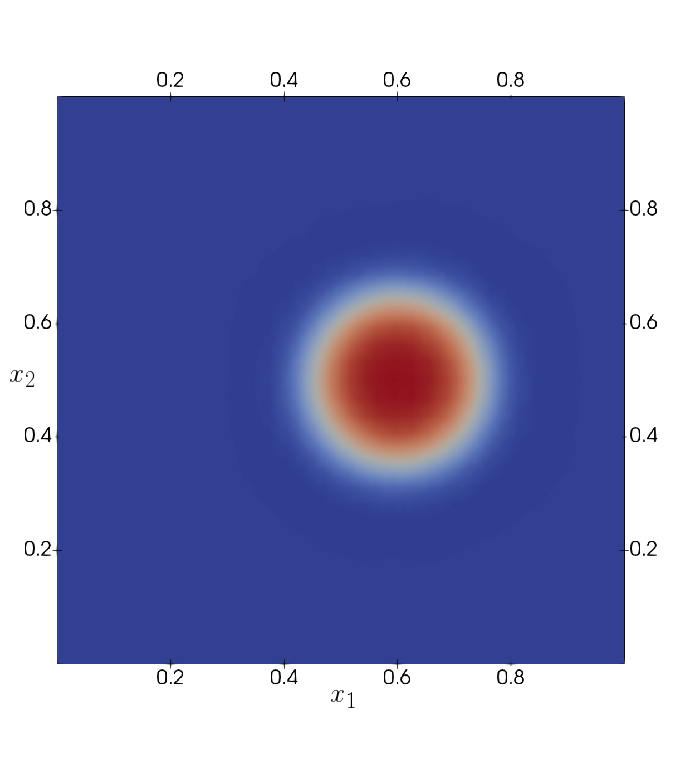}
    \caption{Solution cut over the plane $\{x_3 = 0\}$}
    \label{fig:3D_sol_cut_z0}
  \end{subfigure}
  % \hspace{0.4cm}
  \begin{subfigure}[b]{0.48\textwidth}
    \centering
    \includegraphics[width=\textwidth]{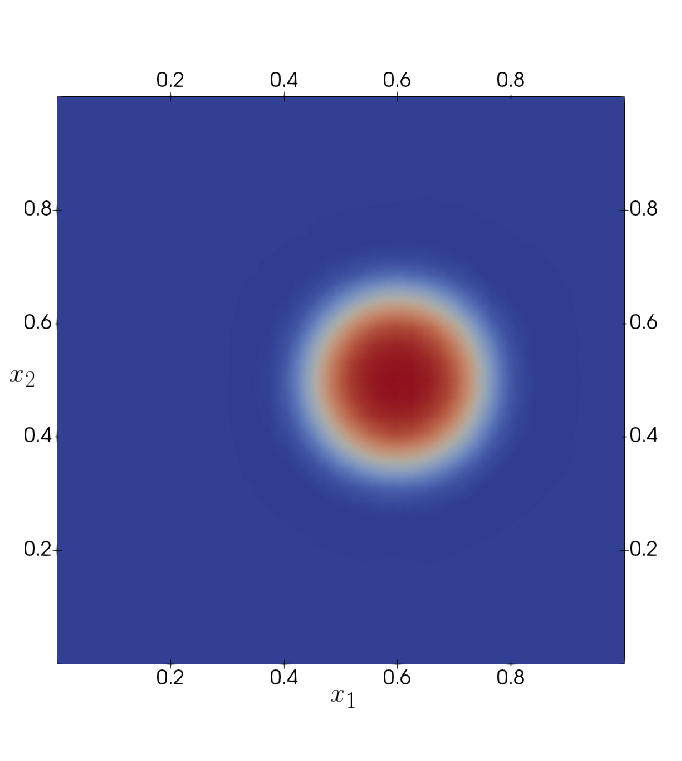}
    \caption{Solution cut over the plane $\{x_3 = 1\}$}
    \label{fig:3D_sol_cut_z1}
  \end{subfigure}
  \caption{3D model problem: solution contours over the whole domain and cuts over three planes at the level 31 of refinement; $p=1$. 159,705 DOFs for the trial space, and 3,569,333 total DOFs}
\end{figure}

\begin{figure}[t!]
    \centering
    \begin{subfigure}[b]{0.48\textwidth}
    \centering
        \includegraphics[width=\textwidth]{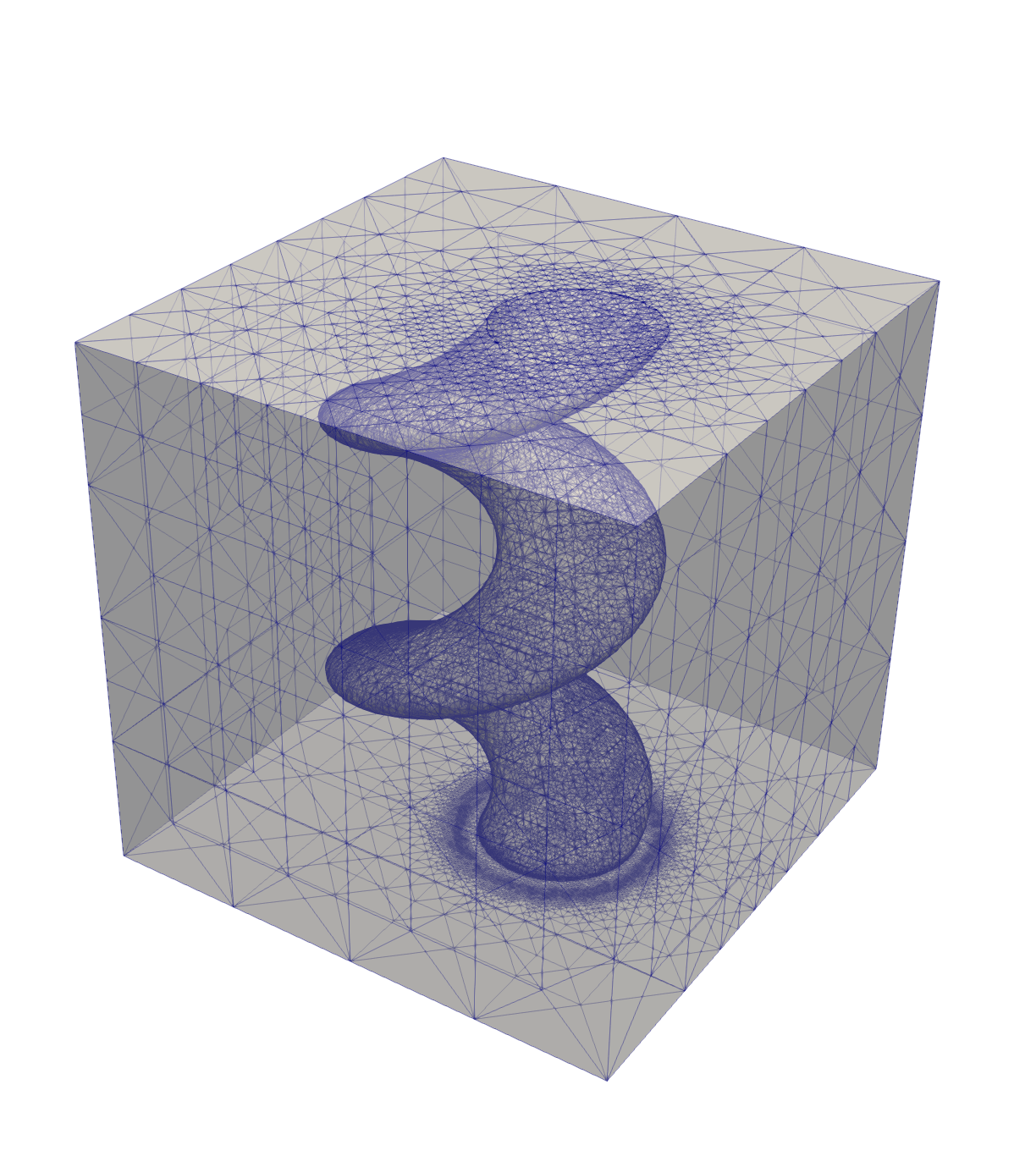}
        \caption{Solution contour lines with the mesh}
        \label{fig:mesh_spiral}
    \end{subfigure}
    %\hspace{0.4cm}
    \begin{subfigure}[b]{0.48\textwidth}
    \centering
        \includegraphics[width=\textwidth]{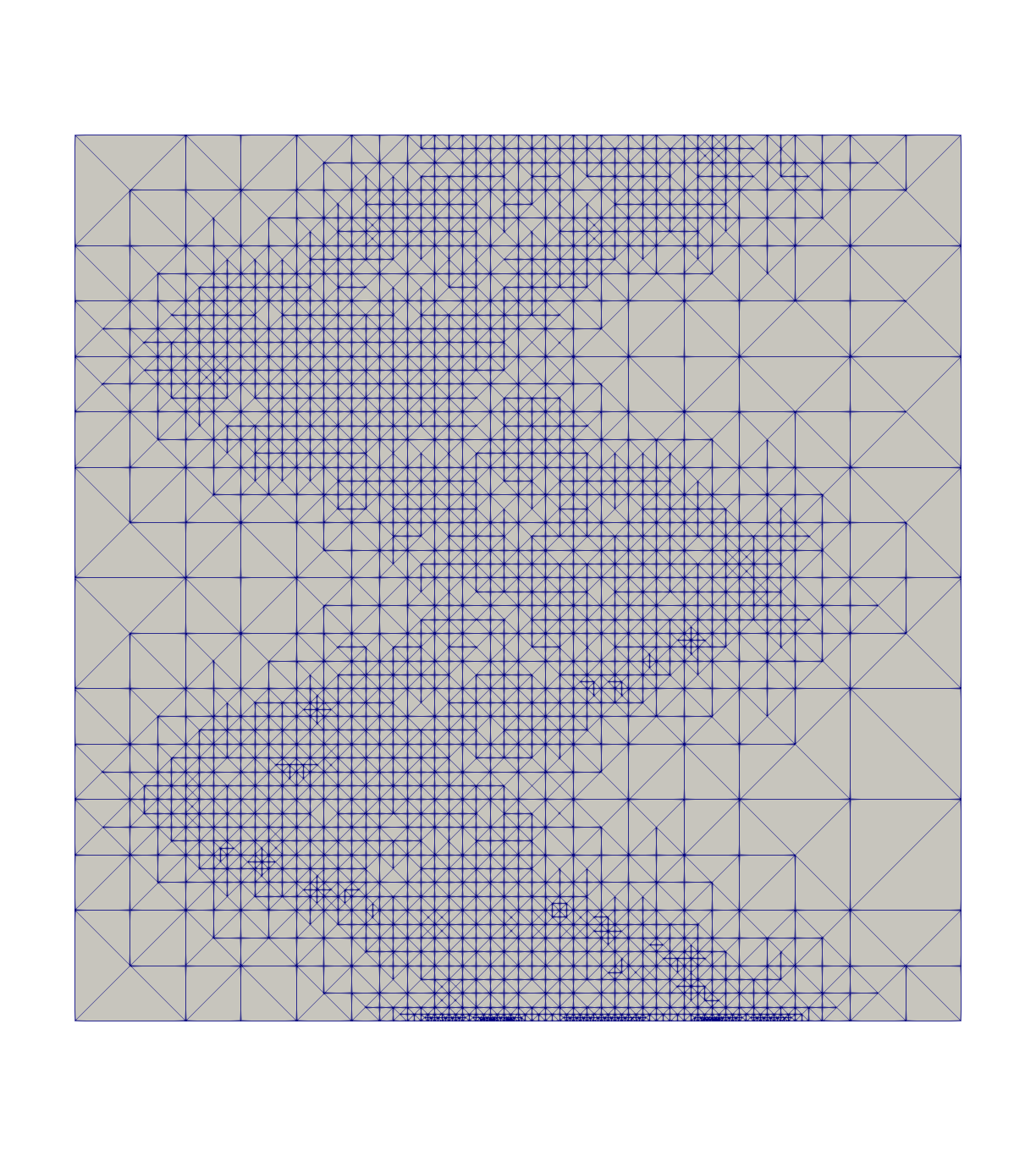}
        \caption{Mesh cut over the plane $\{x_2 = 0.5\}$}
        \label{fig:mesh_cut_y}
    \end{subfigure}
    \vspace{0.4cm}
    \begin{subfigure}[b]{0.48\textwidth}
    \centering
        \includegraphics[width=\textwidth]{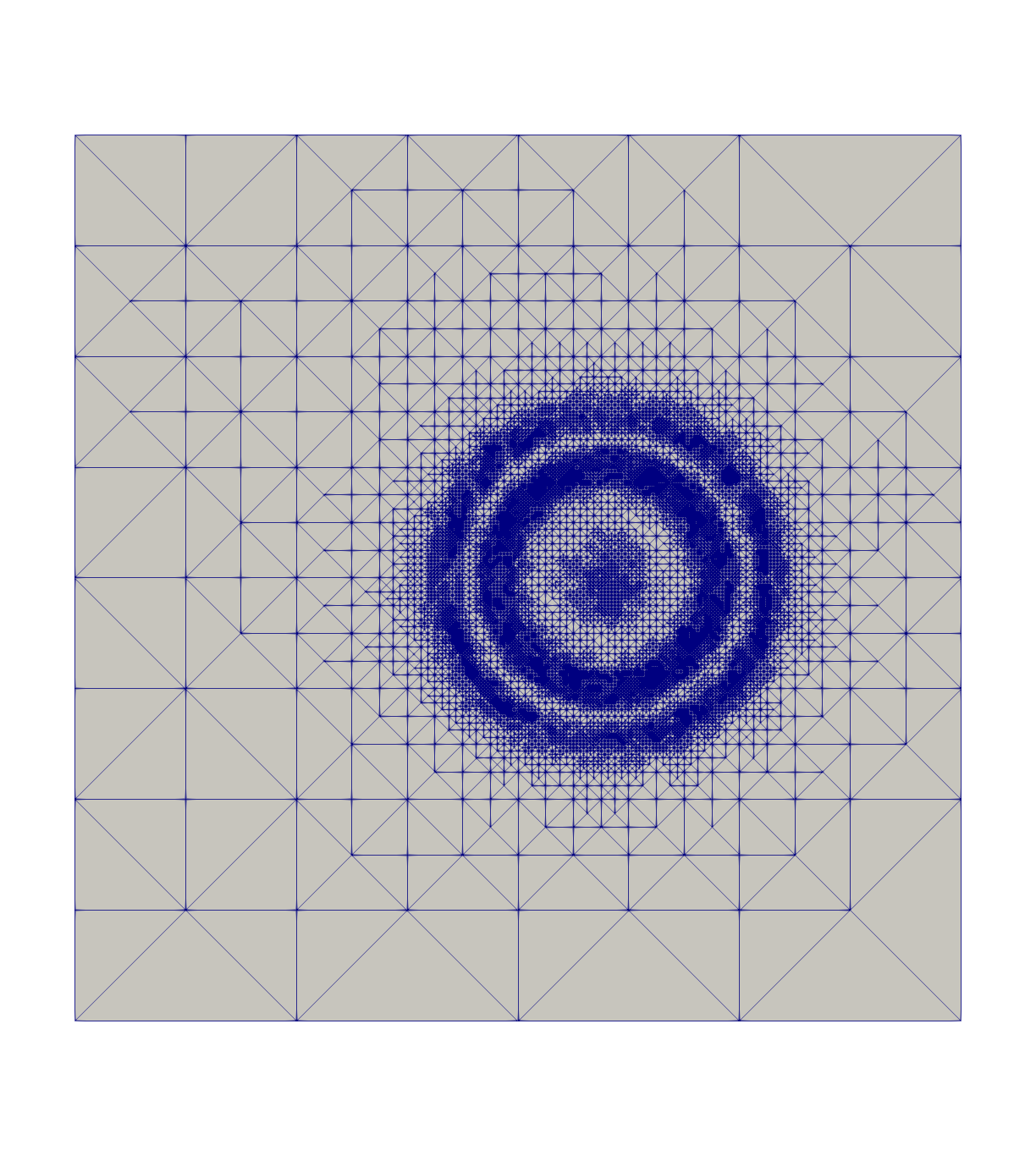}
        \caption{Mesh cut over the plane $\{x_3 = 0\}$}
        \label{fig:mesh_cut_z0}
    \end{subfigure}
    %\hspace{0.4cm}
    \begin{subfigure}[b]{0.48\textwidth}
    \centering
        \includegraphics[width=\textwidth]{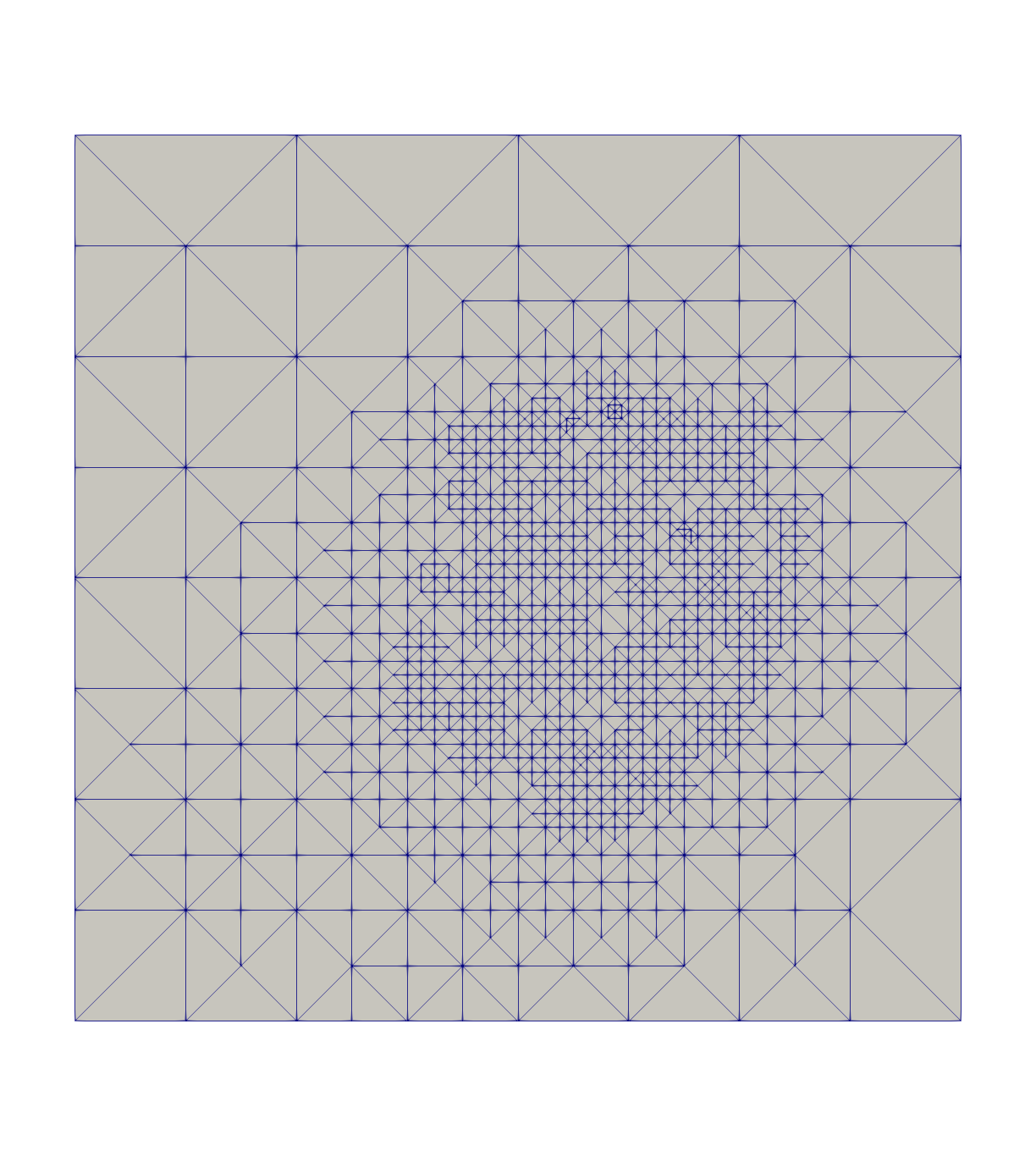}
        \caption{Mesh cut over the plane $\{x_3 = 1\}$}
        \label{fig:mesh_cut_z1}
    \end{subfigure}
   \caption{3D model problem: mesh over the whole domain and cuts over three planes at level 31 of refinement, $p=1$, 159,705 DOFs for the trial space, and 3,569,333 total DOFs}
\end{figure}

\subsubsection{Discussion of the 3D results}

In this section, we explore the performance of the proposed adaptive method with the $\|\cdot\|_{\textrm{up}}$-norm for the 3D model problem. We consider the value $M = 100$ in the exact solution~\eqref{eq:ana_3D} so that $u_M$ is close to the discontinuous limit when $M = \infty$. The method delivers accurate solutions by refining where it is most needed, as Figures~\ref{fig:3D_sol}-\ref{fig:3D_sol_cut_z1}, and~\ref{fig:mesh_spiral}-\ref{fig:mesh_cut_z1} \zzz{exemplify. These figures} show a 3D representation of the contour and cuts over the planes $\{x_2=0.5\}$, $\{x_3=0\}$ (part of inflow) and $\{x_3=1\}$ (part of outflow) for the solution and for the mesh obtained after 31 levels of adaptive refinement. The mesh refinement at the plane $\{x_3 = 0\}$ located at the inflow boundary is finer compared with that at the plane $\{x_3=1\}$ located at the outflow boundary. Finally, Figures~\ref{fig:spiral_l2}-\ref{fig:spiral_Vh} show the convergence in the $L^2$- and $V_h$-norms vs the total number of DOFs for the polynomial degrees $p=1,2$.
%
%%% 3D solutions %%%%% 
\begin{figure}[t!]
    \centering
    \begin{subfigure}[b]{0.49\textwidth}
        \includegraphics[width=\textwidth]{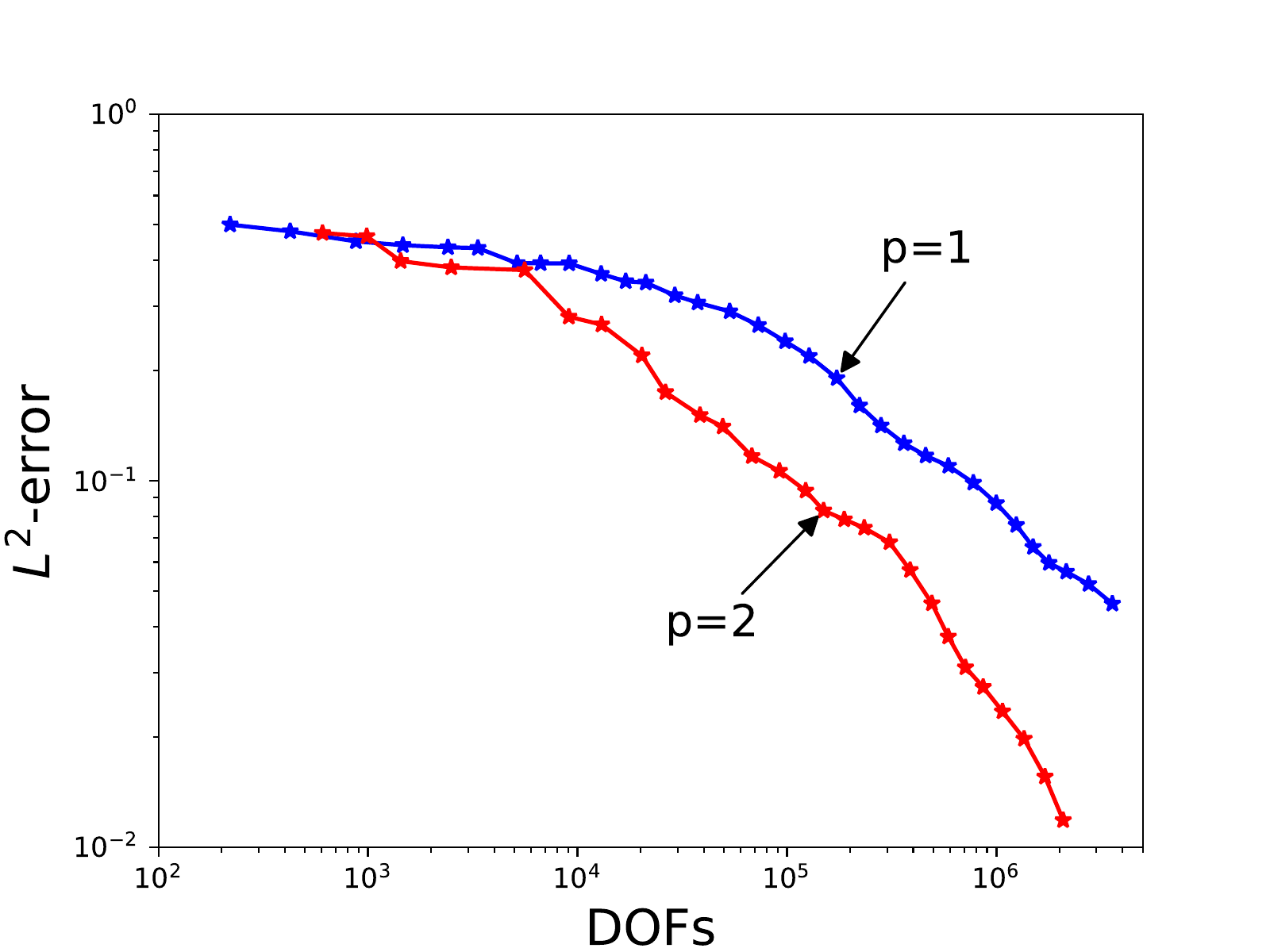}
        \caption{$L^2$ error vs.~DOFs}
        \label{fig:spiral_l2}
    \end{subfigure}
    %
    %\hspace{0.4cm}
    \begin{subfigure}[b]{0.49\textwidth}
        \includegraphics[width=\textwidth]{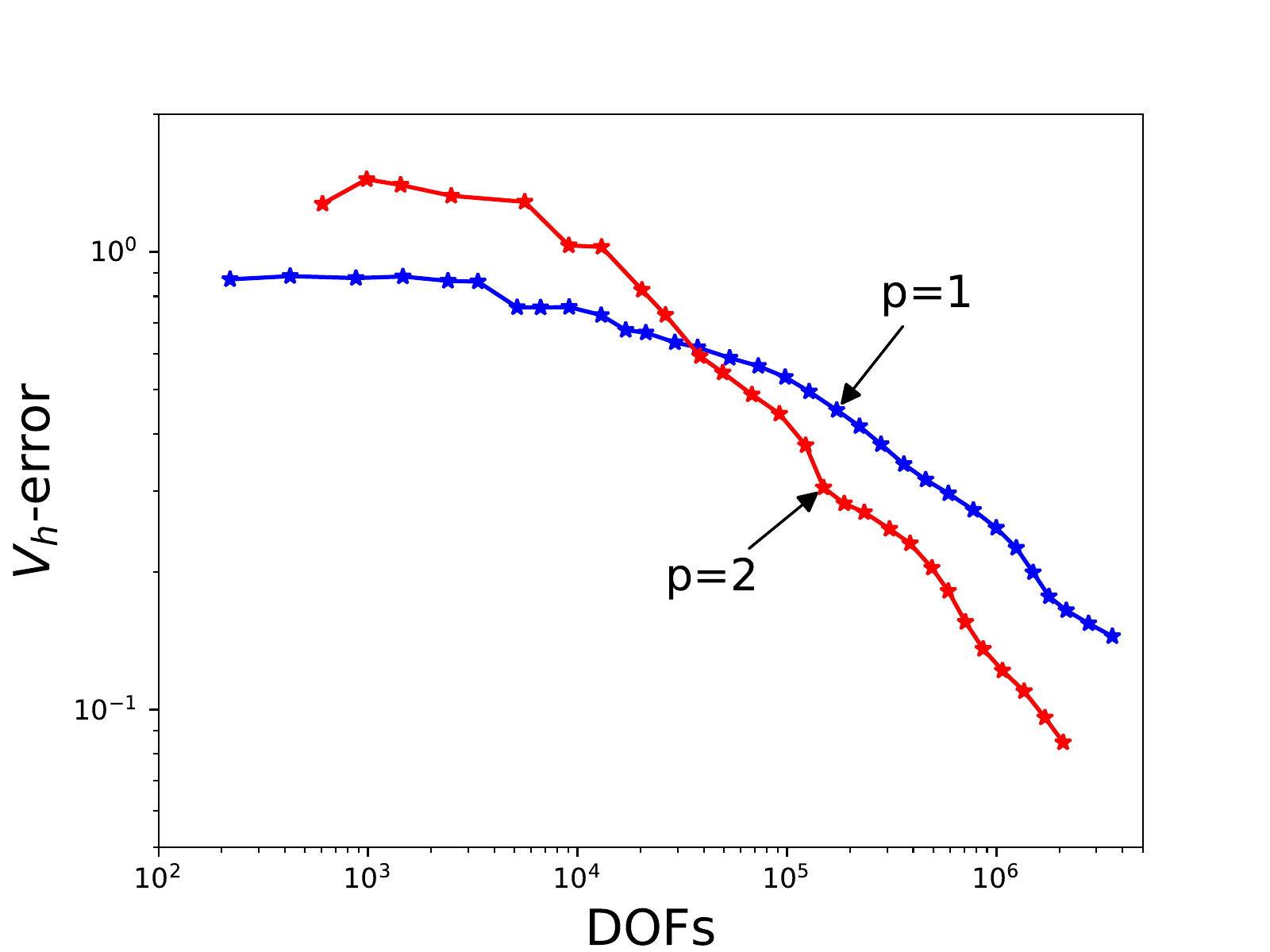}
        \caption{$V_h$ error vs.~DOFs}
        \label{fig:spiral_Vh}
    \end{subfigure}
   \caption{3D model problem: $L^2$-error and $V_h$-error vs.~DOFs for the 3D spiral problem with adaptivity mesh refinement; up norm. $p =1,2$}
   \label{fig:spiral_errors}
 \end{figure}

% As a third example, we consider a~piece-wise continuous polynomial space for the trial space and $\Delta_p = 1$(See Sec.~\ref{sec:disc_trial}).
% In Figure~\ref{fig:dg_plots} we show the $L^2$ and $V_h$-norm error plots for the DG schemes under the same assumptions of the numerical examples section, with $\Delta_p = 1$, and using $\varepsilon_h$ as error estimator for adaptivity. As expected, convergence slopes are clearly performed when using $\varepsilon_h$ as estimator. A slight performance is observed when considering the "up+" norm when comparing with the "up-" norm, but a huge difference is appreciated when comparing both with respect to the "cf" norm. This behaviour was expected since the convergence using the central fluxes scheme is suboptimal.

% \begin{figure}[h!]
%     \centering
%     \begin{subfigure}[b]{0.45\textwidth}
%     \centering
%         \includegraphics[width=\textwidth]{dgstrongl2.pdf}
%         \caption{$L^2$-norm error vs dofs}
%         \label{fig:dg_l2_M500}
%     \end{subfigure}
%     %
%     \hspace{0.4cm}
%     \begin{subfigure}[b]{0.45\textwidth}
%     \centering

%         \includegraphics[width=\textwidth]{dgstronggraph.pdf}
%         \caption{$V_h$-norm error vs dofs}
%         \label{fig:dg_gr_M500}
%     \end{subfigure}
%    \caption{Error plots in log-scale using adaptivity with $M=500$, $\Delta_p=1$ and $p=1,2$.}
%    \label{fig:dg_plots}
% \end{figure}

\section{Conclusions} \label{sec:conc}

In this paper, we propose a new stabilized finite element method based on residual minimization. The key idea is to obtain a residual representation using a dual \zzz{norm defined} over a discontinuous Galerkin space \zzz{that} delivers inf-sup stability. The cost is that one needs to solve a stable, saddle-point problem. The advantage is that one recovers at the same time a stabilized finite element solution and a residual \zzz{representative defined in the discontinuous Galerkin space. This residual representative can drive adaptive mesh refinement process}. Our numerical results on 2D and 3D advective model problems \zzz{with} sharp inner layers indicate that the \zzz{method} leads to competitive error \zzz{rates} on uniformly refined meshes \zzz{compared against discontinuous Galerkin approximations. Additionally, the method delivers} adaptive meshes that sharply capture \zzz{the} inner layers. Further studies are on the way to assess the computational performance of \zzz{our} method, especially in 3D, and on other model problems comprising, for example, systems of first-order PDEs of Friedrichs's type, as in Darcy's equations or in Maxwell's equations.

\appendix
\section{Proof of Theorem~\ref{th:FEMwDG}.}\label{ap:main_proof}
Recall the definition of the discrete operator $B_h:U_h\to V_h^*$ from~\eqref{eq:B_h} (here, the domain of $B_h$ is restricted to $U_h$). Let $B_hU_h\subset V_h^*$ be the range of $B_h$, and $(B_hU_h)^\perp\subset V_h$ be such that
$$
(B_hU_h)^\perp:=\{v_h\in V_h: b_h(z_h,v_h)=0, \forall z_h\in U_h\},
$$  
where $(B_hU_h)^\perp=\ker B_h^*$ with $B_h^*:V_h\to U_h^*$.

We prove the well-posedness of~\eqref{eq:mix_form} by establishing the equivalence between~\eqref{eq:mix_form} and the following constrained minimization problem: 
\begin{equation}\label{eq:cons_min}
\inf_{v_h\in (B_hU_h)^\perp}\left\{{1\over 2}\|v_h\|^2_{V_h}-\left<l_h,v_h\right>_{V_h^*,V_h}\right\}=:\inf_{v_h\in (B_hU_h)^\perp}F(v_h),
\end{equation}
which has a unique solution since the functional $F$ is strictly convex and $(B_hU_h)^\perp$ is a closed subspace of $V_h$. By differentiating with respect to $v_h$, we observe that the minimizer $\tilde v_h\in (B_hU_h)^\perp $ of~\eqref{eq:cons_min} must be a critical point of $F$, that is,
\begin{equation}\label{eq:opt}
(\tilde v_h,v_h)_{V_h} - \left<l_h,v_h\right>_{V_h^*,V_h}=0, \qquad \forall\, v_h\in (B_hU_h)^\perp.
\end{equation}
Any component $\varepsilon_h$ of a solution $(\varepsilon_h,u_h)\in V_h\times U_h$ of the saddle-point problem~\eqref{eq:mix_form} satisfies~\eqref{eq:opt}. Conversely, let $\varepsilon_h=\tilde v_h\in (B_hU_h)^\perp$ be the unique solution of~\eqref{eq:opt}. Then, 
$l_h-R_{V_h}\varepsilon_h$ is in ${(B_hU_h)^\perp}^\perp=B_hU_h$. Hence, there must be $u_h\in U_h$ such that $B_hu_h=l_h-R_{V_h}\varepsilon_h$. In other words, $(\varepsilon_h,u_h)\in V_h\times U_h$ solves the saddle-point problem~\eqref{eq:mix_form}. Finally, the uniqueness of the solution to~\eqref{eq:mix_form} readily follows from the injectivity of $B_h$ (which is a consequence of the inf-sup condition~\eqref{eq:infsup_h}) and the bijectivity of the Riesz isomorphism. 
 
Let us now prove the a~priori bounds~\eqref{eq:bounds} for $\varepsilon_h$ and $u_h$. Testing the equation~\eqref{eq:mix_forma} with $v_h=\varepsilon_h$, we infer that
$$
\|\varepsilon_h\|_{V_h}^2=\left<l_h,\varepsilon_h\right>_{V_h^*,V_h}\leq \|l_h\|_{V_h^*}\|\varepsilon_h\|_{V_h},
$$
and the first a~priori bound follows by dividing the above expression by $\|\varepsilon_h\|_{V_h}$.
For the second a~priori bound in~\eqref{eq:bounds}, we have
\begin{alignat*}{2}\tag{by~\eqref{eq:infsup_h} and~\eqref{eq:B_h}}
\|u_h\|_{V_h}\leq & {1\over C_{\rm sta}}\sup_{0\neq v_h\in V_h}{b_h(u_h,v_h)\over \|v_h\|_{V_h}} = {1\over C_{\rm sta}}\sup_{0\neq v_h\in V_h}{\big(R^{-1}_{V_h}B_hu_h,v_h\big)_{V_h}\over \|v_h\|_{V_h}}\\\tag{since $v_h =R^{-1}_{V_h}B_hu_h$ is the supremizer}
= & {1\over C_{\rm sta}}{\big(R^{-1}_{V_h}B_hu_h,R^{-1}_{V_h}B_h u_h\big)_{V_h}\over \|R^{-1}_{V_h}B_h u_h\|_{V_h}}\\\tag{since $b_h(u_h,\varepsilon_h)=0$}
= & {1\over C_{\rm sta}}\left[{
\big(\varepsilon_h+R^{-1}_{V_h}B_h u_h,R^{-1}_{V_h}B_h u_h\big)_{V_h}\over \|R^{-1}_{V_h}B_h u_h\|_{V_h}}\right]\\
= &{1\over C_{\rm sta}} {\big<l_h,R^{-1}_{V_h}B_h u_h\big>_{V_h^*,V_h}\over\|R^{-1}_{V_h}B_h u_h\|_{V_h}}\le {1\over C_{\rm sta}} \|l_h\|_{V_h^*}.\tag{by~\eqref{eq:mix_forma}}
\end{alignat*}

Finally, we prove the a~priori error estimate~\eqref{eq:apriori}. For any $z\in V_{h,\#}$, we define the projector $P_h:V_{h,\#}\to U_h$ by $P_h z=z_h$, where $z_h\in U_h$ is the second component of the solution $(\varepsilon_h,z_h)$ of the saddle-point problem~\eqref{eq:mix_form} with right-hand side $l_h(v_h)=b_h(z,v_h)$ (i.e., meaningful owing to Assumption~\ref{as:reg-consi}). Using the a~priori bound in~\eqref{eq:bounds} and the bound~\eqref{eq:continuity} in Assumption~\ref{as:bound} leads to
\begin{equation}\label{eq:proy}
\|P_hz\|_{V_h}=\|z_h\|_{V_h}\leq{1\over C_{\rm sta}}\|b_h(z,\cdot)\|_{V_h^*}\leq {C_{\rm bnd}\over C_{\rm sta}}\|z\|_{V_{h,\#}}\,\,.
\end{equation}
Besides, $P_hz_h=z_h$ for any $z_h\in U_h$. Indeed, in that case, $(0,z_h)$ solves the corresponding saddle-point problem~\eqref{eq:mix_form}. To conclude, we observe that for the exact solution $u\in X_{\#}$ and the discrete solution $u_h\in U_h$, we have
\begin{alignat*}{2}
\|u-u_h\|_{V_h} = & 
\|(u-z_h)-P_h(u-z_h)\|_{V_h}\tag{by definition of $P_h$, $\forall z_h\in U_h$}\\
\leq & \|u-z_h\|_{V_h}+{C_{\rm bnd}\over C_{\rm sta}}\|u-z_h\|_{V_{h,\#}}\tag{by the triangle inequality and~\eqref{eq:proy}}\\
\leq & \Big(1+{C_{\rm bnd}\over C_{\rm sta}}\Big)\|u-z_h\|_{V_{h,\#}}.\tag{by Assumption~\ref{as:bound}}
\end{alignat*}
The result follows by taking the infimum over $z_h\in U_h$.
\section{Best-approximation in the upwinding norm}\label{ap:best_proof}
\begin{figure}[t!]
    \centering
    \begin{subfigure}[b]{0.32\textwidth}
        \includegraphics[width=\textwidth]{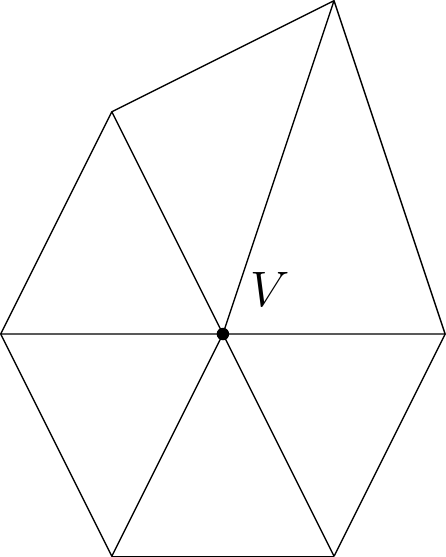}
        \caption{Set $D_V$}
        \label{fig:Dv}
    \end{subfigure}
    \hspace{0.6cm}
    \begin{subfigure}[b]{0.4\textwidth}
        \includegraphics[width=\textwidth]{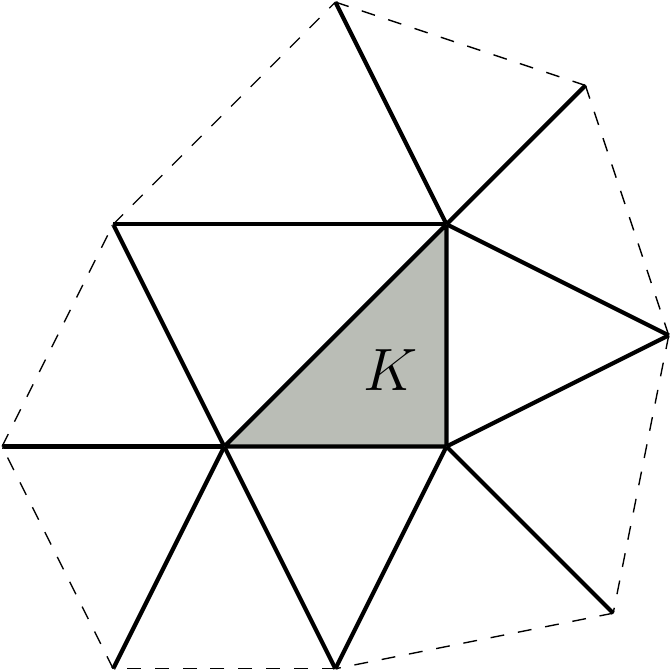}
        \caption{Set $\Sk_K$ (thick lines)}
        \label{fig:SK}
    \end{subfigure}
   \caption{Sets considered for the averaging operator and asspciated with a given node $V$ (here a mesh vertex) and an element $K\in \Par_h$}
   \label{fig:oswald_setting}
 \end{figure}
We introduce the well-known averaging operator (also known as Oswald interpolator) 
$\Pi^{\text{av}}_{h} \, : \, \Pol_p(\D_h) \rightarrow \Pol_p(\D_h) \cap H_0^1(\Omega) = U_h$ 
such that, for any interpolation node $V \in \overline{\Omega}$,
\begin{equation}
\Pi^{\text{av}}_{h}(v_h)(V) := \dfrac{1}{\card(\D_V)} \sum_{K \in \D_V} v_h|_K(V),
\end{equation}
with $\D_V \subset \D_h$ denoting the union of the elements $K$ sharing $V$ as a common node (see the left panel of Figure~\ref{fig:Dv}). In~\cite{ KarPA:03,BurEr:07,ErnGu:16_quasi}, it is shown that, for all $K\in\Par_h$,
\begin{eqnarray}\label{eq:oswald_ineq}
\left\| v_h - \Pi^{\text{av}}_{h}(v_h) \right\|_{K}^2 \lesssim \sum_{e \in \Sk_K\cap \Sk_h^0} h_k \left\| \llbracket v_h \rrbracket \right\|_e^2,
\end{eqnarray} 
with $\Sk_K$ denoting the mesh faces/edges having a non-empty intersection with $\partial K$ (see the right panel of Figure~\ref{fig:SK}). 

Let $v \in H^{s}(\Omega) \cap V$, $s> \frac12$, and denote by $v_h \in V_h$ a function such that
\begin{equation}
\|v-v_h\|_{\textrm{up},\#} := \inf_{z_h \in V_h} \|v-z_h\|_{\textrm{up},\#}.  
\end{equation}
Let $v_h^\ast:= \Pi^{\text{av}}_{h}(v_h)$. Since $v_h^\ast \in U_h$, we have
\begin{eqnarray}
\inf_{z_h \in U_h} \|v-z_h\|_{\textrm{up},\#} \leq  \|v-v_h^\ast\|_{\textrm{up},\#}
\leq \|v-v_h\|_{\textrm{up},\#} + \|v_h-v_h^\ast\|_{\textrm{up},\#}.
\end{eqnarray}
Therefore, we only need to show that $\|v_h-v_h^\ast\|_{\textrm{up},\#} \lesssim \|v-v_h\|_{\textrm{up},\#}$ to prove the claim. 
Using inverse and discrete trace inequalities, we infer that
\begin{equation}
\|v_h-v_h^\ast\|_{\textrm{up},\#}^2 \lesssim \sum_{K\in \Par_h} h_K^{-1}  \|v_h-v_h^\ast\|_K^2.
\end{equation}
Using~\eqref{eq:oswald_ineq} leads to
\begin{equation}
\|v_h-v_h^\ast\|_{\textrm{up},\#}^2 \lesssim \sum_{e\in \Sk_h^0}  \|\llbracket v_h \rrbracket \|_e^2.
\end{equation}
Since $\llbracket v\rrbracket =0$, for all $e\in \Sk_h^0$ (recall that $s>\frac12$), we have $\llbracket v_h \rrbracket
= \llbracket v_h-v \rrbracket$. We can now use the triangle inequality to decompose the jump into the two parts coming from the two cells sharing $e$, and re-arranging the terms leads to
\begin{equation}
\|v_h-v_h^\ast\|_{\textrm{up},\#}^2 \lesssim \sum_{K\in\Par_h}  \|v-v_h\|_{\partial K}^2 \le \|v-v_h\|_{\textrm{up},\#}^2,
\end{equation}
thereby completing the proof.

%$$\dfrac{1}{C_{bnd}} \| \varepsilon_h \|_{V_h} \leq \sup_{0 \neq v_h \in V_h} \dfrac{\|v_h\|_{V_h, \#}}{\|v_h\|_{V_h}} \|\theta_h - u_h \|_{V_h} $$
\section*{Acknowledgements}
This publication was made possible in part by the CSIRO Professorial Chair in Computational Geoscience at Curtin University and the Deep Earth Imaging Enterprise Future Science Platforms of the Commonwealth Scientific Industrial Research Organisation, CSIRO, of Australia. The European Union's Horizon 2020 Research and Innovation Program of the Marie Sk\l{}odowska-Curie grant agreement No. 777778 provided additional support. At Curtin University, The Institute for Geoscience Research (TIGeR) and by the Curtin Institute for Computation, kindly provide continuing support. The work by Ignacio Muga was done in the framework of Chilean FONDECYT research project $\#1160774$. 
%\section*{References}
%
\bibliography{mybibfile}

\end{document}